\documentclass[10pt]{amsart}

\usepackage{comment}
\usepackage[official]{eurosym}
\usepackage[greek,english]{babel}
\usepackage{bbm}
\usepackage[utf8]{inputenc}
\usepackage[T1]{fontenc}
\usepackage{pdfpages}
\usepackage{tensor}
\usepackage{enumerate}
\usepackage{relsize}
\usepackage[normalem]{ulem}

\usepackage[stable]{footmisc}
\usepackage{mathtools}
\usepackage[mathscr]{eucal}
\usepackage[a4paper, twoside=false, vmargin={2cm,3cm}, includehead]{geometry}
\usepackage{pgfgantt}
\usepackage{graphicx}
\usepackage{xcolor}
\ganttset{group/.append style={orange},
	milestone/.append style={red},
	progress label node anchor/.append style={text=red}}
\theoremstyle{plain}
\newtheorem{theorem}{Theorem}[section]
\newtheorem{lemma}[theorem]{Lemma}
\newtheorem{proposition}[theorem]{Proposition}
\newtheorem{question}{Question}
\newtheorem*{question*}{Question}
\newtheorem{conjecture}{Conjecture}
\newtheorem*{conjecture*}{Conjecture}
\newtheorem{remark}{Remark}
\newtheorem{corollary}[theorem]{Corollary}

\newtheorem{definition}[theorem]{Definition}
\newtheorem{claim}{Claim}
\newtheorem*{Graal1*}{Special Case of Problem~\ref{Graal1}}
\newtheorem*{Graal2*}{Special Case of Problem~\ref{Graal2}}

\newtheorem*{theorem*}{Theorem}

\newtheorem{innercustomthm}{Theorem}
\newenvironment{customthm}[1]
{\renewcommand\theinnercustomthm{#1}\innercustomthm}
{\endinnercustomthm}
\usepackage{amssymb}

\theoremstyle{remark}
\newtheorem*{remark*}{{\bf Remark}}
\newtheorem*{remarks*}{{\bf Remarks}}
\newtheorem*{comment*}{{\bf Comment}}

\usepackage{cancel}

\usepackage{amsmath}

\newcommand{\Oh}{{\rm O}}
\newcommand{\oh}{{\rm o}}

\newcommand{\wt}{\widetilde}

\newcommand{\norm}[1]{\left\Vert #1\right\Vert}
\newcommand{\N}{\mathbb{N}}
\newcommand{\Q}{\mathbb{Q}}
\newcommand{\R}{\mathbb{R}}
\newcommand{\T}{\mathbb{T}}
\newcommand{\Z}{\mathbb{Z}}
\newcommand{\A}{\mathcal{A}}
\newcommand{\B}{\mathcal{B}}
\newcommand{\C}{\mathbb{C}}

\newcommand{\D}{\mathbb{D}}

\newcommand{\M}{\mathcal{M}}

\newcommand{\X}{\mathcal{X}}

\newcommand{\m}{\mu}

\newcommand{\Mfg}{\M^\textup{fg}}
\newcommand{\Mfgs}{\Mfg_\textup{s}}
\newcommand{\Mfgp}{\Mfg_\textup{p}}

\newcommand{\1}{\mathbbm{1}}

\renewcommand{\Re}{\mathrm{Re}}

\renewcommand{\P}{\mathbb{P}}

\newcommand{\e}{\varepsilon}

\newcommand{\logE}{\mathbb{E}^{\log}}
\newcommand{\cesE}{\mathbb{E}}

\def \colon{{:}\;}
\usepackage{fancyhdr}
\usepackage{theoremref}

\usepackage{hyperref}

\theoremstyle{definition}
\newtheorem*{counterexample*}{Counterexample}

\usepackage[capitalise]{cleveref}

\title{On multiplicative recurrence along linear patterns}
\date{}
\author{DIMITRIOS CHARAMARAS, ANDREAS MOUNTAKIS AND KONSTANTINOS TSINAS}


\begin{document}

\vspace*{-0.015cm}
\begin{abstract} 
In \cite{Moreira&friends}, Donoso, Le, Moreira and Sun
study sets of recurrence for actions of the multiplicative semigroup
$(\N, \times)$ and provide some sufficient conditions for sets of the form 
$S=\{(an+b)/(cn+d) \colon  n \in \N \} $
to be sets of recurrence for such actions. A necessary condition for $S$ to be a set of multiplicative recurrence is that
for every completely multiplicative function $f$ taking values on the unit 
circle, we have that $\liminf_{n \to \infty} |f(an+b)-f(cn+d)|=0.$ In this article, we fully characterize the integer quadruples $(a,b,c,d)$ which satisfy the latter property. Our result generalizes 
a result of Klurman and Mangerel from \cite{Klu-Man} concerning the pair $(n,n+1)$, as well as some results from \cite{Moreira&friends}. In addition, 
we prove that, under the same conditions on $(a,b,c,d)$, the set $S$ is a set of 
recurrence for finitely generated actions of $(\N, \times)$.
\end{abstract}

\maketitle
\tableofcontents

\section{Introduction}
\label{section_intro}

Since Furstenberg's multiple recurrence theorem \cite{Furstenberg-original}, which yielded an ergodic proof of Szemer\'edi's theorem \cite{Szemeredi_original}, significant progress has been made in understanding recurrence phenomena in dynamical systems generated by a single transformation, that is, for actions of $(\Z,+)$ (see for instance \cite{Bergelson-Leibman-polynomial-VDW,Furstenberg-book,K-MF_vdC, sarkozy}). Several results have also been established for actions of $\Z^d$ \cite{Bergelson_Zm, Bergelson-Lesinge_Zd, FK_commuting}, and other amenable groups \cite{Bergelson_McCutcheon_1}. We refer the reader to \cite{bergelson_survey} for further relevant results and references.

Recurrence involving multiplicative actions of the semigroup $(\N,\times)$ was first studied in \cite{Bergelson_multiplicative}, but the results in this direction are sparse and several interesting questions remain unanswered. Recently, the connection of this topic to major open conjectures concerning the partition regularity of quadratic equations \cite{Fra-Host-structure-multiplicative, fra-klu-mor-2, Fra-Klu-Mor} has sparked renewed interest.

In \cite{Moreira&friends}, Donoso, Le, Moreira and Sun
study sets of recurrence,
in both measurable and topological settings, for actions of 
$(\N, \times)$ and $(\Q_{>0}, \times)$. In particular, they show \cite[Theorem 1.6]{Moreira&friends} that sets of the form 
$S=\{(an+b)/(cn+d) \colon n\in\N\}$,
are sets of topological multiplicative recurrence if $a=c$, and $a\mid b$ or $a\mid d$. However, there are sets, such as
$\{(6n+3)/(6n+2) \colon  n \in \N \}$, for which they do not provide a conclusive answer, and no counterexample is currently known.
This led them to ask the following:

\begin{question}{\cite[Question 7.2]{Moreira&friends}}\label{question DLMS}
     For $a\in \N$, $b,d \in \Z$, is it true that $\{(an+b)/(an+d) \colon
     n\in \N\}$ is a set of topological multiplicative recurrence if and 
     only if $a \mid b$ or $a\mid d$?
\end{question}

Recall that a function $f:\N \to \C$ is called multiplicative if 
$f(mn)=f(m)f(n)$ for all $m,n \in \N$ with $(m,n)=1$, and it is called 
completely multiplicative if the previous identity
holds for all $m,n \in \N$.
It can be shown (using a variant of Lemma \ref{L: recurrence implies Bohr recurrence} below for topological recurrence) that for $a,c \in \N$ and $b,d \in \Z$, if the set $S$ defined above is a set of topological multiplicative recurrence, then for any 
completely multiplicative function $f\colon \N \to \mathbb{S}^1$, 
\begin{equation}\label{eq res 1}
    \liminf_{n\to \infty} |f(an+b)-f(cn+d)|=0.
\end{equation}

In an attempt to shed some light on Question \ref{question DLMS},
we give necessary and sufficient conditions on $a,b,c,d$ so that 
\eqref{eq res 1} holds for every completely multiplicative function 
$f\colon \N \to \mathbb{S}^1$. Note that factoring out the greatest common 
divisor $u=(a,b,c,d)$ in 
\eqref{eq res 1}, we may assume without loss of generality that 
$(a,b,c,d)=1$. Intuitively, if we account for some obvious divisibility obstructions,
the multiplicative structure of the numbers $an+b$ and 
$cn+d$ should be somewhat independent of one another. Thus, for a 
completely multiplicative function $f\colon \N \to \mathbb{S}^1$, one should expect \eqref{eq res 1} to hold unless $f$ interacts with the arithmetic (and archimedean) properties of $a,b,c,d$ in a specific way. In fact, we will see that the main counterexamples to \eqref{eq res 1} come from a class of multiplicative functions which are essentially of the form $\chi(n)n^{it}$ for some Dirichlet character $\chi$ and a real number $t$. 

\subsection{Results on multiplicative functions}
Let us denote by $\M$ the set
\begin{equation}\label{E: definition of space of multiplicative functions}
    \M=\{f\colon \N\to \mathbb{S}^1\colon f\  \text{ is completely multiplicative}\},
\end{equation}
where $\mathbb{S}^1$ denotes the set of complex numbers with modulus $1$.
The first theorem we prove in this paper is the following:

\begin{theorem}\label{T: optimal conditions for Bohr recurrence}
Let $a,c \in \N$, $b,d \in \Z$ with $(a,b,c,d)=1$. Then, the following are equivalent:
\begin{enumerate}
\item For every $f\in \M$ and $\e>0$, the set 
\begin{equation}\label{set of ret times}
    \A(f,\e):=\{n\in \N \colon |f(an+b)-f(cn+d)| < \e\}
\end{equation}
has positive upper logarithmic density.\footnote{See \cref{section_background} for definition.}
\item For every $f\in \M$,
\begin{equation*}
    \liminf_{n\to \infty} | f(an+b) - f(cn+d)| =0.
\end{equation*}
\item  $a=c$, and $b=d$ or $a\mid bd$.
\end{enumerate}
\end{theorem}

\begin{remark}\label{rem lower density}
We note here that if $f$ is pretentious or finitely generated (see Definitions \ref{D: definition of pretentious and aperiodic} and \ref{D: finitely generated}), then the set 
$\A(f, \e)$ above has positive lower logarithmic density. We prove this in \cref{section result f-f}.
\end{remark}

For instance, Theorem \ref{T: optimal conditions for Bohr recurrence} implies that for all $f\in \M$, 
$$\liminf_{n\to \infty} |f(6n+3)-f(6n+2)|=0.$$
This provides positive evidence that
$\{(6n+3)/(6n+2) \colon n \in \N \}$ is a set of multiplicative recurrence.

In \cite[Theorem 1.1]{Klu-Man}, Klurman and Mangerel prove
that for any $f\in \M$, 
$$\liminf_{n\to \infty} |f(n+1) - f(n)|=0.$$
\cref{T: optimal conditions for Bohr recurrence} generalizes this result,
as in place 
of $(n+1, n)$ we can put any pair $(an+b, an+d)$ with 
$au \mid bd$, where $u=(a,b,d)$, and our conditions on $a,b,d$ are
optimal.  Furthermore, Theorem \ref{T: optimal conditions for Bohr recurrence} generalizes
\cite[Corollary 1.7]{Moreira&friends}.

We also prove the following result concerning pairs of multiplicative functions.
\begin{theorem}\label{a result on pairs}
For any $a\in \N$, $f,g \in \M$ and $\e>0$, the set 
\begin{equation*}
    \A(f,g,\e):=\{n\in \N \colon |f(an+1)-g(an)| < \e \}
\end{equation*}
has positive upper logarithmic density. In particular, 
this implies that for any $a\in \N$ and $f,g \in \M$
\begin{equation*}
\liminf_{n\to \infty} |f(an+1)-g(an) |=0.
\end{equation*}
\end{theorem}

Theorem \ref{a result on pairs} also generalizes \cite[Theorem 1.1]{Klu-Man}, as it holds for pairs $f,g$ in $\M$, and
$(n+1, n)$ can be replaced by $(an+1, an)$ for any $a\in \N$.

\begin{remark}
In Section \ref{proofs pairs} we show the existence 
of a pair $f,g\in \M$ that satisfy the property $\liminf_{n\to \infty} 
|f(n+2)-g(n)|>0$. Therefore,  even for $a=1$ in the context of Theorem \ref{a result on pairs}, one cannot 
replace $(n+1, n)$ by $(n+k, n)$ for any $k\in \Z$. Hence, it seems 
difficult to prove an analog of Theorem \ref{T: optimal conditions for Bohr recurrence} for pairs $f,g\in \M$, unless one restricts to a smaller class of multiplicative functions.
\end{remark}

\subsection{Multiplicative recurrence for finitely generated systems}

Motivated by \cite[Theorem 1.6 and Question 7.2]{Moreira&friends}, where the authors 
investigate under what conditions on $a,b,c,d$ the set $S = \{(an+b)/(cn+d)\colon 
n\in\N\}$ is a set of topological multiplicative recurrence, we are interested in 
finding sufficient conditions so that $S$ is a set of measurable multiplicative 
recurrence (see Definition~\ref{D: definition of measurable multiplicative recurrence}) (a property stronger than the topological one). In view of 
\cref{T: optimal conditions for Bohr recurrence}, it is reasonable to ask whether the 
condition $a=c$, and $b=d$ or $a\mid bd$ is sufficient to show that $S$ is a set of 
measurable (or at least topological) multiplicative recurrence. Note that this 
condition is more general than the one in \cite[Theorem 1.6]{Moreira&friends}. We are 
able to prove that, under this condition, $S$ is indeed a set of measurable 
multiplicative recurrence for an important class of multiplicative systems. 

\begin{definition}
    Let $(X,\X,\m)$ be a Borel probability space. We say that a sequence of invertible transformations $T_n\colon X\to X$, $n\in \N$, induces a \emph{multiplicative measure-preserving action} on $(X,\X,\m)$ if the following hold:
    \begin{enumerate}
        \item Each map $T_n$ is \emph{measure-preserving}: for any set $A\in \X$, we have $T_n^{-1}A\in \X$ and $\m(T_n^{-1}A)=\m( A)$. 
        \item The $T_n$ act multiplicatively on $X$: for any $n,m\in \N$, we have the relation $T_n\circ T_m(x)=T_{nm}(x)$ for almost all $x\in X$.
    \end{enumerate}
    We call $\left(X,\X,\m,(T_n)_{n\in \N}\right)$ a \emph{multiplicative system}.
\end{definition}

It is evident that a multiplicative system is uniquely determined by the transformations $T_p$, where $p$ is prime.
A multiplicative system is called {\em finitely generated}
if $\{T_p\colon p\in\P\}$ is finite. 

Finitely generated multiplicative systems are important due to their connections with finitely generated multiplicative functions (see Definition~\ref{D: finitely generated}), such as the Liouville function $\lambda$, which are central to understanding multiplicative properties of the integers. This subclass of multiplicative systems has been studied in recent works of Bergelson and Richter \cite{Bergelson-Richter} and the first author \cite{Charamaras-multiplicative}, investigating various aspects of the interplay between ergodic theory and multiplicative number theory. In both these works the finitely generated property is essentially exploited, and all the relevant results there are not known to hold for systems that do not share this property.

\begin{definition}\label{D: definition of measurable multiplicative recurrence}
    Let $(p_n)_{n\in\N}$ and $(q_n)_{n\in\N}$ be two sequences of positive integers. Then the set $S=\{p_n/q_n\colon n\in \N\}$ is called a \emph{set 
    of (measurable) multiplicative recurrence} for the system $\left(X,\X,\m,(T_n)_{n\in 
    \N}\right)$ if and only if for every set $A\in \X$ with positive 
    measure, there exist infinitely many\footnote{This is necessary in some cases to ensure that the result is non-trivial, as a set of rationals that contains 1 satisfies the relevant positivity property for some $n$.} $n\in \N$, such that 
    $$\m(T^{-1}_{p_n}A\cap T^{-1}_{q_n} A)>0.$$
    It is called a \emph{set of multiplicative recurrence} if it is a set 
    of multiplicative recurrence for all multiplicative systems.
\end{definition}

Throughout this paper, we omit the word measurable\footnote{In \cite{Moreira&friends}, the authors study Question \ref{question DLMS} under the notion of \emph{topological multiplicative recurrence}, which is weaker than the one in Definition \ref{D: definition of measurable multiplicative recurrence}.}
, and we will only refer to sets 
satisfying Definition \ref{D: definition of measurable multiplicative recurrence}
as sets of multiplicative recurrence. 

Our next theorem shows that under the same assumptions as in Theorem \ref{T: optimal conditions for Bohr recurrence}, we have multiplicative recurrence for all finitely generated systems.

\begin{theorem}\label{T: multiplicative recurrence for finitely generated systems}
    Let $a\in \N$, $b,d \in \Z$ with $(a,b,d)=1$ 
    satisfy $b=d$ or $a\mid bd$. Then, the set 
    $$S=\left \{\frac{an+b}{an+d}\colon n\in \N\right\}$$ 
    is a set of multiplicative recurrence for all finitely generated 
    multiplicative systems. 
\end{theorem}
We note that the assumption $(a,b,d)=1$ in Theorem \ref{T: multiplicative recurrence for finitely generated systems} is not restrictive; for general 
$a,b,d$, we can always factor out $u=(a,b,d)$, and we have that
$S$ is equal to $\{(a'n+b')/(a'n+d') \colon n\in \N\}$,
where $a'=a/u, b'=b/u, d'=d/u$ and $(a',b',d')=1$.

\begin{remark}
We remark that the argument we use for Theorem \ref{T: multiplicative recurrence for finitely generated systems} also implies that
the set $\{2n/(n+1) \colon n\in \N \}$
is a set of recurrence for finitely generated multiplicative systems
(see Remark \ref{new_remark}). From 
Theorem \ref{T: optimal conditions for Bohr recurrence} we know that the 
previous set is not a set of 
recurrence for all multiplicative systems. This establishes that 
recurrence for all multiplicative systems is strictly stronger than 
recurrence for finitely generated multiplicative systems. 
\end{remark}

\subsection{Open questions}
From Theorem \ref{T: optimal conditions for Bohr recurrence} and 
Lemma \ref{L: recurrence implies Bohr recurrence}, one sees that 
for $a,c \in \N$, $b,d\in \Z$ with $(a,b,c,d)=1$,
if the set $S=\{(an+b)/(cn+d)\colon n\in \N \}$ is a set of multiplicative 
recurrence, then $a=c$, and $b=d$ or $a\mid bd$. We conjecture 
that those conditions on $a,b,c,d$ are also sufficient for the set $S$ to 
be a set of multiplicative recurrence.

\begin{conjecture}\label{conj_1}
    Let $a,c\in \N$, $b,d \in \Z$ with $(a,b,c,d)=1$. The set $S=\{(an+b)/(cn+d) \colon n\in \N\} $
    is a set of measurable multiplicative recurrence if and only if $a=c$, and $b=d$ or $a\mid bd$.
\end{conjecture}

We remark that if this conjecture is true, then this would yield a negative answer to \cite[Question 7.2]{Moreira&friends}.
Note that in the previous, if $b=d$, then $S=\{1\}$, which is trivially a 
set of multiplicative recurrence, so the interesting case is $b\neq d$ and 
$a\mid bd$.

Given a finite collection $f_1, \ldots, f_{\ell} \in \M$, 
as in Lemma \ref{L: recurrence implies Bohr recurrence} one can consider 
the multiplicative action on $(\mathbb{S}^1)^{\ell}$ defined by 
$T_n(z_1, \ldots, z_{\ell})=(f_1(n)z_1, \ldots, f_{\ell}(n) z_{\ell}).$
Then a special case of Conjecture \ref{conj_1} is the following:
\begin{conjecture}\label{conj_2}
Let $a\in \N$, $b,d \in \Z$ with $(a,b,d)=1$. If $b=d$ or 
$a\mid bd$, then for any finite collection $f_1, \ldots, f_{\ell} \in \M$, 
$$\liminf_{n\to \infty} \sum_{j=1}^{\ell} 
|f_j(an+b)-f_j(an+d)| =0.$$
\end{conjecture}

\subsection{Proof strategy}

Our proof strategy for \cref{T: optimal conditions for Bohr recurrence} is substantially different from the one used in 
\cite{Klu-Man}. 
In order to establish \cref{T: optimal conditions for Bohr recurrence}, we employ the $Q$-trick, which appeared in \cite{Fra-Klu-Mor}.
The main idea is that when one wants to handle the additive 
correlations of two multiplicative functions $f,g$, it is often useful to
average over the progression $Qn$, where $Q$ belongs to a multiplicative
F{\o}lner sequence. Then in a lot of cases, one can factor out from the 
averages an $f(Q)$ term, and then performing a multiplicative averaging 
over $Q$, it follows that  
the correlations vanish unless $f$ is the trivial multiplicative function. 

However, in our case, averaging over the progression $Qn$ is not 
sufficient. In order to make the $Q$-trick work, we need to average over
arithmetic progressions $Q^2 n +r_Q$ for some appropriately chosen 
$r_Q \in \{0,1, \ldots, Q^2 -1\}$.
We choose the $r_Q$'s using the Chinese Remainder Theorem so that some helpful (but somewhat complicated)
congruence relations are satisfied. The condition that $a\mid bd$ is 
crucial for those congruences to have solution.

In order to establish the result, we then have to handle  
correlations of the form 
\begin{equation}\label{proof strategy eq 1}
    \frac{1}{|\Phi_K|}\sum_{Q\in \Phi_K} \frac{1}{\log X_m} \sum_{n\leq X_m} 
\frac{f^{\ell}(a(Q^2n+rQ)+b)
\overline{f^{\ell}(a(Q^2n+r_Q)+d)}}{n} 
\end{equation}
for a finite collection of powers of $f^{\ell}$ of $f$, where the
outer averaging is with respect to 
a multiplicative F{\o}lner sequence in $\N$. 
In the previous, $(X_m)_{m\in \N}$ is a sequence that we choose so that 
the following holds: for every $\ell$ in a finite collection of natural numbers, 
if $f^{\ell}$ is aperiodic, then along $(X_m)_{m\in \N}$, $f^{\ell}$
behaves like a strongly
aperiodic function. This means that, for all Dirichlet characters $\chi$ (up to some modulus), we have 
$$\inf_{|t|\leq X_m} \D(f^{\ell}(n), \chi(n) n^{it};1,X_m) \to + \infty$$ as $m\to \infty$. For these powers, we apply Tao's theorem \cite{Tao-Chowla} on 
two-point correlations of multiplicative functions, while for the powers for which $f^{\ell}$ is pretentious, we use
some concentration estimates to simplify the averages.
The proof of Theorem \ref{T: optimal conditions for Bohr recurrence} is the context 
of Sections \ref{section_optimal-conditions} and \ref{section result f-f}. The estimates for binary correlations are all contained in Proposition \ref{P: binary correlations proposition}, which we prove more generally for two functions $f,g\in \mathcal{M}$.

Let us now turn our attention to the proof of Theorem \ref{T: multiplicative recurrence for finitely generated systems} and its 
connection to Conjectures \ref{conj_1} and \ref{conj_2}.
If one wants to establish Conjecture \ref{conj_1} using our proof 
strategy, they have to simultaneously handle correlations of the form 
\begin{equation}\label{eq fin gen and general}
\frac{1}{|\Phi_K|}\sum_{Q\in \Phi_K} 
\frac{1}{\log X_m} \sum_{n\leq X_m} \frac{f(a(Q^2n+rQ)+b)
\overline{f(a(Q^2n+r_Q)+d)}}{n} 
\end{equation}
for all $f\in \M$ along some appropriately chosen arithmetic
progressions $Q^2n+r_Q$ and some sequence of natural numbers $X_m \to \infty$.
In the previous, we would like $(X_m)_{m\in \N}$ to be a sequence which satisfies the 
following: for all $f\in \M$, if $f$
is aperiodic, then along $(X_m)_{m\in \N}$, $f$ behaves like a strongly
aperiodic function. 

However, even for two functions $f,g\in \M$ which are aperiodic but 
not strongly aperiodic\footnote{Functions of this form were initially constructed in \cite{MRT_averaged} as counterexamples to the Elliott conjecture.}, it is
unclear how to choose a sequence $X_m \to \infty$ that works for both simultaneously. This is the main obstruction one would have to overcome to 
prove Conjectures \ref{conj_1} and \ref{conj_2}. 

On the other hand, for the proof of Theorem \ref{T: multiplicative recurrence for finitely generated systems}, we only need to handle 
\eqref{eq fin gen and general} for the collection of 
finitely generated functions $f\in \M$. In this case, we show that all these functions are either pretentious or strongly aperiodic, which also implies that we can always 
take $X_m=m$. We handle each of these cases using Proposition \ref{P: binary correlations proposition}.

\subsection{Notational conventions} 
We let $\N=\{1, 2, \ldots\}$, $\N_0=\{0, 1, 2, \ldots\}$, 
$\mathbb{S}^1$ be the unit circle in $\C$, $\mathbb{U}$ be the 
closed complex unit disk and $\T=\R/\Z$ be the torus.
For $x\in \R$, we let $e(x):=e^{2\pi i x}$ and
$\norm{x}_{\T}$ be the distance of $x$ from its nearest integer.

We use $\P$ for the set of prime numbers, and throughout 
we use $p$ to denote prime numbers. In addition, we denote the $j$-th
prime number by $p_j$. For $p \in \P$, $\ell \in \N_0$ and 
$a\in \Z$, we say that $p^{\ell}$ fully divides $a$, and we denote this by 
$p^{\ell} \parallel a$, if $p^{\ell} \mid a $ and $p^{\ell+1} \nmid a$. 
For $a,b\in \Z$, we write $a\mid b^{\infty}$ if there is $\ell\in \N$ so that 
$a\mid b^{\ell}$.
We always write $(a_1,\ldots, a_k)$ to be the greatest common divisor of 
$a_1,\ldots, a_k$. The inverse of an invertible element $a$ in $\Z_{m}^{\times}$ will be denoted by $a^{-1}$.

For two functions $f,g$, if there is a constant 
$C>0$ such that $|f(X)| \leq C |g(X)|$ for all $X$ sufficiently large, then we write
$f(X)=\Oh(g(X))$ or $f(X)\ll g(X)$. Furthermore, when 
the constant $C$ depends on some parameter $K$, we write
$f(X)=\Oh_{K} (g(X)).$ 
Finally, we write $\oh_{K; X\to \infty}(1)$ to 
denote an error term that goes 
to $0$ as $X$ grows to infinity but depends on $K$, or in other words, 
an error term of the form $\e(X,K)$ with 
$\lim_{X\to \infty}\e(X,K) = 0$ for $K$ fixed.

\vspace{2mm} 
\noindent
\textbf{Acknowledgements.} The authors are thankful to Nikos 
Frantzikinakis for useful discussions and suggestions on previous 
versions of the paper. The authors are also thankful to Joel Moreira and 
Florian K. Richter for several helpful comments.
Finally, the authors would like to thank the 
anonymous referee for a number of useful comments and suggestions.

The first and the third authors were supported by the Swiss National
Science Foundation grant TMSGI2-211214.
The second author was supported by the Research Grant 
ELIDEK HFRI-NextGenerationEU-15689.

\section{Background material on multiplicative functions}
\label{section_background}

\subsection{Averages and density}

Given a non-empty finite set
$A\subset \N$ and a function $f\colon A \to \C$, we denote the Ces\`aro average of $f$ by
$$\mathlarger{\cesE}_{n\in A} f(a):= \frac{1}{|A|} \sum_{n\in A} f(n),$$ 
and the logarithmic average by
$$\mathlarger{\logE}_{n\in A} f(n):=\frac{1}{\sum_{n\in A} \frac{1}{n}} \sum_{n\in A} \frac{f(n)}{n}.$$ 
In particular, since $\sum_{n\leq X} \frac{1}{n}= (1+ 
o_{X\to \infty} (1))\log X$, we define the logarithmic average along the interval $[1,X]$ by the simpler expression
$$\mathlarger{\logE}_{n\leq X} f(n):=
\frac{1}{\log X} \sum_{n\leq X} \frac{f(n)}{n}.$$

Given a subset $E\subseteq \N$, we define its \emph{upper} and 
\emph{lower logarithmic densities} respectively by 
\begin{align*}
    \overline{d}_{\log}(E)&:=\limsup_{X\to\infty} \  \mathlarger{\logE}_{n\leq X}\1_{A}(n),\\
     \underline{d}_{\log}(E)&:=\liminf_{X\to\infty} \  \mathlarger{\logE}_{n\leq X} \1_{A}(n),
\end{align*}
and when they coincide, the limit exists, and we
say that $E$ has logarithmic density equal to this limit.

In our proofs, we will need to average over a suitably chosen progression
with a highly divisible difference. For each $K\in \N$ we define
\begin{equation}\label{E: definition of multiplicative Folner}
    \Phi_K=\left\{p_1^{\theta_1}\cdots p_K^{\theta_K}\colon 2^K<
    \theta_1, \ldots, \theta_K\leq 2^{K+1}\right\},
\end{equation}
where $p_1,\ldots, p_K$ are the first $K$ primes. A similar
sequence is used in \cite{Fra-Klu-Mor}, with $K<\theta_i\leq 2K$ in place of $2^K<\theta_i\leq 2^{K+1}$. We pick our inequalities so that
no number belongs to two distinct sets $\Phi_{K_1},\Phi_{K_2}$, but 
otherwise the choice of inequalities is unimportant.

The sequence $\Phi_K$ satisfies 
\begin{equation}\label{E: limit of multiplicative Folner is 1}
    \frac{|\Phi_K\cap\ \Phi_{K}/p|}{|\Phi_K|}\to 1
\end{equation} as $K\to \infty$,
where for $A\subset \N$ we define 
$$A/p := \{n\in \N\colon pn\in A\}.$$
Sequences that satisfy \eqref{E: limit of multiplicative Folner is 1} are called \emph{multiplicative F{\o}lner} sequences.

\subsection{Distances of multiplicative functions and aperiodicity.}
Let $f,g$ be multiplicative functions whose modulus is bounded by $1$. Following the terminology of Granville and Soundararajan \cite{Granville-Sound-book}, we define the \emph{pretentious distance} of two multiplicative functions between two real numbers $A,B$ by the formula\begin{equation}\label{E: definition of pretentious distance}
    \D(f,g;A,B)
    :=\left(\sum_{A< p\leq B}\frac{1-\Re(f(p)\overline{g(p)})}{p}\right)^{1/2}.
\end{equation}
We define $\D(f,g):=\lim\limits_{X\to \infty} \D(f,g;1,X)$ and say that $f$ \emph{pretends} to be $g$, which we denote by $f\sim g$,
if  $\D(f,g)<+\infty$. It can be shown that the distance function satisfies a triangle inequality. More specifically, for any multiplicative functions $f_1,f_2,g_1,g_2$, we have \begin{equation*}
    \D(f_1,g_1;1,X)+\D(f_2,g_2;1,X)\geq \D(f_1f_2,g_1g_2;1,X),
\end{equation*}which readily implies that\begin{equation*}
    \D(f,g;1,X)+\D(g,h;1,X)\geq \D(f,h;1,X)
\end{equation*}for any $1$-bounded multiplicative functions $f,g,h$. We will use the triangle inequality freely throughout the proofs without further references.

For $z,w \in \mathbb{S}^1$, we have $|z-w|^2=2(1-\Re(z\overline{w}))$, so when $f$ and 
$g$ take values in $\mathbb{S}^1$, we have 
\begin{equation*}
     \D(f,g;A,B)= 
     \left(\frac{1}{2}\sum_{A< p\leq B}\frac{|f(p)-g(p)|^2}{p}\right)^{1/2}.
\end{equation*}

\begin{definition}\label{D: definition of pretentious and aperiodic}
A $1$-bounded multiplicative function $f$ is called \emph{pretentious} if there exists a Dirichlet character $\chi$ and a real number $t$, such that $\D(f(n),\chi(n)n^{it})<+\infty$. Otherwise, $f$ is called \emph{non-pretentious} or \emph{aperiodic}.

A multiplicative function $f$ is called {\em strongly aperiodic} if for any $B>0$,
\begin{equation}\label{unif dist}
    \lim_{X\to \infty}\inf_{|t|\leq BX, q\leq B}\D(f(n),\chi(n)n^{it};1,X) = +\infty,
\end{equation}
where the infimum is taken over all Dirichlet characters $\chi$ of conductor
$q$ with $q\leq B$.
\end{definition}

Throughout, whenever we have an expression as in \eqref{unif dist},
it is implicit that infimum is taken over all Dirichlet characters $\chi$ 
of conductor $q$ with $q\leq B$.

The term aperiodic is related to the following  consequence of Hal\'asz's 
theorem \cite{Halasz}: for a 1- bounded multiplicative function $f\colon \N \to 
\mathbb{U}$, if $\D(f(n),\chi(n) n^{it})=+\infty$ for all Dirichlet 
characters $\chi$ and $t\in \R$, then 
\begin{equation*}
 \lim_{X\to \infty}   \frac{1}{X}\sum_{n\leq X} f(an+b)=0
\end{equation*}for all $a,b\in \N$.
If we consider logarithmic averages instead, then one only 
needs to consider the distances in the case where $t=0$, since otherwise 
the associated average vanishes.
More specifically, we recall an elementary bound for logarithmic means of bounded multiplicative functions.
\begin{lemma}
    Let $f \colon \N \to \mathbb{U}$ be a multiplicative function. 
    Then, we have \begin{equation}\label{E: Halasz for log averages}
        \left|\mathlarger{\logE}_{n\leq X} f(n)\right|\ll \exp\Bigg(-\frac{1}{2}\sum_{p\leq X} \frac{1-\Re f(p)}{p}\Bigg).
    \end{equation}
\end{lemma}
This follows from \cite[Proposition 1.2.6]{Granville-Sound-book}. 
Roughly, it asserts that if a function does not pretend to be 1, then its 
logarithmic means vanish.
We have the following corollary, which is an extension of the previous 
lemma along arithmetic progressions.
\begin{corollary}\label{C: mean value does not depend on progression}
    Let $Q$ be a positive integer and  $f\colon \N \to \mathbb{U}$ be a 
    multiplicative function such that $\mathbb{D}(f,\chi)=+\infty$ for all 
    Dirichlet characters of modulus $Q$. Then, we have 
    \begin{equation}\label{E: independence of mean values from progression}
        \max_{1\leq a\leq Q, (a,Q)=1}  \left| \mathlarger{\logE}_{n\leq X} 
        f(Qn+a)\right|=\oh_{f,Q; X\to \infty}(1),
    \end{equation}
where the error term may depend on $f, Q$.
\end{corollary}
\begin{proof}
     We use the expansion $${\bf 1}_{a\pmod{Q}}(n)=\frac{1}{\phi(Q)}\sum_{\chi\pmod{Q}} \chi(n)\overline{\chi(a)}, $$
     where $\phi$ denotes Euler's totient function, 
     and apply the previous lemma to the functions $f\overline{\chi}$, which are not $1$-pretentious.
\end{proof}

\begin{definition}\label{D: finitely generated}
    A multiplicative function $f\colon \N \to \C$ is called finitely generated if the set 
$\{f(p)\colon p \in \P\}$ is finite.
\end{definition}

The upside of working with finitely generated functions hinges on the next proposition.

\begin{proposition}\label{P: finitely generated non-pretentious is strongly aperiodic}
    Let $f\in \M$ be a non-pretentious multiplicative function. If $f$ is
    finitely generated, then it is strongly aperiodic. 
\end{proposition}

We postpone the proof of this result to the \cref{appendix}. In fact, we shall establish strong
aperiodicity for a larger class of multiplicative functions in $\M$, 
including, for instance, all aperiodic multiplicative functions 
for which the set 
$\{f(p)\colon p\in \P\}$ is not dense in $\mathbb{S}^1$.

A particular case of finitely generated functions are functions that take finitely many values on the unit circle. We call these multiplicative functions \emph{finite valued}. It is straightforward to check that a multiplicative function is finite valued if and only if there exists a positive integer $k$ such that $f^k\equiv 1$. The following is immediate from Proposition~\ref{P: finitely generated non-pretentious is strongly aperiodic}.
\begin{corollary}\label{C: finite valued non-pretentious is strongly aperiodic}
    Let $f$ be a non-pretentious multiplicative function that is finite valued. Then, $f$ is strongly aperiodic.
\end{corollary}

From this, we further deduce the following result.

\begin{corollary}\label{C: non-pretentious almost supported on roots of unity is strongly aperiodic}
    Let $f$ be a non-pretentious multiplicative function such that $f^{\ell}$ is pretentious for some positive integer $\ell$. Then $f$ is strongly aperiodic.
\end{corollary}
\begin{proof}
Our assumption implies that there exist $t\in \R$ and a Dirichlet character $\chi$ such that $\D(f^{\ell}(n),\chi(n) n^{it})<+\infty$. By raising to an appropriate power (to make $\chi^{m}$ principal), we may assume that $\D(f^{\ell}(n),n^{it})<+\infty$.

\underline{Claim}: Given $\ell \in \N$ and $a\in \mathbb{S}^1$, there is $j\in 
\{0,1,\ldots, \ell-1 \}$ so that $|a-e(j/\ell)|\leq \frac{2\pi}{4 \ell} | 
a^{\ell}-1|.$
\begin{proof}[Proof of the Claim.]
Write $a=e(\theta)$ for some $\theta\in [0,1)$. Take $j\in \{0,1, \ldots, 
\ell-1\} $ so that $\norm{\theta- j/\ell}_{\T}= \min_{0\leq k \leq 
\ell -1} \norm{\theta - k/\ell}_{\T}$. 
Then it is not difficult to see that $\theta \in ((j-1)/\ell, j/ \ell]$
or $\theta \in [j/\ell, (j+1)/ \ell)$, where the intervals are taken $\pmod{1}$.
Without loss of generality, assume that $\theta \in ((j-1)/\ell, j/\ell].$ 
Then we have that $|a-e(j/\ell)|\leq 2 \pi 
\norm{\theta - j/\ell}_{\T}$.
The map $\psi \colon \T \to \T, \psi(x)=\ell x$, when restricted to 
$((j-1)/\ell, j/\ell]$ is invertible, and its inverse is Lipschitz 
continuous with Lipschitz
constant $1/\ell$, and therefore 
$\norm{\theta - j/\ell}_{\T}\leq \frac{1}{\ell} \norm{\ell \theta}_{\T}$.

We deduce that $|a-e(j/\ell)|\leq \frac{2\pi}{\ell} 
\norm{\ell \theta}_{\T}\leq \frac{\pi}{2 \ell}| e(\ell \theta) -1|=
\frac{\pi}{2 \ell}| a^{\ell} -1|$, where for the second inequality we use that 
for all $x\in \R$, $\norm{x}_{\T} \leq \frac{1}{4} |e(x)-1|. $ This concludes the 
proof of the Claim. 
\renewcommand\qedsymbol{$\triangle$}
\end{proof}

Using the Claim, we deduce that for each $p\in \P$, there exists $\theta_p \in 
\{0,1, \ldots, \ell-1\}$ so that $|f(p)p^{-it/\ell} -e(\theta_p/\ell)|
\leq \frac{\pi}{2\ell} |f^{\ell}(p)p^{-it} -1|$. Consider 
the completely multiplicative function $g$ defined on the primes by $g(p)=e(\theta_{p}/\ell)$.
Then 
\begin{align*}
    \D(f(n)n^{-it/\ell}, g(n))^2 &= \frac{1}{2} \sum_{p\in \P} 
    \frac{|f(p)p^{-it/\ell} -e(\theta_p/\ell)|^2}{p}
    \leq \frac{1}{2}
    \frac{\pi ^2}{4 \ell ^2} \sum_{p\in \P} 
    \frac{|f^{\ell}(p)p^{-it} -1|^2}{p}\\
    &= \frac{\pi ^2}{4 \ell ^2} 
    \D(f^{\ell}(n), n^{it})^2 <\infty.
\end{align*}
Therefore, $g$ pretends to be $f(n)n^{-it/\ell}$ and, thus it is 
non-pretentious. However, $g$ is also finite valued, so Corollary 
\ref{C: finite valued non-pretentious is strongly aperiodic} implies that it is 
strongly aperiodic. Therefore, $f(n)n^{-it/\ell}$, and thus $f$, is 
strongly aperiodic. Indeed, suppose there exist $B>0$ and an increasing 
sequence of integers $(N_k)_{k\in\N}$ such that $\inf_{|t'|\leq BN_k,
q\leq B} 
\D(f(n)n^{-it/\ell}, \chi(n)n^{it'}; 1, N_k)$ is bounded. Then, the 
pretentious triangle inequality implies that 
$\inf_{|t'|\leq BN_k, q\leq B} 
\D(g(n), \chi(n)n^{it'}; 1, N_k)$ will be bounded, which is a 
contradiction.
\end{proof}

Finally, we will use the following lemma, an easy consequence of the triangle inequality.
\begin{lemma}\label{L: powers of f pretentious implies gcd pretentious}
    Let $p,q$ be integers and let $f$ be a multiplicative function such that $f^{p}$ and $f^{q}$ are both pretentious. Then $f^d$ is pretentious, where $d=(p,q)$.
\end{lemma}
\begin{proof}
    Let $k,\ell\in\Z$ with $kp-\ell q=d$. If $\D(f^p(n),\chi_1(n)n^{it_1})<+\infty$ and $\D(f^q(n),\chi_2(n)n^{it_2})<+\infty$, then \begin{align*}
        \D(f^d(n),\chi_{1}^k(n)\overline{\chi_{2}^{\ell}(n)}n^{i(kt_1-\ell t_2)})&
        \leq \D(f^{kp}(n),\chi_1^k(n) n^{ikt_1})+\D(f^{\ell q}(n),\chi_2^{\ell}(n) n^{i\ell t_2})\\
        &\leq k\D(f^p(n),\chi_1(n)n^{it_1})+\ell\D(f^q(n),\chi_2(n)n^{it_2})<+\infty.
    \end{align*}
\end{proof}

\begin{corollary}\label{C: minimal power that makes f pretentious}
    Assume that $f$ is a multiplicative function such that $f^{k}$ is pretentious for some positive integer $k$. Then, there exists a positive integer $\ell_0$ such that $f^{\ell}$ is pretentious if and only if $\ell_0\mid \ell$.
\end{corollary}
\begin{proof}
    Take $\ell_0$ to be the minimal positive integer such that $f^{\ell_0}$ is pretentious and apply the Lemma~\ref{L: powers of f pretentious implies gcd pretentious} for $\ell$ and $\ell_0$ to conclude that $\ell_0\mid \ell$.
\end{proof}

Finally, we will need a way to control the distances of aperiodic functions that are not necessarily strongly aperiodic. This can be done by passing to an appropriate subsequence. The central lemma we will need is the following lemma of Elliott \cite[Lemma 17]{Elliott_additive} which asserts that non-pretentious functions cannot pretend to be different twisted characters on all scales. The proof is a simple exercise in applying the pretentious triangle inequality, coupled with some estimates for the distances between twisted characters. For a generalization, see \cite[Lemma 5.3]{Klu-Man-Ter}.

\begin{lemma}
    Let $f$ be a $1$-bounded multiplicative function such that for all $X$, we have that \begin{equation*}
        \inf_{|t|\leq BX, q\leq B} \D(f(n),\chi(n)n^{it};1,X)
    \end{equation*}is bounded. Then, there exists $t\in \R$ and a Dirichlet character $\chi$ of modulus at most $B$, such that 
    $\D(f(n),\chi(n)n^{it})<+\infty$. 
\end{lemma}

The following corollary is straightforward.
    \begin{corollary}\label{C: locally pretentious on all scales implies pretentious}
    Let $B>0$ and let $f$ be a $1$-bounded non-pretentious multiplicative 
    function. Then, there exists a 
    sequence of integers $N_k\to \infty$, such that
    \begin{equation*}
        \lim_{k\to\infty} \inf_{|t|\leq BN_k,q\leq B} \D(f(n),\chi(n)n^{it};1,N_k)=+\infty.
    \end{equation*}
\end{corollary}

\subsection{Some asymptotic estimates for multiplicative functions}

The following lemma is a fairly standard consequence of \eqref{E: limit of multiplicative Folner is 1}.

\begin{lemma}\label{L: multiplicative averages}
    Let $f$ be a bounded completely multiplicative function. Then, \begin{equation*}
     (1-f(p))   \mathlarger{\cesE}_{n\in \Phi_K} f(n) \ll 2^{-K}
    \end{equation*}
    holds for any prime $p$ dividing all the numbers in $\Phi_K$, where the implied constant is absolute.
\end{lemma}

We deduce that if a multiplicative function is not identically $1$,
then its averages along $(\Phi_K)_{K\in \N}$ vanish.

A key idea in our proof strategy is to consider averages along a suitable arithmetic progression, where the steps will be smooth numbers from $\Phi_K$.
This is because we can approximate pretentious multiplicative functions by simpler expressions using the following concentration estimate.

\begin{lemma}[{Cf. \cite[Lemma 2.5]{Klu-Man-Po-Ter}}]
\label{L: concentration estimate}
    Let $f\colon \N \to \mathbb{U}$ be multiplicative and let $\chi$ be a Dirichlet character of conductor $q$. Let $Q,K$ be positive integers, such that $\prod_{j\leq K} p_j |Q $ and $q|Q$. Then, for any $a\in \N$ coprime to $Q$ and any $t\in \R$, we have 
    \begin{multline}\label{E: concentration estimate}
        \sum_{n\leq X}\frac{1}{n}\left|f(Qn+a) -\chi(a)(Qn)^{it}\exp(F(Q,X))\right| \\
        \ll \log X \left(\mathbb{D}(f,\chi(n)n^{it}; p_K, X)+\frac{1}{\sqrt{p_K}} + \oh_{a,t,Q; X\to \infty}(1)\right)
    \end{multline}
    where \begin{equation}\label{E: oscillatory term}
        F(Q,X)=\sum_{\substack{p\leq X\\
        p\nmid Q}} \frac{f(p)\overline{\chi(p)}p^{-it} -1}{p}.
    \end{equation}
We note here that the implicit constant in \eqref{E: concentration estimate}
is absolute.
\end{lemma}

The result in \cite{Klu-Man-Po-Ter} uses Ces\'{a}ro averages, but the same proof works with logarithmic averages with a logarithmic version of the Tur\'{a}n-Kubilius inequality.

The previous lemma is useful whenever the term $\mathbb{D}
(f(n),\chi(n)n^{it},K,X )$ is small and, in particular, for pretentious 
functions $f$ (with the appropriate choice of $\chi$ and $t$).

Finally, we will need the following result of Tao on the two-point correlations of multiplicative functions. 
\begin{customthm}{A}{\cite[Theorem 1.3]{Tao-Chowla}}\label{T: Tao theorem}
Let $a_1,a_2,b_1,b_2$ be natural numbers such that $a_1b_2-a_2b_1\neq 0$. Let $\e>0$ and suppose that $B$ is sufficiently large depending on $\e,a_1,a_2,b_1,b_2$.
Let $f,g\colon\N\to\C$ be $1$-bounded multiplicative functions. Then, for any $X\geq B$, if $$\mathbb{D}(f(n),\chi(n)n^{it}; 1,X)\geq B$$
for all Dirichlet characters of modulus at most $B$
and all real numbers $|t|\leq BX$, then
$$\left|\mathlarger{\logE}_{n\leq X}
f(a_1n+b_1)g(a_2n+b_2)\right|\leq\varepsilon.$$
\end{customthm}

This result illustrates the necessity to consider strongly aperiodic functions as separate from general aperiodic functions. In fact, it is known that there are examples of non-pretentious functions \cite{MRT_averaged} (see also \cite{Klu-Man-Ter}),
which locally pretend to be $n^{it_k}$ on a sequence of scales $X_k\to \infty$ and a fast-growing sequence of positive real numbers $(t_k)_{k\in\N}$. These functions are non-pretentious and satisfy \begin{equation*}
    \lim\limits_{k\to \infty} \mathlarger{\logE}_{n\leq X_k}f(a_1n+b_1)g(a_2n+b_2)=1.
\end{equation*}

Since these functions are non-pretentious, one can show using Corollary \ref{C: locally pretentious on all scales implies pretentious} and Theorem \ref{T: Tao theorem} that there is a subsequence along which the averages are very small. We use this crucially in the proof of \ref{T: optimal conditions for Bohr recurrence}. However, when one is dealing with several non-pretentious functions at the same time, it is not clear that one can find a subsequence that works for all functions simultaneously. This is the main reason that we have to impose the finitely generated assumption in Theorem \ref{T: multiplicative recurrence for finitely generated systems}.

\section{Optimal conditions for multiplicative recurrence}\label{section_optimal-conditions}

In this section, we prove the implication $(2)\implies (3)$ of Theorem \ref{T: optimal conditions for Bohr recurrence}, i.e., we prove that for 
$a,c \in \N$ and $b,d \in \Z$ with $(a,b,c,d)=1$, if for every $f\in \M$,
\begin{equation}\label{liminf f-f}
    \liminf_{n\to \infty} |f(an+b)-f(cn+d)|=0,
\end{equation}
then we must have $a=c$, and $b=d$ or $a\mid bd$.

First, we prove that given two sequences $(p_n)_{n\in \N}, (q_n)_{n\in \N}$
of natural numbers, if the set $R=\left\{ {p_n}/{q_n} \colon n\in 
\N\right \}$ is a set of multiplicative recurrence, then for all $f\in \M$, 
$$\liminf_{n\to \infty} |f(p_n)-f(q_n)|=0.$$
This implies that for the set 
$S=\left\{ (an+b)/(cn+d) \colon n\in \N\right \}$ to be a set 
of multiplicative recurrence, we must have $a=c$, and 
$b=d$ or $a\mid bd$.
We conjecture that these assumptions on $a,b,c$ are also 
sufficient for multiplicative recurrence for general systems.

\begin{lemma}\label{L: recurrence implies Bohr recurrence}
Let $(p_n)_{n\in \N}, (q_n)_{n\in \N}$ be two sequences of natural numbers.
If $S=\left\{ {p_n}/{q_n} \colon n\in 
\N\right \}$ is a set of multiplicative recurrence, then for all $f\in \M$ 
\begin{equation}\label{eq rec implies bohr}
    \liminf_{n\to \infty} |f(p_n)-f(q_n)|=0.
\end{equation}
\end{lemma}

\begin{proof}
Let $\e>0$ and $f\colon \N \to \mathbb{S}^1$ be a completely multiplicative 
function. Consider the unit circle with the Borel 
$\sigma$-algebra and the Lebesgue measure, and the multiplicative 
system on this space generated by the transformations 
$T_n\colon \mathbb{S}^1 \to \mathbb{S}^1$, $T_nz=f(n)z$ for every 
$n\in \N$ and $z\in \mathbb{S}^1$.

Let $A$ be any arc of length $\e/2$ in $\mathbb{S}^1$. Observe that 
$\m(A)>0$ and since $R$ is a set of multiplicative recurrence, there exists
a sequence $n_k \to \infty$ so that for 
all $k\in \N$, $$\m(T^{-1}_{p_{n_k}}A\cap T^{-1}_{q_{n_k}}A)>0.$$
From the previous we get that for all $k\in \N$ there exists $z_k\in 
\mathbb{S}^1$ such that $T_{p_{n_k}}z_k\in A$ and $T_{q_{n_k}}z_k\in A$. 
By the choice of $A$, this implies that $|T_{p_{n_k}}z_k-
T_{q_{n_k}}z_k|< \e$, or equivalently 
\begin{equation*}
    \left|f(p_{n_k})-f(q_{n_k})\right|<\e
\end{equation*}
for all $k\in \N$. This concludes the proof.
\end{proof}

Given $a,c\in \N$, $b,d \in \Z$ we want to find necessary conditions on these parameters so that the set
$S=\big\{(an+b)/(cn+d) \colon n\in \N\big\}$ is a set of multiplicative recurrence.
In view of Lemma \ref{L: recurrence implies Bohr recurrence}, it suffices to examine under what conditions \eqref{liminf f-f} holds for the quadruple $(a,b,c,d)$. 
These will arise by considering \eqref{liminf f-f} for modified characters. Here, a \emph{(twisted) modified character} is a multiplicative function $f\in \M$, such that there exist $t\in \R$ and a Dirichlet character $\chi$ such that 
$f(p)=\chi(p)p^{it}$ for all primes $p$ with $\chi(p)\neq 0$. 
We know that $\chi(p)=0$ for at most finitely many primes, so by  
modifying the values of $\chi$ at those primes we get a function 
$\wt{\chi}\in \M$ with $f(p)=\wt{\chi}(p)p^{it}$ for all $p\in \P$.

The first condition is that we must have $a=c$. This condition was already observed in \cite[Proposition 3.2]{Moreira&friends}. We provide a different proof for convenience.
\begin{proposition}\label{P: a=c}
    Let $a,c\in \N$, $b,d \in \Z$ and assume that \eqref{liminf f-f} holds 
    for all $f\in \M$. Then $a=c$.
\end{proposition}
\begin{proof}
    We pick $f(n)=n^{it}$ and use the fact that $(u_1n+u_2)^{it}=(u_1n)^{it}+\Oh_{t,u_1,u_2}(1/n)$ for all integers $u_1,u_2$. We conclude that 
    \begin{align*}
        |f(an+b)-f(cn+d)|&=|(an+b)^{it}-(cn+d)^{it}|\\
        &=|(an)^{it}-(cn)^{it}|+\oh_{n\to \infty}(1)=|(ac^{-1})^{it}-1|+\oh_{n\to \infty}(1).
    \end{align*}
    If $a\neq c$, then there is $t$ so that $(ac^{-1})^{it}=-1$, and using
    the previous we see that \eqref{liminf f-f} cannot hold for
    $f(n)=n^{it}$, which is a contradiction.
\end{proof}
Hence, from now on, we assume that $a=c$, and we study the set 
$$S=\left\{\frac{an+b}{an+d} \colon n\in \N\right\}$$
with $(a,b,d)=1$. If $b=d$, \eqref{liminf f-f} is satisfied trivially, and 
also $S=\{1\}$, which is trivially a set of multiplicative recurrence. 
Therefore, we also assume that $b\neq d$.

\begin{proposition}\label{P: optimal conditions}
    Let $a\in \N$, $b,d \in \Z$ with $(a,b,d)=1$, $b\neq d$
    and assume that $a$ does not divide $bd$. Then, there is a function $f\in\M$ for which \eqref{liminf f-f} fails.
\end{proposition}

\begin{proof}
Since $a\nmid bd$, there is 
$p\in \P$ such that $p^{k} \parallel a$, $p^{\ell} \parallel bd$ and $k>\ell$.

First of all, we consider the case $\ell=0$.
Then, $p|a$, $p \nmid b$ and $p \nmid d$. 

If $p>2$, we
take $u \in \N$ such
that $b \not \equiv d \pmod{p^u} $. As $p$ is an odd prime, the multiplicative
group $\Z_{p^u} ^{\times}$ is cyclic. Let $g$ be a generator. Then 
for each $m\in \Z_{p^u} ^{\times}$, there is a unique $r\in \{0,1, \dots, \phi(p^u)-1\} $ so that $m=g^r$, where $\phi$ denotes
Euler's totient function. Consider the Dirichlet character 
$\chi \mod p^u$ defined by $\chi(g^r) = e(r/\phi(p^u))$. 
For any distinct $r_1, r_2 \in \{0,1,\ldots, \phi(p^u)-1\}$, we have $|e(r_1/\phi(p^u))-e(r_2/\phi(p^u)) |\geq 2\pi/\phi(p^u)$.
Then for each $n\in \N$, we have $an+b \not \equiv an+d \pmod{ p^u}$ and 
$(an+b, p^u)=(an+d, p^u)=1$. We infer that $\chi(an+b), \chi(an+d)\in 
\{e(r/\phi(p^u))\colon r \in \{0,1, \dots, \phi(p^u)-1\}\}  $ and 
$\chi(an+b)\neq \chi(an+d)$. Thus, $|\chi(an+b)-\chi(an+d)|\geq 2\pi/\phi(p^u)$. 
Now, consider the modified Dirichlet character $\wt{\chi}$ defined on 
the primes by
\begin{equation*}
    \wt{\chi}(q)=\begin{cases}
        \chi(q),& q\neq p\\
        1,& q=p
    \end{cases}.
\end{equation*}
For each $n\in \N$, $p \nmid an +b ,an+d$. Thus, $\wt{\chi}(an+b)= 
\chi(an+b)$ and $\wt{\chi}(an+d)= \chi(an+d)$, from which 
we conclude that 
$$\liminf_{n \to \infty} |\wt{\chi}(an+b)-
\wt{\chi}(an+d)| \geq 2\pi/\phi(p^u) >0. $$

On the other hand, if $p=2$ then $2|a, 2\nmid b$ and $2\nmid d$. The numbers $b, d$ are odd, so $2 \mid (b-d)$. Consider the minimal number $u \in \N$ so that 
$2^ u \nmid (b-d) $. Then, $u \geq 2$ and there exists a Dirichlet character 
$\chi \mod 2^u$ so that $\chi(b) \neq \chi(d)$. Again, consider the
modified Dirichlet character $\wt{\chi}$ defined by 
\begin{equation*}
    \wt{\chi}(q)=\begin{cases}
        \chi(q),& q\neq 2\\
        1,& q=2
    \end{cases}.
\end{equation*}
Then, $\wt{\chi}(b)=\chi(b)$ and $ \wt{\chi}(d)={\chi}(d)$, so 
$\wt{\chi}(b)\neq \wt{\chi}(d)$.
Write $a=2^k a_0$ with $2\nmid a_0$. Then for each $n\in \N$, 
$2^k \mid a n$, $2^{u-1}\mid (b-d)$ (by the minimality of $u$), hence
$2^{k+u-1} \mid a n (b-d)$. Since $k\geq 1$, we have that $k+u-1 \geq u$, so $2^u \mid a n(b-d)$. Therefore, we infer that
$anb \equiv and \pmod{ 2^u} $, which implies that 
$b(an+d)=anb+bd \equiv and+bd = d(an+b) \pmod{2^u}$. Now, 
$(2,b(an+d))= (2,d(an+b) )=1  $, so we have that 
\begin{align*}
\wt{\chi}(b)\wt{\chi}(an+d)&=
\wt{\chi}(b(an+d))=\chi(b(an+d))= \chi(d(an+b))\\
&=\wt{\chi}(d(an+b)) = \wt{\chi}(d)\wt{\chi}(an+b),
\end{align*}
from which we obtain that ${\wt{\chi}(b)}/{\wt{\chi}(d)}
={\wt{\chi}(an+b)}/{\wt{\chi}(an+d)} $. Then we have that
\begin{align*}
\liminf_{n \to \infty} |\wt{\chi}(an+b)-\wt{\chi}(an+d)| &=
\liminf_{n \to \infty} \left|\frac{\wt{\chi}(an+b)}{\wt{\chi}(an+d)}-1\right| =
\liminf_{n \to \infty} \left|
\frac{\wt{\chi}(b)}{\wt{\chi}(d)}-1\right|\\
&= 
| \wt{\chi}(b)- \wt{\chi}(d) | >0.
\end{align*}

Now, assume that $\ell >0$. Then, $p\mid a, p\mid bd $, and since $(a,b,d)=1$, we have that
either $p\nmid b$ or $p\nmid d$. Assume without loss of generality that
$p \nmid d$. We write $a=p^k a_0, b=p^{\ell} b_0$, where 
$p\nmid a_0, b_0$ and $k>\ell$. Let $\chi$ be any Dirichlet character modulo $p$. Then
$\{ \chi(m) \colon m\in \Z_p ^{\times} \}$ is a finite subset 
of $\mathbb{S}^1$, so there is $\theta \in [0,1)$, such that $e(\theta) \chi(m_1) \neq \chi(m_2)$ for all $m_1, m_2 \in \Z_p ^{\times}$.
Let $\eta= \min_{m_1, m_2 \in \Z_p ^{\times}} |e(\theta) \chi(m_1) - \chi(m_2)|
>0$. 
Then, consider the
modified Dirichlet character $\wt{\chi}$ defined by 
\begin{equation*}
    \wt{\chi}(q)=\begin{cases}
        \chi(q),& q\neq p\\
        e(\theta/\ell),& q=p
    \end{cases}.
\end{equation*}
For every $n\in \N$, we have $p\nmid b_0, d$ and since $k>k-\ell\geq 1$
we deduce that 
\begin{align*}
    |\wt{\chi}(an+b) - \wt{\chi}(an+d) |&=
    |\wt{\chi}(p^{\ell}(p^{k-\ell}a_0 n+b_0)) - \wt{\chi}(p^k a_0 n+d) |\\
    &=|e(\theta) \chi(b_0) - \chi(d) |\geq \eta
\end{align*}
so $\liminf_{n \to \infty} |\wt{\chi}(an+b) - \wt{\chi}(an+d) | \geq \eta >0$. The conclusion follows.
\end{proof}

Propositions \ref{P: a=c} and \ref{P: optimal conditions} imply the following,
which establishes $(2)\implies (3)$ of Theorem \ref{T: optimal conditions for Bohr recurrence}:
\begin{corollary}\label{C: recurrence implies a=c and au divides bd}
Let $a, c \in \N$, $b, d\in \Z$ with $(a,b,c,d)=1$ so that 
\eqref{liminf f-f} holds for 
all $f\in \M$. Then $a=c$, and $b=d$ or $a\mid bd$.
\end{corollary}

\section{Correlations of multiplicative functions along progressions with multiplicative steps}
\label{section_cor-along-progr-with-multi-steps}

In this section, we study the asymptotic behavior of averages involving the 
correlations ${f(a_1n+b_1)g(a_2n+b_2)}$,
where we restrict $n$ to a suitable arithmetic progression coming from the 
F{\o}lner sequence $\Phi_K$ defined in 
\eqref{E: definition of multiplicative Folner}.

Our main proposition asserts that for any $f,g$ we can make 
the previous correlation small on average on suitable arithmetic progressions 
(possibly by passing to a subsequence), unless both $f$ and $g$ are
modified characters.
\begin{proposition}\label{P: binary correlations proposition}
    Let $\delta>0$, $a_1,a_2\in \N,b_1, b_2\in \Z$ with $a_1b_2\neq a_2b_1$ and $(a_1,a_2)=(a_1,b_1)=(a_2,b_2)=1$.
    Suppose $f,g\colon \N\to \mathbb{U}$ are
    completely multiplicative functions and let $\Phi_K$ be defined as in 
    \eqref{E: definition of multiplicative Folner}. Then, for each $Q\in \Phi_K$, 
    there exists a non-negative integer $r_Q$, such that all of the following hold:

    (i) There exists a positive constant $B$ depending only on $a_1,b_1,a_2,b_2,\delta$ and $K$, such that, for all $X$ sufficiently large, if $$\inf_{|t|\leq BX,q\leq B} \D(f(n),\chi(n)n^{it}; 1, X)>B \:\:\text{ or }\:\: \inf_{|t|\leq BX,q\leq B }\D(g(n),\chi(n)n^{it};1, X)>B,$$then \begin{equation*}
       \sup_{Q\in \Phi_K} \left| \mathlarger{\logE}_{n\leq X}
       f(a_1Q^2n+a_1r_Q+b_1)g(a_2Q^2n+a_2r_Q+b_2)  \right|<\delta.
    \end{equation*}

(ii) If one of the functions $f,g$ is pretentious and the other non-pretentious, then we have \begin{equation*}
  \lim_{K\to \infty} \sup_{Q\in \Phi_K}\limsup_{X\to \infty} \left|\mathlarger{\logE}_{n\leq X} f(a_1Q^2n+a_1r_Q+b_1) g(a_2Q^2n+a_2r_Q+b_2)   \right|=0.
\end{equation*}

    (iii) If both $f$ and $g$ are pretentious, then either
     \begin{equation*}
 \lim\limits_{K\to \infty}\  \limsup\limits_{X\to \infty}\     
 \left|\mathlarger{\cesE}_{Q\in \Phi_K} \mathlarger{\logE}_{n\leq X} f(a_1Q^2n+a_1r_Q+b_1) 
 g(a_2Q^2n+a_2r_Q+b_2)   \right|=0,
    \end{equation*}
    or there exist $t\in \R$ and Dirichlet characters $\chi_f,\chi_g$ with conductors $q_f, q_g$ respectively, which satisfy $q_f\mid a_1^{\infty}$,
    $f(p)=\chi_f(p)p^{it}$ for all primes $p$ with $p\nmid a_1$,\footnote{Namely, $f$ is a modified character.} and $\D(g(n),\chi_g(n)n^{-it})<+\infty$.
    Furthermore, if $g=\overline{f}$, we have
    that $f(n)=n^{it}$ for all $n\in \N$.
\end{proposition}

\begin{remark}\label{R: symmetric}
In the previous proposition, if $a_1=1$, then in the second part of (iii)
we get that $f(p)=p^{it}$ for all $p\in \P$.

Furthermore, Proposition \ref{P: binary correlations proposition} is symmetric
in $f$ and $g$. That is to say, if in the proof of the proposition we write
$Q=\prod_{p\leq p_K} p^{\theta_{p,Q}}\in \Phi_K$ as $Q=A'W'$, where 
$A'=\prod_{\substack{p\leq p_K\\ p\mid a_2}} p^{\theta_{p,Q}}$ and 
$W'=\prod_{\substack{p\leq p_K\\ p\nmid a_2}} p^{\theta_{p,Q}}$, then 
there exists $r_Q'$ so that
(i) and (ii) of the proposition hold, and instead of (iii) we get the following:

(iii$\hspace*{0.05cm}'\hspace*{-0.05cm}$) If both $f$ and $g$ are pretentious, then either
\begin{equation*}
 \lim\limits_{K\to \infty}\  \limsup\limits_{X\to \infty}     
 \left|\mathlarger{\cesE}_{Q\in \Phi_K} \mathlarger{\logE}_{n\leq X} 
 f(a_1Q^2n+a_1r_Q '+b_1)g(a_2Q^2n+a_2r_Q '+b_2)  \right|=0,
    \end{equation*}
    or there exist $t\in \R$ and Dirichlet characters $\chi_f,\chi_g$ 
    with conductors $q_f, q_g$ respectively, which satisfy 
    $q_g\mid a_2^{\infty}$,
    $g(p)=\chi_g(p)p^{it}$ for all primes $p$ with 
    $p\nmid a_2$,\footnote{Namely, $g$ is a modified character.} 
    and $\D(f(n),\chi_f(n)n^{-it})<+\infty$.
\end{remark}

The binary correlations appearing in the previous proposition will be 
relevant to our main problem of determining whether 
$$\liminf_{n\to \infty} |f(an+b)- g(cn+d)| =0.$$

\begin{proof}[Proof of Proposition \ref{P: binary correlations proposition}]
All asymptotics are allowed to depend on the fixed parameters $a_1, a_2$, $b_1, b_2$ and the multiplicative functions $f,g$ and we will suppress this dependence in our notation. We will be explicit only when the asymptotics depend on $K$ or $X$.

Given $Q \in \Phi_K$ of the form $Q=\prod_{p\leq p_K} p^{\theta_{p,Q}}$, we
write $Q=AW$, where \begin{equation*}
     A=\prod_{\substack{p\leq p_K\\ p\mid a_1}} 
p^{\theta_{p,Q}} \:\:\text{ and }\:\: 
W=\prod_{\substack{p\leq p_K\\ p\nmid a_1}} p^{\theta_{p,Q}}.
\end{equation*}
We have that $(A,W)=1$ and 
we can assume that $K$ is large enough so that $a_1 \mid A$.
We examine the averages
\begin{equation}\label{E: first averages along progression}
  \mathlarger{\logE}_{n\leq X} f(a_1Q^2n+a_1r_Q+b_1)g(a_2Q^2n+a_2r_Q+b_2)   
\end{equation}where $Q\in \Phi_K$.

We will describe now how we choose $r_Q$. If $\D(f(n), \chi_f(n) n^{it_f})<+\infty$ where the Dirichlet character $\chi_f$ has conductor $q_f$, then we set 
\begin{equation}\label{E: definition of m_1,v_1}
    \mu_1= \max\{\ell\colon p^{\ell} \parallel q_f,\: p \mid a_1  \},\  
\nu_1= \max\{\ell\colon p^{\ell} \parallel  q_f,\: p \nmid a_1  \},
\end{equation}
while if $f$ is non-pretentious, we set $\mu_1=\nu_1=1$.
Similarly, if $\D(g(n), \chi_g(n) n^{it_g})<+\infty $ for a character $\chi_g$ with conductor $ q_g$, we set \begin{equation}\label{E: definition of m_2,v_2}
    \mu_2= \max\{\ell\colon p^{\ell} \parallel q_g,\: p \mid a_1  \},\ \nu_2= \max\{\ell\colon p^{\ell} \parallel q_g,\: p \nmid a_1  \}
\end{equation}
while if $g$ is non-pretentious, then we set $\mu_2=\nu_2=1$ as above.
Now, we define \begin{equation}\label{E: definition of m,v}
    \mu=\max\{\mu_1, \mu_2\},\ \nu=\max\{\nu_1, \nu_2\}. 
\end{equation}

Using the Chinese Remainder Theorem, we can now choose an integer $r_Q$ satisfying
$0\leq r_Q < \prod_{p^{\ell}\parallel A} p^{\ell+1 + \mu} \prod_{p^{\ell}\parallel W} p^{\ell+1 + \nu}
< Q^2 $,
such  that 
\begin{align}\label{E: congruence 1}
    r_Q \equiv -a_1^{-1}b_1 + a_1^{-1}W \pmod{ p^{\ell+1+\nu}} &\text{ if } p^{\ell}\parallel  W\\
\label{E: congruence 2}
    r_Q \equiv (A b_1 -b_2 W ) a_2^{-1}W^{-1} \pmod{ p^{\ell+1+\mu}} &\text{ if } p^{\ell}\parallel A
\end{align}Note that the element $a_1^{-1}$ is well defined in the ring $\Z_{p^{\ell+1+\nu}}$, since $p\mid W$ which is coprime to $a_1$. Similarly, the inverses of $a_2, W$ are well-defined in \eqref{E: congruence 2}.

If $p^{\ell}\parallel W$, then $p^{\ell}\parallel a_1 r_Q +b_1$, and if $p^{\ell}\parallel A$, then 
$p^{\ell}\parallel a_2r_Q +b_2$, so if we set \begin{equation}\label{E: definition of l_Q, m_Q}
    \ell_Q=\frac{a_1r_Q + b_1}{W}, \:\:\:
m_Q= \frac{a_2r_Q + b_2}{A},
\end{equation}
then $\ell_Q$ and $m_Q$ are integers. Also, it is not too difficult to see that 
from the definitions of $\ell_Q, m_Q$ we have $(\ell_Q,W)=1=(m_Q,A)$. 

Without loss of generality, we will assume below that $a_1b_2>a_2b_1$ and denote $u:=a_1 b_2 - a_2 b_1.$

We observe that $\ell_Q$ shares no common factors with $A$, since $b_1$ is 
coprime to $a_1$. Thus, $(\ell_Q,AW)=1$, which also implies that 
$(\ell_Q,a_1A^2W^2)=1$. 
Since $b_2$ is coprime to $a_2$, we have that $(m_Q,a_2)=1$.
We will now prove that $(m_Q, W^2)=u$.
We may assume without loss of generality that $K$ is sufficiently large so that 
every prime factor of $u=a_1b_2 -a_2 b_1$ is contained in $\{p_1, \ldots, p_K\}$,
and so that the exponent of every prime appearing in the prime factorization 
of $u$ is smaller than $2^K$. 
From $(a_1,a_2)=(a_1,b_1)=(a_2,b_2)=1$, we get that $(u,a_1a_2)=1$, and also 
$(u,A)=1$. Therefore, if $p\mid u$, then $p\mid W$ and 
the exponent of every prime dividing $W$ is at least $2^K$, so 
we get that $u\mid W$.

Let $p\leq p_K$ be a prime so that $p^{s_0}\parallel u$, $p^{s_1}\parallel m_Q$, with 
$s_0\geq 1$, $s_1 \geq 0$. Let also $p^{s_2}\parallel W$, with $s_2>2^K>s_0$.
Then $p^{s_2}\parallel a_1 r_Q +b_1 $.
We infer that $p^{s_0}\mid u= a_1(a_2r_Q + b_2)- a_2(a_1r_Q+b_1)$, 
and since $p^{s_2}\mid a_2(a_1 r_Q +b_1)$ and $s_2>s_0$, we obtain that 
$p^{s_0}\mid a_1(a_2r_Q + b_2)$.
Using that $(p,A)=1$ and the definition of $m_Q$, we get that 
$p^{s_0}\mid m_Q$, so $s_1\geq s_0$. Therefore, $u\mid m_Q$, and since 
also $u\mid W$, we have that $u\mid (m_Q,W)$. 

Now assume $p$ is a prime that divides 
$(m_Q,W)$, and take $s_0, s_1, s_2 \in \N$ so that 
$p^{s_0} \parallel (m_Q,W)$, $p^{s_1} \parallel m_Q$, $p^{s_2} \parallel W$.
Since $s_0\leq s_1, s_2$, and $s_2 > 2^K$. Then $p^{s_2}\parallel a_1 r_Q + b_1$
and $p^{s_1}\mid a_2 r_Q + b_2$, we obtain 
that $p^{\min\{s_1,s_2\}}\mid a_1b_2-a_2b_1=u$, and by the assumptions on 
$s_0, s_1, s_2$, we have $\min\{s_1,s_2\}\geq s_0$. 
As a result, $p^{s_0} \mid u$, and since $p$ was an arbitrary prime factor 
of $(m_Q,W)$, we obtain that $(m_Q, W) \mid u$. Combining with $u\mid (m_Q,W)$, 
we get that $(m_Q,W)=u$.
We will actually prove that $(m_Q, W^2)=u$. 
Indeed, if we let $u= q_1 ^ {\beta_1} \cdots q_s ^{\beta_s}$ be the prime 
factorization of $u$, then
$u \mid (m_Q , W^2)$, and it is not too difficult to see 
that $(m_Q , W^2)\mid u^2 = p_1 ^ {2\beta_1} \cdots p_s ^{2\beta_s} $. 
Write $(m_Q , W^2)= 
p_1^{\gamma_1} \cdots p_s^{\gamma_s}$, where $\gamma_i \in \{\beta_i, \dots, 2\beta_i\}$. If $\gamma_i>\beta_i $, then we have that $p_i^{\gamma_i}\mid m_Q $, and since the exponent of $p_i$ in $W$ is 
at least $2^K>2\beta_i $ (which holds if $K$ is large enough), 
we have that $p_i ^{\gamma_i} \mid  W $. Hence, $p_i^{\gamma_i}\mid 
(W, m_Q)$, which is a contradiction. Therefore, for each $i$ we have 
$\gamma_i \leq \beta_i $, so $(m_Q, W^2) \mid u$, and,
combining this with the previous facts, we deduce that 
$ (m_Q, W^2) = u $.

To summarize all the above, we have 
\begin{equation}\label{E: coprimalities}
    (\ell_Q,a_1 A^2 W)=1, \:\:\:(a_2AW^2, m_Q)=u
\end{equation}
and $\ell_Q \equiv m_Q \pmod{p} $ whenever $p\mid a_1$ by \eqref{E: congruence 2}.

We rewrite the averages in \eqref{E: first averages along progression} as 
\begin{multline}\label{E: second averages along progression}
      f(W){g}(A) \mathlarger{\logE}_{n\leq X}
    f(a_1 A^2 W n + \ell_Q ) g(a_2A W^2 n + m_Q)\\
    =g(u)f(W)g(A)\mathlarger{\logE}_{n\leq X}
    f(a_1 A^2 W n + \ell_Q ) g\left(u^{-1}a_2A W^2 n 
    + u^{-1}m_Q\right) .
\end{multline}

Since $a_1 b_2\neq a_2b_1$, for each $Q\in \Phi_K$ we have 
$a_1 A^2 Wu^{-1}m_Q \neq u^{-1}a_2A W^2 \ell_Q$.
Hence, we can apply Theorem \ref{T: Tao theorem} for these 
averages for all binary correlations corresponding to $Q\in 
\Phi_K$ (which are finitely many). We get then that there exists $B>0$ 
depending only on $a_1,a_2,b_1,b_2$, $\delta$ and $K$, such that if
$X\geq B$ is sufficiently large and the 1-bounded multiplicative functions $h_1,h_2$ satisfy 
\begin{equation}
     \left|\mathlarger{\logE}_{n\leq X}
     h_1\left(a_1A^2Wn+\ell_Q\right)h_2\left(u^{-1}a_2AW^2n+u^{-1}m_Q \right)\right|\geq \delta,
\end{equation}
then \begin{equation*}
   \inf_{|t|\leq BX,q\leq B}\D(h_1(n),\chi(n)n^{it};1,X)\leq B \:\: \text{ and } \:\: \inf_{|t|\leq BX,q\leq B}\D(h_2(n),\chi(n)n^{it};1,X)\leq B.
\end{equation*}
Applying this theorem finitely many times, we find a $B$ that works for all choices of $Q\in \Phi_K$.

Part (i) of the Proposition follows immediately from the conditions above.
Indeed, if there exists at $Q\in \Phi_K$ such that 
\begin{equation*}
    \left|\mathlarger{\logE}_{n\leq X}
    f(a_1 A^2 W n + \ell_Q ) g\left(u^{-1}a_2A W^2 n 
    + u^{-1}m_Q\right)\right|\geq \delta,
\end{equation*}
then both of the inequalities 
$$\inf_{|t|\leq BX,q\leq B}\D(h_1(n),\chi(n)n^{it};1, X)\leq B \:\:\text{ and }\:\: \inf_{|t|\leq BX,q\leq B}\D(h_2(n),\chi(n)n^{it};1, X)\leq B$$ hold, which 
contradicts the assumptions in part (i).

We now prove part (ii). Without loss of generality, we suppose that $f$ is 
pretentious and $g$ is non-pretentious. Carrying out similar factorizations 
as in \eqref{E: second averages along progression}, it suffices to show that 
\begin{equation}\label{E: third averages along progression}
    \sup_{Q\in \Phi_K} \limsup_{X \to \infty} 
    \left|  \mathlarger{\logE}_{n\leq X}
    f(a_1 A^2 W n + \ell_Q) {g}\left(u^{-1}a_2A W^2 n 
    + u^{-1}m_Q\right)\right| =\oh_{K\to \infty}(1).
\end{equation}

Assume that $\D(f(n),\chi_f(n) n^{it_f})<+\infty$ for some Dirichlet character 
$\chi_f$ and a real number $t_f$. 
We apply Lemma \ref{L: concentration estimate} to deduce that 
\begin{multline*}
    \mathlarger{\logE}_{n\leq X}\left|f(a_1 A^2 W n )-\chi_f(\ell_Q)
    (a_1A^2Wn+\ell_Q)^{it_f}\exp(F(K,X))\right| \\
    \ll 
    \D(f(n),\chi_f(n)n^{it_f};p_K,X)+\frac{1}{\sqrt{p_K}} + 
    \oh_{t_f,Q;X\to \infty}(1),
\end{multline*}
where $$F(K,X)=\sum_{p_K<p\leq X} \frac{f(p)\overline{\chi_f(p)}p^{-it_f}-1}{p}$$
and $p_K$ is the $K$-th prime.
We remark in passing that $F(K,X)$ does not depend 
on $a_1A^2W$ (or, equivalently, on $Q$) but only on $K$. We have that 
$\D(f(n),\chi_f(n)n^{it_f};p_K,X)= \oh_{K\to \infty}(1)$ and we also pick $K$ 
sufficiently large so that 
the conductor $q_f$ of $\chi_f$ divides $Q$ for any $Q\in \Phi_K$.

Using the last approximation, we rewrite the averages in 
\eqref{E: third averages along progression} as 
\begin{align}\label{E: applying concentration estimate in pretentious-non-pretentious case}
\mathlarger{\logE}_{n\leq X} &
f(a_1 A^2 W n + \ell_Q )g\left(u^{-1}a_2A W^2 n + u^{-1}m_Q\right) \\
& = \chi_f(\ell_Q) (a_1 A^2 W )^{it_f} \exp{(F(K,X))} 
\mathlarger{\logE}_{n\leq X}g\left(u^{-1}a_2A W^2 n +  u^{-1}m_Q\right)n^{it_f} \notag \\
& \hspace*{6.7cm} + \oh_{t_f,Q;X\to \infty}(1)  +\oh_{K\to \infty}(1). \notag
\end{align}
Using the mean-value theorem, we have \begin{equation}\label{E: Taylor asymptotic for archimedean characters}
n^{it_f}=\left(u^{-1}a_2A W^2\right)^{-it_f} \left(u^{-1}a_2A W^2 n + u^{-1} 
m_Q \right)^{it_f} +\Oh_{K}\left(\frac{|t_f|}{n}\right).
\end{equation}
Since $g$ is non-pretentious, we have that $g(n)n^{it_f}$ does not pretend to be any Dirichlet character. Therefore by Corollary \ref{C: mean value does not depend on progression} we have 
\begin{equation}\label{E: non pretentious g has zero average along progression}
\mathlarger{\logE}_{n\leq X}
g\left(u^{-1}a_2A W^2 n + u^{-1}m_Q\right)
\left(u^{-1}a_2A W^2 n + u^{-1}m_Q\right)^{it_f} 
 =\oh_{K;X\to \infty}(1).
\end{equation}

Combining \eqref{E: applying concentration estimate in pretentious-non-pretentious case}, \eqref{E: Taylor asymptotic for archimedean characters} and \eqref{E: non pretentious g has zero average along progression}, we conclude that  
$$ \sup_{Q\in \Phi_K}\limsup_{X \to \infty} \left|
\mathlarger{\logE}_{n\leq X}
f(a_1 A^2 W n + \ell_Q ) g\left(u^{-1}a_2A W^2 n 
+ u^{-1} m_Q\right)\right|= \oh_{K\to \infty}(1)$$
and part (ii) follows.

Finally, we prove part (iii). If both $f$ and $g$ are pretentious, we can 
find $t_f,t_g\in \R$ and Dirichlet characters $\chi_f,\chi_g$, such that 
$\D\left(f(n), \chi_f(n) n^{it_f}\right)$ and $\D\left(g(n), \chi_g(n) n^{it_g}\right)$ are both finite. 

Assume that \begin{equation}\label{E: the negation of part iii in Proposition 4.1}
    \limsup\limits_{K\to\infty}\  \limsup\limits_{X\to\infty}\     
    \left|\mathlarger{\cesE}_{Q\in \Phi_K} \mathlarger{\logE}_{n\leq X} 
    f(a_1Q^2n+a_1r_Q+b_1){g(a_2Q^2n+a_2r_Q+b_2)}   \right|>0,
\end{equation}
and we will prove that $f$ is a modified character  modulo $q_f$ with 
$q_f\mid a_1^{\infty}$ and that $t_g=-t_f$.
We rewrite the inner double average as 
\begin{equation}\label{E: double average along progressions}
  \mathlarger{\cesE}_{Q\in \Phi_K} g(u)f(W)g(A) 
  \mathlarger{\logE}_{ n\leq X}
    f(a_1 A^2 W n + \ell_Q ) g\left(u^{-1}a_2A W^2 n + u^{-1}m_Q \right) 
\end{equation}
We apply Lemma \ref{L: concentration estimate} to deduce that 
\begin{multline*}
    \mathlarger{\logE}_{ n\leq X} \left|f(a_1 A^2 W n + \ell_Q )-\chi_f(\ell_Q)(a_1A^2Wn)^{it_f}\exp(F(K,X))\right| \\
   \ll \D(f(n),\chi_f(n)n^{it_f};p_K,X)+\frac{1}{\sqrt{p_K}} + \oh_{t_f, Q; X\to \infty}(1)
\end{multline*}
and 
\begin{multline*}
    \mathlarger{\logE}_{ n\leq X} 
    \left|g\left(u^{-1}a_2A W^2 n 
    + u^{-1}m_Q\right)-\chi_g\left(u^{-1}m_Q\right)\left(u^{-1}a_2AW^2n 
    \right)^{it_g}\exp(G(K,X))\right|\\
    \ll \D(g(n),\chi_g(n)n^{it_g};p_K,X)+\frac{1}{\sqrt{p_K}}+ \oh_{t_g, Q; X\to \infty}(1),
\end{multline*}
where 
$$F(K,X)=\sum_{p_K<p\leq X} \frac{f(p)\overline{\chi_f(p)}p^{-it_f}-1}{p}, \:\:\:
G(K,X)=\sum_{p_K<p\leq X} \frac{g(p)\overline{\chi_g(p)}p^{-it_g}-1}{p}$$
and $p_K$ is the $K$-th prime.
Using that $\D(f(n), \chi_f(n)n^{it_f};p_K, \infty), \D(f(n),
\chi_f(n)n^{it_f}; p_K, \infty)$ and $p_K^{-1/2}$ tend to $0$ as $K\to \infty$,
the expression in \eqref{E: double average along progressions} can be rewritten as 
\begin{multline*}
    a_1^{it_f}a_2^{it_g}u^{-it_g} g(u)\exp\left(F(K,X)+G(K,X)\right)
    \Bigg(\mathlarger{\cesE}_{Q\in \Phi_K}f(W)W^{i(t_f+2t_g)}g(A)A^{i(t_g+2t_f)}\\
    \chi_f(\ell_Q)\chi_g\left(u^{-1}m_Q \right) \Bigg)
    \mathlarger{\logE}_{n\leq X}n^{i(t_f+t_g)}+\oh_{K\to\infty}(1)+ \oh_{t_f, Q; X\to \infty}(1)+\oh_{t_g, Q; X\to \infty}(1)
\end{multline*}
As $f,g$ are pretentious, the exponential involving $F(K, X)$ and $G(K, X)$ is bounded from above and below.
If we assume that $t_f+t_g\neq 0$, then by taking limits, first as $X\to \infty$ and then $K\to \infty$, we see that \eqref{E: the negation of part iii in Proposition 4.1} is false, which is a contradiction. Therefore, we must have $t_f=-t_g$ and the last expression can be rewritten as\begin{multline}\label{E: penultimate expression}
    (a_1a_2^{-1})^{it_f}u^{it_f} g(u)
    \exp\left(F(K,X)+G(K,X)\right)\Bigg(\mathlarger{\cesE}_{Q\in \Phi_K}f(W)W^{-it_f}g(A)A^{it_f}\\
    \chi_f(\ell_Q)\chi_g\left(u^{-1}m_Q
    \right)\Bigg)+\oh_{K\to\infty}(1)+ \oh_{t_f, Q; X\to \infty}(1)+\oh_{t_g, Q; X\to \infty}(1)
\end{multline}
We pick $K$ large enough so
that the prime factors of $a_1, a_2$ are all $\leq p_K$ and such that 
$u$ does not completely eliminate any prime factor of $a_2AW^2 $ in the quotient $a_2AW^2/u$. 
 Therefore, \eqref{E: the negation of part iii in Proposition 4.1} and \eqref{E: penultimate expression} imply that \begin{equation}\label{E: final expression}
    \limsup_{K\to\infty}\limsup_{X\to\infty} 
    \left|\mathlarger{\cesE}_{Q\in \Phi_K}f(W)W^{-it_f}g(A)A^{it_f}\\
    \chi_f(\ell_Q)\chi_g(u^{-1}m_Q)\right|>0.
\end{equation}
 
Write $q_f = q_{f,1} q_{f,2} $, where $q_{f,1}$ consists of the prime powers in the factorization of $q_f$ involving primes $p$
so that $p\mid a_1 $. This also implies that $(q_{f,2},p)=1$ for all primes $p\mid a_1$. Then, we can find two 
Dirichlet characters $\chi_{f,1}, \chi_{f,2}$ of modulus 
$q_{f,1}, q_{f,2} $ respectively so that 
$\chi_f= \chi_{f,1} \chi_{f,2}$ (in case 
$q_{f,i}=1$, take $\chi_{f,i}$ to be identically 1).
We carry out the exact same factorization for $q_g$ and $\chi_g$ in order to find 
$q_{g,1}, q_{g,2}, \chi_{g,1} \chi_{g,2}$, where $\chi_g=\chi_{g,1}\chi_{g,2}$ and $(p,q_{g,2})=1$ for all primes $p\mid a_1$.

First we establish several congruence relations between $\ell_Q,\ell_{Qp}$ and $m_Q,m_{Qp}$, from which we obtain some very useful properties for 
$\chi_{f,1}, \chi_{f,2}, \chi_{g,1}, \chi_{g,2}$.
All primes appearing in the statements and arguments below are assumed to 
occur in the factorizations of elements of $\Phi_K$. In Claims \ref{Claim: congruence relations for l_Q} and \ref{Claim: congruence relations for m_Q}
below, if $p$ is a prime and $s\in \Z$, we denote by $\nu_p(s)$ the exponent 
of $p$ in $s$, i.e., $p^{\nu_p(s)}\mid s$ and $p^{\nu_p(s)+1}\nmid s$.

\begin{claim}\label{Claim: congruence relations for l_Q}
If $p\nmid a_1$, then for all except at most $2^{K(K-1)}$ elements 
$Q \in \Phi_K$, we have
\begin{enumerate}
    \item $\chi_{f,1}(\ell_{Qp})=\overline{\chi_{f,1}(p)}\chi_{f,1}(\ell_Q)$,
    \item $\chi_{f,2}(\ell_{Qp})=\chi_{f,2}(\ell_Q)$.
\end{enumerate}
\end{claim}

\begin{proof}[Proof of Claim \ref{Claim: congruence relations for l_Q}]
Observe that for all but $2^{K(K-1)}$ elements $Q\in \Phi_K $, we also have 
$Qp \in \Phi_K$ and these are precisely the elements $Q\in \Phi_K$ that we work with.
Let $Q=AW$, $Qp=A' W'$ and observe that $W'=Wp$, $A'=A$.

Let $q$ be a prime so that $q\mid q_{f,1}$ and $q^{\ell} \parallel A=A'$. 
Since $K$ is assumed to be sufficiently large, we deduce that $\ell > \mu$, 
and hence 
$ {q^{\mu}}\mid A$. Then,  \eqref{E: congruence 2} yields the congruences 
$$a_1 r_Q + b_1 \equiv -a_1b_2a_2^{-1} + b_1 \pmod{ q^{\mu}} \:\: \text{ and } 
\:\: a_1 r_{Qp} + b_1 \equiv -a_1b_2a_2^{-1} + b_1 \pmod{ q^{\mu}}$$
and, thus
$a_1 r_Q + b_1 \equiv a_1 r_{Qp} + b_1  \pmod{ q^{\mu}}$. Using $(W, q)=1$,
we obtain 
$$(a_1 r_Q + b_1)/{W} \equiv (a_1 r_{Qp} + b_1)/{W}
\pmod{ q^{\mu}}.$$ Recall that $\ell_Q = (a_1 r_Q + b_1)/{W}$
and $p\ell_{Qp}= (a_1 r_{Qp} + b_1)/{W} $, so the previous implies that  
$\ell_Q \equiv p\ell_{Qp} \pmod{q^{\mu}}$.
Finally, since $\nu_q(q_{f,1}) \leq 
\mu$, we get that  $q^{\nu_q(q_{f,1})}\mid \ell_Q - p\ell_{Qp}$, and 
since $q$ was arbitrary, this implies that 
$q_{f,1}\mid \ell_Q - p\ell_{Qp}$, i.e., $\ell_Q \equiv p\ell_{Qp} 
\pmod{q_{f,1}}$. This
implies that $\chi_{f,1}(\ell_Q) =\chi_{f,1}(p) \chi_{f,1}(\ell_{Qp})$,
or equivalently, (as $|\chi_{f,1}(p)|=1$)
\begin{equation*}
    \chi_{f,1}(\ell_{Qp})=\overline{\chi_{f,1}(p)}\chi_{f,1}(\ell_Q).
\end{equation*}
and establishes (1) of the Claim. 

Now, let $q$ be a prime such that
that $q| q_{f,2}$ and $q^{\ell}\parallel W$. 
If $q\neq p $, then $q^{\ell}\parallel W, q^{\ell}\parallel Wp$ and using 
\eqref{E: congruence 1} we infer that
$$ a_1 r_Q +b_1 \equiv W \pmod{q^{\ell+1+\nu}} \:\:
\text{ and } \:\:
a_1 r_{Qp} +b_1 \equiv W p \pmod{q^{\ell+1+\nu}}. $$
These two congruences imply that $p(a_1 r_Q +b_1) \equiv a_1 r_{Qp} +b_1 
\pmod{q^{\ell+1+\nu}}$ so 
\begin{equation}\label{eq congr 1}
    q^{\nu+1} \mid  \left(p(a_1 r_Q +b_1)/{q^{\ell}} - (a_1 r_{Qp} +b_1)/q^{\ell}\right) .
\end{equation}
Here, we remind the reader that \eqref{E: congruence 1} implies that if
$p^{\beta}\parallel W$, then $p^{\beta}\parallel a_1r_Q+b_1$.
Since $(W/ q^{\ell}, q^{\nu+1})=1 $ from \eqref{eq congr 1}, we infer that 
$$q^{\nu+1} \mid \left( p(a_1 r_Q +b_1)/{W} - 
(a_1 r_{Qp} +b_1)/{W} \right)= p\ell_Q - p\ell_{Qp}.$$
Recalling \eqref{E: definition of m_1,v_1}, we know that $\nu+1$ is bigger than the 
exponent of $q^{\nu_{q}(q_{2,f})}$ of $q$ in $q_{f,2}$, so from the previous, 
using that $q\neq p$ we get $q^{\nu_{q}(q_{2,f})}\mid (\ell_Q - \ell_{Qp})$.

On the other hand, if $q=p$, then 
$q^{\ell}\parallel W$, $q^{\ell+1}\parallel Wp=W'$ and from \eqref{E: congruence 1} we have 
that
$$ a_1 r_Q +b_1 \equiv W \pmod{q^{\ell+1+\nu}} \:\:
\text{ and } \:\:
a_1 r_{Qp} +b_1 \equiv W p \pmod{q^{\ell+2+\nu}}. $$
The previous implies that $p(a_1 r_Q +b_1) \equiv a_1 r_{Qp} +b_1 
\pmod{q^{\ell+2+\nu}}$ so 
\begin{equation}\label{eq cong 2}
q^{\nu+2} \mid ( p(a_1 r_Q +b_1)/{q^{\ell}} - 
(a_1 r_{Qp} +b_1){q^{\ell}}).
\end{equation}
Since 
$(W/ q^{\ell}, q^{\nu+2})=1 $, we have that 
$q^{\nu+2} \mid ( p(a_1 r_Q +b_1)/{W} - 
(a_1 r_{Qp} +b_1)/{W} )= p\ell_Q - p\ell_{Qp}$.
Recall that $q=p$, so we get that  
$q^{\nu+1} \mid (\ell_Q - \ell_{Qp})$. 
Finally, as $\nu+1$ is bigger than the exponent $\nu_{q}(q_{f,2})$ of $q$ in 
$q_{f,2}$ we get that $q^{\nu_{q}(q_{f,2})}\mid (\ell_Q - \ell_{Qp})$.

Combining the previous, we conclude that for every prime $q$ we have that  
$q^{\nu_{q}(q_{f,2})}\mid   \ell_Q - \ell_{Qp} $, which then implies that 
$q_{f,2}\mid  \ell_Q - \ell_{Qp} $, i.e., 
$\ell_Q \equiv \ell_{Qp} \pmod{q_{f,2}}$.
From that we obtain 
\begin{equation*}
    \chi_{f,2}(\ell_{Qp})=\chi_{f,2}(\ell_Q).
\end{equation*}
which proves (2) and concludes the proof of the Claim \ref{Claim: congruence relations for l_Q}.
\renewcommand\qedsymbol{$\triangle$}
\end{proof}

Next, we establish similar properties for $\chi_{g,1}, \chi_{g,2}$.
\begin{claim}\label{Claim: congruence relations for m_Q}
    If $p\nmid a_1$, then for all except at most $2^{K(K-1)}$ elements 
    $Q \in \Phi_K$, we have
\begin{enumerate}
    \item $\chi_{g,1}(u^{-1}m_{Qp})
=\overline{\chi_{g,1}(p)}\chi_{g,1}(u^{-1}m_Q)$,
    \item $\chi_{g,2}(u^{-1}m_Q)= \chi_{g,2}(u^{-1}m_{Qp})$.
\end{enumerate}
\end{claim}

\begin{proof}[Proof of Claim \ref{Claim: congruence relations for m_Q}]
As in the previous claim, we will work with the elements $Q\in \Phi_K $ for which
$Qp \in \Phi_K$ (the exceptional set has at most $2^{K(K-1)}$ elements).
Let $Q=AW$, $Qp=A' W'$ and again observe that $W'=Wp$, $A=A'$. 
We know that $u=a_1 b_2 -a_2b_1 \mid W $, so if $p\mid u $, then 
$p\nmid a_1 $. Define $$\gamma:=\max\{ \kappa \in \N \colon \text{ there is } p\in 
\P \text{ with } p^{\kappa}\mid u \}.$$  

Let $q\in \P$ so that $q\mid q_{g,1}$ and $q^{\ell} \parallel A=A'$.
Then, using \eqref{E: congruence 2} we infer that
$$r_Q \equiv a_2^{-1}A b_1 W^{-1} -b_2 a_2^{-1} \pmod{q^{\ell+1+\mu}} \:\: 
\text{ and } \:\: 
r_{Qp} \equiv a_2^{-1}A b_1 W^{-1}p^{-1} -b_2 a_2^{-1} \pmod{q^{\ell+1+\mu}}, $$
from which we deduce that 
$a_2r_Q + b_2 \equiv Ab_0 W^{-1} \equiv p(a_2r_{Qp}+b_2) \pmod{q^{\ell+1+\mu}}$,
and hence $(a_2r_Q + b_2)/q^{\ell} \equiv p(a_2r_{Qp}+b_2)/q^{\ell}
\pmod{q^{\mu+1}}.$

Since $(q^{\ell}, A/q^{\ell})=1$, from the previous we get that 
$(a_2r_Q + b_2)/A \equiv p(a_2r_{Qp}+b_2)/{A} 
\pmod{q^{\mu+1}}$, i.e., $m_Q \equiv p m_{Qp} \pmod{q^{\mu+1}}$. Recall that 
$(q, u)=1$, and therefore 
$u^{-1}m_Q\equiv p u^{-1}m_{Qp}
\pmod{q^{\mu+1}}.$
Since $\nu_q(q_{g,1})\leq \mu $ we conclude that
$u^{-1}m_Q \equiv p u^{-1} m_{Qp}
\pmod{q^{\nu_q(q_{g,1})}}.$

In short, for every prime $q$ with $q\mid q_{g,1}$ we have that  
$u^{-1} m_Q \equiv p u^{-1}m_{Qp} \pmod{q^{\nu_{q}(q_{g,1})}}$, and therefore
$u^{-1} m_Q \equiv p u^{-1}m_{Qp} \pmod{q_{g,1}}$. 
From that we get $\chi_{g,1}\left(u^{-1}m_Q\right) 
=\chi_{g,1}(p) \chi_{g,1}\left(u^{-1}m_{Qp}\right)$,
or equivalently (as $|\chi_{g,1}(p)|=1$), 
\begin{equation*}
\chi_{g,1}(u^{-1}m_{Qp})
=\overline{\chi_{g,1}(p)}\chi_{g,1}(u^{-1}m_Q),
\end{equation*}
which establishes (1) above.

Now, take $q\in \P$ so that $q\mid q_{g,2}$ and $q^{\ell}\parallel W$. We can
also assume (by taking $K$ large enough) that the inequality 
$\ell>\gamma+ \nu_q (q_{g,2})$ holds.

If $q\neq p $, then $q^{\ell}\parallel W, q^{\ell}\parallel Wp$ and from 
\eqref{E: congruence 1} we have that

$$ r_Q \equiv -a_1^{-1}b_1  \pmod{q^{\ell}} \:\:
\text{ and } \:\:
r_{Qp} \equiv -a_1^{-1}b_1 \pmod{q^{\ell}},$$
which implies that $a_2r_Q + b_2 \equiv a_2r_{Qp} + b_2  \pmod{q^{\ell}}$. Now, 
since $(q, A)=1$ and $A=A'$, we get that 
$(a_2r_Q + b_2)/A \equiv (a_2r_{Qp} + b_2)/A  \pmod{q^{\ell}}$,
i.e., $m_Q \equiv m_{Qp} \pmod{q^{\ell}}$. From this we get 
$m_Q /q^{\nu_q(u)} \equiv m_{Qp}/q^{\nu_q(u)} \pmod{q^{\ell-\nu_q(u)}}.$
Since $\nu_q(u)\leq \gamma$, we have that 
$\ell- \nu_q(u) > \nu_q (q_{g,2}) $, and therefore 
$m_Q/ q^{\nu_q(u)} \equiv 
m_{Qp}/q^{\nu_q(u)} \pmod{q^{\nu_q{(q_{g,2})}}}.$
Finally, since $(q, u/{q^{\nu_q(u)}})=1$, 
the last relation yields $u^{-1}m_Q \equiv 
u^{-1} m_{Qp} \pmod{q^{\nu_q(q_{g,2})}}$, i.e., 
$q^{\nu_q(q_{g,2})} \mid u^{-1}m_Q - u^{-1} m_{Qp}.$

On the other hand, if $q=p$, then 
$p^{\ell}\parallel W$, $p^{\ell+1}\parallel Wp=W'$ and from \eqref{E: congruence 1} we have
that $$ r_Q \equiv -a_1^{-1}b_1  \pmod{p^{\ell}} \:\:
\text{ and } \:\:
r_{Qp} \equiv -a_1^{-1}b_1 \pmod{p^{\ell}},$$
which implies that $a_2r_Q + b_2 \equiv a_2r_{Qp} + b_2  \pmod{p^{\ell}}$.  
Since $(p, A)=1$ and $A=A'$, we get that 
$(a_2r_Q + b_2)/A \equiv (a_2r_{Qp} + b_2)/A  \pmod{p^{\ell}}$,
i.e., $m_Q \equiv m_{Qp} \pmod{p^{\ell}}$, which implies that 
$m_Q/p^{\nu_p(u)} \equiv m_{Qp}/p^{\nu_p(u)} \pmod{p^{\ell-\nu_p(u)}}.$
Since $\nu_p(u)\leq \gamma$, we have that 
$\ell- \nu_p(u) > \nu_p (q_{g,2}) $, and therefore 
$m_Q/p^{\nu_p(u)} \equiv m_{Qp}/p^{\nu_p(u)} \pmod{p^{\nu_p{(q_{g,2})}}}.$
Finally, since $(p,u/p^{\nu_p(u)})=1$, again we conclude that
$u^{-1}m_Q\equiv u^{-1}m_{Qp} \pmod{p^{\nu_p(q_{g,2})}},$ i.e., 
$p^{\nu_p(q_{g,2})} \mid u^{-1}m_Q - u^{-1}m_{Qp}$.

To summarize, we have shown that for every prime $q$ with $q\mid q_{g,2}$ we 
have that $q^{\nu_{q}(q_{g,2})}\mid u^{-1}m_Q \equiv u^{-1} m_{Qp}$, which 
implies that $q_{g,2}\mid u^{-1}m_Q \equiv u^{-1} m_{Qp}$. From the last relation
we infer that 
\begin{equation*}
    \chi_{g,2}(u^{-1}m_Q)= \chi_{g,2}(u^{-1}m_{Qp}),
\end{equation*}
which establishes (2) and concludes the proof of Claim 
\ref{Claim: congruence relations for m_Q}.
\renewcommand\qedsymbol{$\triangle$}
\end{proof}

Now we return to the proof of the case (iii) in the proposition. If there is 
some prime $p$ with $p\nmid a_1$ and 
$f(p)\neq p^{it_f} \chi_{f,1}(p) \chi_{g,1}(p)$,
then \begin{multline*}
   f(p) p^{-it_f}\overline{ \chi_{f,1}(p)\chi_{g,1}(p)} \:
   \mathlarger{\cesE}_{Q\in \Phi_K}f(W)W^{-it_f}g(A)A^{it_f}
    \chi_f(\ell_Q)\chi_g\left(u^{-1}m_Q\right)\\
    =\mathlarger{\cesE}_{Q\in \Phi_K}f(pW)(pW)^{-it_f}g(A)A^{it_f}
    \chi_{f,1}(\ell_{Qp})\chi_{f,2}(\ell_{Qp})
    \chi_{g,1}\left(u^{-1}m_{Qp}\right)\chi_{g,2}\left(u^{-1}m_{Qp}\right) 
    +\Oh\left(\frac{1}{2^K}\right),
\end{multline*}
where we used Claims 
\ref{Claim: congruence relations for l_Q} and
\ref{Claim: congruence relations for m_Q}, 
and the error term arises from the exceptional sets in Claims 
\ref{Claim: congruence relations for l_Q} and
\ref{Claim: congruence relations for m_Q}. 
Using the fact that $\Phi_K$ is a multiplicative F{\o}lner sequence, 
the last expression can be rewritten as 
\begin{equation*}
    \mathlarger{\cesE}_{Q\in \Phi_K}f(W)(W)^{-it_f}g(A)A^{it_f}
    \chi_{f}(\ell_{Q})
    \chi_{g}\left(u^{-1}m_{Q}\right)+\oh_{K\to\infty}(1).
\end{equation*}
We conclude that 
\begin{equation*}
    \left(f(p) p^{-it_f}\overline{ \chi_{f,1}(p)\chi_{g,1}(p)}-1 \right)\mathlarger{\cesE}_{Q\in \Phi_K}f(W)W^{-it_f}g(A)A^{it_f}
    \chi_f(\ell_Q)\chi_g\left(u^{-1}m_Q\right)
    =\oh_{K\to \infty}(1).
\end{equation*}
Since we assumed that there exists a prime $p\nmid a_1$ for which $f(p)\neq p^{it_f} \chi_{f,1}(p) \chi_{g,1}(p)$, we conclude that \begin{equation*}
    \mathlarger{\cesE}_{Q\in \Phi_K}f(W)W^{-it_f}g(A)A^{it_f}
    \chi_f(\ell_Q)\chi_g\left(u^{-1}m_Q\right)
    =\oh_{K\to\infty}(1).
\end{equation*}which contradicts \eqref{E: the negation of part iii in Proposition 4.1}. We conclude that for every $p\nmid a_1$, we have $$f(p)= p^{it_f} \chi_{f,1}(p) \chi_{g,1}(p),$$ which is the desired conclusion.

Finally, assume that $g=\overline{f}$.
Assume that $f$ is pretentious and the averages in \eqref{E: final expression} are positive, that is 
\begin{equation}\label{E: final expression_f_f}
    \limsup_{K\to\infty}\limsup_{X\to\infty} \left|
    \mathlarger{\cesE}_{Q\in \Phi_K}f(W)W^{-it_f}\overline{f(A)} A^{it_f}
    \chi_f(\ell_Q)\overline{\chi_f\left(u^{-1}m_Q \right)}\right|>0.
\end{equation}

As before, we get that if there is $p\nmid a_1$ so that $f(p)\neq p^{it_f} 
\chi_{f,1} (p) \overline{\chi}_{f,1}(p)$, or equivalently $f(p)\neq p^{it_f}$, then the 
averages in \eqref{E: final expression_f_f} are $0$. Hence, for all 
$p\nmid a_1 $, $f(p)=p^{it_f}$, which implies that $ \D(f(n), n^{it_f})<+\infty$. 
Combining this with the assumption $\D(f(n), \chi_f(n) n^{it_f})<+\infty$ we get that $\D(\chi_f, 1)<+\infty$, so $\chi_f$ is the principal character of conductor $q_f$. 

Recall that
$(\ell_Q, Q)=(Q, u^{-1}m_Q)=1$, so for $K$ sufficiently 
large and $Q\in \Phi_K$, $(\ell_Q, q_f)=(q_f, u^{-1}m_Q)=1$, 
which implies that $\chi_f(\ell_Q)=
\overline{\chi}_f\left(u^{-1}m_Q\right)=1$, since $\chi_f$ is principal. Therefore, combining this with the fact that $f(p)=p^{it_f} $ for $p\nmid a_1$ we get that the 
averages in \eqref{E: final expression_f_f} are equal to 
$\mathlarger{\cesE}_{Q\in \Phi_K}\overline{f}(A) A^{it_f}$. 
If there is some $p\mid a_1$ so that 
$f(p) \neq p^{it_f}$, then Lemma \ref{L: multiplicative averages} implies that 
\begin{equation*}
   (\overline{f(p)} p^{it_f} -1) 
   \mathlarger{\cesE}_{Q\in \Phi_K} \overline{f(A)}A^{it_f}=
    \oh_{K\to \infty}(1),
\end{equation*}
and,  since $\overline{f(p)} p^{it_f} \neq 1 $, we conclude that
$ \mathlarger{\cesE}_{Q\in \Phi_K} \overline{f(A)}A^{it_f} \to 0$ as $K\to\infty$, which is a
contradiction. Therefore, we obtain that for each $p\mid a_1$, 
$f(p)=p^{it_f}$, and combining this with the fact that $f(p)=p^{it_f}$ for all $p\nmid a_1$,  we conclude that 
$f(p)=p^{it_f} $ for all $p\in \P$.
\end{proof}

\section{Concluding the proof of \cref{T: optimal conditions for Bohr recurrence}}\label{section result f-f}
\label{section_return-times}

In this Section we finish the proof of Theorem \ref{T: optimal conditions for Bohr recurrence}, and we also prove Remark \ref{rem lower density}.
Note that in Theorem \ref{T: optimal conditions for Bohr recurrence}, 
the implication (1) $\implies$ (2) is obvious, 
and we proved implication (2) $\implies$ (3) in Section 
\ref{section_optimal-conditions}. 
Therefore, it remains to establish the implication (3) $\implies$ (1),
that is, for $a\in \N$ and $b,d \in \Z$ with 
$(a,b,d)=1$, if $b=d$ or $a\mid bd$, then for all $f\in \M$ and $\e>0$ 
the set $\A(f,\e)$ defined in \eqref{set of ret times} has positive
upper logarithmic
density. Furthermore, if $f$ is pretentious or finitely generated, then 
$\A(f, \e)$ has positive lower logarithmic density. 

In case $b=d$, then for all $f\in \M$ and $\e>0$, $\A(f,\e)=\N$, so the 
desired conclusions follow trivially.
Therefore, from now on we assume without loss of generality that $b\neq d$
and $a\mid bd$. Then our results follow from the following proposition:

\begin{proposition}\label{C: a=c and au mid bd imply recurrence}
    Let $a\in \N$, $b,d\in \Z$ such that $(a,b,d)=1$, $b\neq d$ and $a\mid bd$.
    Then for any $\e>0$ and any completely multiplicative 
    $f\colon \N \to \mathbb{S}^1$, the set  
    \begin{equation*}
    \A(f,\e) =\left\{n\in \N\colon  \left|f(an+b) -f(an+d)\right|<\e\right\}
    \end{equation*}has positive upper logarithmic density. Furthermore, if $f$ is pretentious or finitely generated, it has positive lower logarithmic density.
\end{proposition}

In order to prove Proposition \ref{C: a=c and au mid bd imply recurrence},
we will need to separate our analysis into two cases depending on whether or 
not some power of $f$ is identically equal to $n^{it}$ for some $t\in \R$. 
Then Proposition \ref{C: a=c and au mid bd imply recurrence} follows from
Propositions \ref{P: Main proposition for f-f} and \ref{P: main proposition for f-f in the essentially finite valued case} below:

\begin{proposition}\label{P: Main proposition for f-f}
    Let $a\in \N$, $b,d \in \Z$ such that $(a,b,d)=1$, $b\neq d$ and 
    $a\mid bd$. Let $f\in \M$, and assume that there are no $k\in \N$ and 
    $t\in \R$ so that $f^{k}$ is identically equal to $n^{it}$. Then for any 
    $\e>0$, 
    the set  
    \begin{equation*}
    \A(f,\e) =\left\{n\in \N\colon  \left|f(an+b) -f(an+d)\right|<\e\right\}
    \end{equation*}
    has positive upper logarithmic density. Furthermore, if $f$ is pretentious or finitely generated, it has positive lower logarithmic 
    density.
\end{proposition}

\begin{proposition}\label{P: main proposition for f-f in the essentially finite valued case}
    Let $a\in \N$, $b,d \in \Z$ such that $(a,b,d)=1$, $b\neq d$ and 
    $a\mid bd$. Let $f\in \M$, and assume that there are $k\in \N$ and 
    $t\in \R$ so that $f^{k}$ is identically equal to $n^{it}$.
    Then for any $\e>0$ the set  
    \begin{equation*}
    \A(f, \e) =\left\{n\in \N\colon  \left|f(an+b) -f(an+d)\right|<\e\right\}
    \end{equation*}
    has positive lower logarithmic density.
\end{proposition}

\subsection{An approximation of the indicator function $\1_{[0,\e]\cup(1-\e, 1)}$ on $\T$}\label{fourier approx}
We will need an approximation of
the indicator function of the interval $[0,\e)\cup(1-\e,1)$ by 
trigonometric polynomials. Consider a continuous function $h_{\e}\colon \T \to
[0,1]$ so that $h_{\e}(x) \leq \1_{[0,\e]\cup(1-\e, 1)}(x)$ and 
$\int_{0} ^1 h_{\e}(x) dx =\e$ (for example we can take $h_{\e}$ to be 
an appropriate trapezoid function).

Since $h_{\e}$ is continuous, we know that the Ces\`aro averages of its 
Fourier partial sums converge uniformly to $h_{\e}$. Namely, there exists a
sequence $(c_{\ell,R})_{\ell\in \Z,R\in \N}\subset \C $, such that 
\begin{equation}\label{eq four app 1}
    \norm{h_{\e}(x)-\sum_{|\ell|<R}c_{\ell,R}e(\ell x)}_{\infty}=
    \sup_{x\in \T} \left| h_{\e}(x)-\sum_{|\ell|<R}c_{\ell,R}e(\ell x) \right|
    \to 0
\end{equation}
as $R\to \infty$, where $c_{0,R} = \int_{0}^1 h_{\e}(x) dx=\e$.
Then, we may pick $R$ sufficiently large so that the quantity in 
\eqref{eq four app 1} is smaller than $\e^2$, and if $\e$ is sufficiently small,
for this $R$ we have 
\begin{equation}\label{E: lower bound for h_T}
  h_{\e}(x)\geq  \e+\Re\left( \sum_{1\leq |\ell|<R}c_{\ell, R}e(\ell x)  
  \right)-\e^2 \geq 
  \e^2+\Re\left( \sum_{1\leq |\ell|<R}c_{\ell,R}e(\ell x)  \right)
\end{equation} 
for all $x\in [0,1]$. From now on, we denote $c_{\ell,R}$ simply by $c_{\ell}$ for convenience.

\subsection{The proof of Proposition \ref{P: Main proposition for f-f}}
We are now ready to prove Proposition \ref{P: Main proposition for f-f}.
Let $a\in \N$, $b,d \in \Z$ such that $(a,b,d)=1$, $b\neq d$ and $a\mid bd$. 
Fix a function $f\in \M$ and $\e>0$.
Let $n$ be a natural number not in $\A(f,\e)$, so that 
\begin{equation}
    \left|f(an+b)-f(an+d)\right|\geq \e.
\end{equation} This implies that $$\frac{1}{2\pi}\arg\left(f(an+b)\overline{f(an+d)}\right)\notin \left(-\frac{\e}{2\pi}, \frac{\e}{2\pi}\right), $$ 
where we use the principal branch of the argument function. Redefining $\e$ to absorb the constant $1/(2\pi)$, we have that
$$n\in \A\iff \left\{ \frac{1}{2\pi}\arg\left(f(an+b)\overline{f(an+d)}\right)
\right\}\in [0,\e)\cup(1-\e,1),$$ 
where $\big\{\frac{1}{2\pi}\arg\big(f(an+b)\overline{f(an+d)}\big)
\big\}$ denotes the fractional part of $\frac{1}{2\pi}\arg
\big(f(an+b)\overline{f(an+d)}\big)$.
Our task is to show that \begin{equation}\label{E: set A has positive log density}
    \limsup_{X\to\infty}\: \mathlarger{\logE}_{n\leq X} 
    \1_{\A(f,\e)}(n) >0.
\end{equation}

\subsubsection{A reduction to binary correlations of multiplicative functions}\label{subsec red to mult}
We will now use the approximation carried out in Section
\ref{fourier approx} to reduce the problem to studying 
binary correlations of multiplicative functions. Since 
\begin{multline}\label{a useful inequality}
\1_{\A(f,\e)}(n)=\1_{ [0,\e)\cup(1-\e,1)}
\left(\left\{\frac{1}{2\pi} 
\arg\left(f(an+b)\overline{f(an+d)}\right)\right\}\right)\\
\geq h_{\e}\left(\left\{\frac{1}{2\pi} 
\arg\left(f(an+b)\overline{f(an+d)}\right)\right\}\right),
\end{multline}
it suffices to show that 
\begin{equation}\label{E: function h_T has positive log density}
    \limsup_{X\to\infty}\: \mathlarger{\logE}_{n\leq X} 
   h_{\e}\left(\left\{\frac{1}{2\pi} 
    \arg\left(f(an+b)\overline{f(an+d)}\right)\right\}\right)>0.
\end{equation}

In order to show that \eqref{E: function h_T has positive log density} holds, it suffices to show that there exists $L\in \N$ and $0\leq r_L\leq L-1$ such that 
\begin{equation}\label{E: function h_T has positive log density on progression}
     \limsup_{X\to\infty} \mathlarger{\logE}_{n\leq X} 
     h_{\e} \left(\left\{\frac{1}{2\pi} \arg\left(f(a(Ln+r_L)+b)\overline{f(a(Ln+r_L)+d)}\right)\right\}\right)>0.
\end{equation}This follows from the inequality 
\begin{equation}\label{E: relation between averages along progressions and normal averages}
    \mathlarger{\logE}_{n\leq X} a(Ln+r_L)
    \leq L\left(1+\Oh\left(\frac{\log L}{\log X}\right)\right)
    \mathlarger{\logE}_{n\leq LX+r_L} a(n)
\end{equation}
for any non-negative sequence $a(n)$, which is proved by dropping all the terms on the right-hand side that are not equal to $r_L$ modulo $L$.

Applying \eqref{E: lower bound for h_T}, we see that \eqref{E: function h_T has positive log density on progression} would follow if we could prove that \begin{multline*}
    \limsup_{X\to\infty}\Re\left( \sum_{1\leq |\ell|<R} c_{\ell}\:
    \mathlarger{\logE}_{n\leq X}
    e\left(\ell \left(\frac{1}{2\pi}\arg(f(a(Ln+r_L)+b)\overline{f(a(Ln+r_L)+d)})\right)\right)\right)> -\e^2. 
\end{multline*}

Thus, we will establish our claim if we show that there exist
$L\in \N$ and $0\leq r_L\leq L-1$ such that 
\begin{equation}\label{f-f conclusion 1}
    \limsup_{X\to\infty}\Re\left( \sum_{1\leq |\ell|<R} 
    c_{\ell}\: \mathlarger{\logE}_{n\leq X}     
    f^{\ell}\left(a(Ln+r_L)+b\right)\overline{f^{\ell}
    \left(a(Ln+r_L)+d\right)}\right)> -\e^2.
\end{equation}

\subsubsection{Correlations of multiplicative functions along progressions}
We will now use Proposition \ref{P: binary correlations proposition} to find 
$Q\in \N$ so that \eqref{f-f conclusion 1} holds with 
$L=Q^2$ and $r_L=r_Q$. 

Let $v=(a,b)$ and write $a=va_0$ and $b=vb_0$. Since $a\mid bd$ and 
$a,b,d$ are coprime, we conclude that $(v,d)=1$. Furthermore, we have 
$a_0\mid b_0d\implies a_0\mid d $. Writing $d=d_0a_0$, we conclude that
$(d_0,v)=1$. By factoring out the greatest common divisors and 
redefining $c_{\ell}$ to absorb the terms 
$f^{\ell}(v)\overline{f^{\ell}(a_0)}$, we can rewrite \eqref{f-f conclusion 1}
as
\begin{equation}\label{f-f conclusion 2}
     \limsup_{X\to\infty}\Re\left( \sum_{1\leq|\ell|<R} 
     c_{\ell}\: \mathlarger{\logE}_{n\leq X} 
   f^{\ell}\left(a_0Ln+a_0r_L+b_0\right)
    \overline{f^{\ell}\left(vLn+vr_L+d_0\right)}\right)> -\e^2.
\end{equation}

We let $K$ be a fixed large parameter and define
the set $\Phi_K$ as in \eqref{E: definition of multiplicative Folner} and for each $Q\in \Phi_K$, we let $0\leq r_Q\leq Q^2-1$ be the numbers provided by Proposition \ref{P: binary correlations proposition}. These numbers exist because the coprimality conditions $(a_0,b_0)=(v,d_0)=(a_0,v)=1$ are satisfied.

We will prove that if $K$ is sufficiently large, then for some 
$Q\in \Phi_K$ we have that 
\begin{equation*}
    \limsup_{X\to\infty}\Re\left( \sum_{1\leq|\ell|<R} 
    c_{\ell}\: \mathlarger{\logE}_{n\leq X} 
    f^{\ell}\left(a_0Q^2n+a_0r_Q+b_0\right)\overline{f^{\ell}
    \left(vQ^2n+vr_Q+d_0\right)}\right)> -\e^2
\end{equation*}
In order to prove this, we will show that the double average \begin{equation}\label{E: shit equation that I will only cite one time}
    \mathlarger{\cesE}_{Q\in \Phi_K} 
    \sum_{1\leq|\ell|<R} c_{\ell}\: \mathlarger{\logE}_{n\leq X_m} 
    f^{\ell}\left(a_0Q^2n+a_0r_Q+b_0\right)
    \overline{f^{\ell}\left(vQ^2n+vr_Q+d_0\right)}
\end{equation}
is very close to zero along some subsequence $X_m\to\infty$, for $K$ large
and $m\to\infty$. 
This would imply that there exists at least one $Q\in \Phi_K$ for which \begin{equation*}
      \limsup_{m\to\infty}\Re\left( \sum_{1\leq|\ell|<R} 
      c_{\ell}\: \mathlarger{\logE}_{n\leq X_m} 
      f^{\ell}\left(a_0Q^2n+a_0r_Q+b_0\right)
      \overline{f^{\ell}\left(vQ^2n+vr_Q+d_0\right)}\right)> -\e^2
\end{equation*}
holds.

\subsubsection{Separating pretentious and non-pretentious functions}\label{subsec X_k}

We split $\{-R+1,\ldots,R-1\}\setminus\{0\}$ into two sets $\B_1$ and $\B_2$, where $\B_2$ consists of those integers $\ell\in [-R+1,R-1]\setminus\{0\} $ such that $f^{\ell}$ is pretentious. The sets $\B_1,\B_2$ are symmetric around 0. We let $\B_{1,+}=\{a_1,\ldots, a_{k_1}\}$ and $\B_{2,+}=\{b_1,\ldots,b_{k_2}\}$ denote the sets of positive elements of $\B_1$ and $\B_2$ respectively,
where all elements are written in increasing order.

We consider two cases:

i) Assume that $f^{a_1\ldots a_{k_1}}$ is pretentious. Then, Corollary \ref{C: non-pretentious almost supported on roots of unity is strongly aperiodic} implies that $f^{a_i}$ is strongly aperiodic for any $i\in
\{1,\ldots,k_1\}$. Thus, for any $B>0$ we have 
\begin{equation*}
    \lim\limits_{X\to\infty} \inf_{|t|\leq BX, q\leq B} \D(f^{a_i}(n),\chi(n)n^{it}; 1,X)=+\infty
\end{equation*} for all $a_i\in \A_{+}$.

Furthermore, if $f$ is pretentious, then $f^{a_1\ldots a_{k_1}}$ is pretentious and, thus, we always land in this case.

ii) Assume that $f^{a_1\ldots a_{k_1}}$ is non-pretentious and let $B>0$ be arbitrary. Then, 
Corollary \ref{C: locally pretentious on all scales implies pretentious}
implies that there exists an increasing sequence of integers 
$(X_m)_{m\in\N}$ such that 
$$\inf_{|t|\leq (a_1\ldots 
a_k)BX_m, q\leq B} \D(f^{a_1\ldots a_{k_1}}(n),\chi(n)n^{it}; 
1,X_m)\to+\infty.$$
Using the triangle inequality, we deduce that for any 
$t\in [-BX_m,BX_m]$ and any character $\chi$ of conductor at most $B$, we have 
$$\D(f^{a_i}(n),\chi(n)n^{it}; 1,X_m)\geq \frac{a_1\ldots a_{k_1}}{a_i}
\D(f^{a_1\ldots a_k}(n),\chi^{\frac{a_1\ldots a_{k_1}}{a_i}}(n)
n^{it\frac{a_1\ldots a_{k_1}}{a_i}}; 1,X_m)$$
which implies that for any $a_i$ we have $$\lim\limits_{m\to\infty} 
\inf_{|t|\leq BX_m,q\leq B} \D(f^{a_i}(n),\chi(n)n^{it}; 1,X_m)=+\infty.$$

Furthermore, if $f$ is finitely generated, then $f^{a_1\ldots a_{k_1}}$ is finitely
generated and, thus, strongly aperiodic. In this case, we can take $X_m=m$ for all
$m\in \N$.

To summarize, we deduce (in both cases) that for any $B>0$, there exists a 
sequence of integers $X_m\to\infty$ such that 
\begin{equation}\label{E: large distance along subsequence}
    \lim\limits_{m\to\infty} \inf_{|t|\leq BX_m,q\leq B} \D(f^{a_i}(n),\chi(n)n^{it}; 1,X_m)=+\infty
\end{equation}
for all $a_i \in \B_{1,+}$. Furthermore, when $f$ is pretentious or finitely generated we can take $X_m=m$. From now on, we will work with averages 
along $(X_m)_{m\in \N}$ and the constant $B$ will be specified later.

We write 
\begin{align*}
   E_1(K,X_m)=\sum_{\ell\in \B_1}c_{\ell}\: \mathlarger{\cesE}_{Q\in \Phi_K}\:
     \mathlarger{\logE}_{n\leq X_m} 
     f^{\ell}\left(a_0Q^2n+a_0r_Q+b_0\right)\overline{f^{\ell}\left(vQ^2n+
     vr_Q+d_0\right)}\\
     E_2(K,X_m)=\sum_{\ell\in \B_2}c_{\ell}\: \mathlarger{\cesE}_{Q\in \Phi_K}
     \:\mathlarger{\logE}_{n\leq X_m} 
     f^{\ell}\left(a_0Q^2n+a_0r_Q+b_0\right)
     \overline{f^{\ell}\left(vQ^2n+vr_Q+d_0\right)}
\end{align*}
and we will show that both expressions have real parts very close to zero.

\subsubsection{Concluding the proof}\label{subsec concluding the proof}
For $\ell \in \B_2$, $f^{\ell}$ is pretentious, so by Proposition 
\ref{P: binary correlations proposition} we get that  
\begin{equation*}
  \limsup_{X\to\infty}\left|  \mathlarger{\cesE}_{Q\in \Phi_K}\:
     \mathlarger{\logE}_{n\leq X} f^{\ell}\left(a_0Q^2n+a_0r_Q+b_0\right)
     \overline{f^{\ell}\left(vQ^2n+vr_Q+d_0\right)}\right|=o_{K\to \infty}(1),
\end{equation*}
unless there exists $t\in \R$ such that $f^{\ell}(n)=n^{it}$. 
However, our assumptions on $f$ in the hypothesis of Proposition 
\ref{P: Main proposition for f-f}
exclude this possibility and thus the last asymptotic holds
for 
all $\ell\in \B_2$. Therefore, by picking $K$ sufficiently large, we deduce that 
\begin{equation}\label{E: bound for E_2 _1}
    \left|E_2(K,X_m)\right|\leq \frac{\e^2}{4}
\end{equation}for all $m$ sufficiently large.

For this $K$, we have that for all $m$ sufficiently large,
\begin{multline}\label{E: bound for E_2 _2}
    \left|\mathlarger{\cesE}_{Q\in \Phi_K} \Re\left( \sum_{\ell\in \B_2}
    c_{\ell} \: \mathlarger{\logE}_{n\leq X_m} 
    f^{\ell}\left(a_0Q^2n+a_0r_Q+b_0\right)
    \overline{f^{\ell}\left(vQ^2n+vr_Q+d_0\right)} \right)\right|\\
    =\left|\Re \left(E_2(K,X_m)\right)\right|\leq \frac{\e^2}{4},
\end{multline}
from which we obtain that there is some $Q\in \Phi_K$ so that 
\begin{equation}\label{E: bound for E_2 _3}
\limsup_{m\to \infty} \Re\left( \sum_{\ell\in \B_2}c_{\ell} \:
     \mathlarger{\logE}_{n\leq X_m} 
     f^{\ell}\left(a_0Q^2n+a_0r_Q+b_0\right)
     \overline{f^{\ell}\left(vQ^2n+vr_Q+d_0\right)} \right) \geq  -
     \frac{\e^2}{4}.
\end{equation}

Now, from Theorem \ref{T: Tao theorem} we know that there exists $B'>0$ such 
that if the 1-bounded multiplicative function $g$ satisfies 
\begin{equation*}
     \left|\mathlarger{\logE}_{n\leq X} 
g\left(a_0Q^2n+a_0r_Q+b_0\right)\overline{g\left(vQ^2n+vr_Q+d_0\right)}
     \right|\geq \frac{\e^2}{4\sum_{1\leq |\ell|<R}{|c_{\ell}|}},
\end{equation*}
and $X$ is sufficiently large, then 
\begin{equation*}
    \inf_{|t|\leq B'X,q\leq B'}\D(g(n),\chi(n)n^{it};1,X)\leq B'.
\end{equation*}
Applying this for all instances of $Q\in \Phi_K$, we find a $B'$ that works for all $Q$.

We pick our $B$ above to be equal to $B'$. Our construction of the sequence $(X_m)_{m\in\N}$ implies that for all 
$\ell\in \B_1$, the function $f^{\ell}$ satisfies 
\eqref{E: large distance along subsequence}. Therefore, we have 
\begin{equation*}
     \left|\mathlarger{\logE}_{n\leq X_m}  
     f^{\ell}\left(a_0Q^2n+a_0r_Q+b_0\right)\overline{f^{\ell}
     \left(vQ^2n+vr_Q+d_0\right)}\right|< 
     \frac{\e^2}{4\sum_{1\leq |\ell|<R}{|c_{\ell}|}}
\end{equation*}
for all $\ell\in \B_1$ and all $m$ sufficiently large, 
and thus 
\begin{equation}\label{E: bound for E_1 _1}
    \limsup_{m\to \infty} \left| 
    \Re\left( \sum_{\ell\in \B_1}c_{\ell} \:
     \mathlarger{\logE}_{n\leq X_m}  
     f^{\ell}\left(a_0Q^2n+a_0r_Q+b_0\right)
     \overline{f^{\ell}\left(vQ^2n+vr_Q+d_0\right)} \right)
    \right|< \frac{\e^2}{4}.
\end{equation}

Combining \eqref{E: bound for E_2 _3} and \eqref{E: bound for E_1 _1} we 
obtain that 
\begin{equation*}
\limsup_{m\to \infty} \Re\left( \sum_{1\leq  |\ell| < R}  
     c_{\ell}\: 
     \mathlarger{\logE}_{n\leq X_m} 
     f^{\ell}\left(a_0Q^2n+a_0r_Q+b_0\right)
     \overline{f^{\ell}\left(vQ^2n+vr_Q+d_0\right)} \right) > - \e ^2,
\end{equation*}
i.e., \eqref{f-f conclusion 2} holds with $L=Q^2$ and $r_L =r_Q$. This 
concludes the proof that $\A(f,\e)$ has positive upper logarithmic density. 

\subsubsection{Positive lower logarithmic density}\label{subsec lower density}
If $f$ is pretentious or finitely generated, we have 
seen that we can pick $X_m=m$. Working as in the previous case, we can find 
some $K\in \N$ and $Q\in \Phi_K$ so that 
\begin{equation}\label{E: bound for E_2 _3 _lim}
\limsup_{X\to \infty} \Re\left( \sum_{\ell\in \B_2}c_{\ell} \:
     \mathlarger{\logE}_{n\leq X} 
     f^{\ell}\left(a_0Q^2n+a_0r_Q+b_0\right)
     \overline{f^{\ell}\left(vQ^2n+vr_Q+d_0\right)} \right) \geq -
     \frac{\e^2}{4}.
\end{equation}
For each $\ell\in \B_2$ the limit $$\lim_{X\to \infty} 
 \mathlarger{\logE}_{n\leq X} 
     f^{\ell}\left(a_0Q^2n+a_0r_Q+b_0\right)
     \overline{f^{\ell}\left(vQ^2n+vr_Q+d_0\right)}$$
exists (see \cite[Theorem 1.5]{correl_klu} or \cite[Theorem 2.14]{fra-lem-deL}), 
which implies that in \eqref{E: bound for E_2 _3 _lim} we can
replace the $\limsup_{X \to \infty}$ by $\lim_{X\to \infty}$, so 
\begin{equation}\label{E: bound for E_2 _3 _lim _1}
    \lim_{X\to \infty} \Re\left( \sum_{\ell\in \B_2}c_{\ell}\: 
    \mathlarger{\logE}_{n\leq X}
     \frac{f^{\ell}\left(a_0Q^2n+a_0r_Q+b_0\right)
     \overline{f^{\ell}\left(vQ^2n+vr_Q+d_0\right)}}{n} \right) \geq  -
     \frac{\e^2}{4}.
\end{equation}
Then as in \ref{subsec concluding the proof}, and since $X_m=m$, 
we have that  
\begin{equation}\label{E: bound for E_1 _1 _lim}
    \limsup_{X\to \infty} \left| 
    \Re\left( \sum_{\ell\in \B_1}c_{\ell} \:
     \mathlarger{\logE}_{n\leq X}
     \frac{f^{\ell}\left(a_0Q^2n+a_0r_Q+b_0\right)
     \overline{f^{\ell}\left(vQ^2n+vr_Q+d_0\right)}}{n} \right)
    \right|< \frac{\e^2}{4}
\end{equation}
so from \eqref{E: bound for E_2 _3 _lim _1} and 
\eqref{E: bound for E_1 _1 _lim} we obtain that 
\begin{equation*}
    \liminf_{X\to \infty} 
    \Re\left( \sum_{1\leq |\ell| < R} c_{\ell}\: 
     \mathlarger{\logE}_{n\leq X} 
     f^{\ell}\left(a_0Q^2n+a_0r_Q+b_0\right)
     \overline{f^{\ell}\left(vQ^2n+vr_Q+d_0\right)} \right) > -\e ^2.
\end{equation*}
Combining the previous with \eqref{a useful inequality} we infer that 
\begin{equation}\label{E: lower log density}
    \liminf_{X \to \infty} \mathlarger{\logE_{n\leq Q^2X+r_Q}}{\1_{\A(f,\e)}(n)} >0.
\end{equation}
Then writing $Y\in \N$ as $Y=Q^2 \lfloor Y/Q^2\rfloor  + r_Y$, where 
$r_Y\in \{0,1, \dots, Q^2 -1\}$, we have that
\begin{align*}
    \mathlarger{\logE}_{n\leq Y} {\1_{\A(f,\e)}(n)} & \geq  
    \frac{\log\left({ Q^2 \left( \lfloor Y/Q^2\rfloor -1\right)  + 
    r_Q} \right)}{\log\left({ Q^2 \left( \lfloor Y/Q^2 \rfloor 
    -1\right)  + r_Q} \right) + \log\left(\dfrac{ Q^2 \lfloor Y/Q^2 \rfloor  
    + r_Y}{ Q^2 \left( \lfloor Y/Q^2\rfloor -1\right)
    + r_Q} \right)} \\
    & \hspace*{4.5cm}
    \mathlarger{\logE}_{n\leq Q^2 \left(\lfloor Y/ Q^2 
    \rfloor -1 \right) + r_Q } {\1_{\A} (n)}
\end{align*}
so then taking $\liminf_{Y\to \infty}$, and using that 
$$\lim_{Y\to\infty}\frac{ Q^2 \left\lfloor Y/Q^2\right\rfloor + r_Y}{ Q^2 \left(\left\lfloor Y/Q^2\right\rfloor -1\right) + r_Q}
= 1,$$
as well as \eqref{E: lower log density},
we get that $\A(f,\e)$ has positive lower logarithmic density. 

\subsection{The proof of Proposition \ref{P: main proposition for f-f in the essentially finite valued case}}
Let us now give the proof of Proposition \ref{P: main proposition for f-f in the essentially finite valued case}. Let $a\in \N$, $b,d \in \Z$ such that 
$(a,b,d)=1$, $b\neq d$ and $a\mid bd$. Fix a function $f\in \M$ such that
$f^k(n)=n^{it}$ for some $k\in \N$ and some $t\in \R$. Also fix some 
$\e>0$. We will prove that the set $\A(f, \e)$ appearing in
Proposition \ref{P: main proposition for f-f in the essentially finite valued case} has positive lower logarithmic density. 

Similarly to the proof of Proposition \ref{P: Main proposition for f-f}, 
let $v=(a,b)$. Then for each $n\in \N$, $an+b=v(a_0n +b_0), an+d= a_0(vn+d_0)$,
where $(a_0,b_0)=(v,d_0)=(v,a_0)=1$. 

Let $\ell_0$ be the minimal natural number so that $f^{\ell_0}$ is pretentious,
and let $\chi,t$ be such that $\D(f^{\ell_0}(n), \chi(n) n^{it})<+\infty$. Let 
$\ell_1 \in \N$ be the minimal
natural number such that there is $t'\in\R$ for which $f^{\ell_1}(n) = n^{it'}$ for 
all $n \in \N$. Since $f^{\ell_1}$ is pretentious, Corollary \ref{C: minimal power that makes f pretentious} implies that $\ell_0\mid \ell_1$. Furthermore, since $\D\left(f^{\ell_1}(n),(\chi(n)n^{it})^{\ell_1 /\ell_0}\right)<+\infty$ by the triangle inequality, we conclude that $t'=\ell_1t/\ell_0$ and $\chi^{\ell_1/\ell_0}$ is principal.

Writing $g(n)=f(n)n^{-it/\ell_0}$ and using the approximation 
$$(an+b)^{it/\ell_0} =(an+d)^{it/\ell_0 }+O_{a,b,d}
\left(\frac{|t|}{\ell_0 n}\right)$$ 
one sees that it suffices to prove that
\begin{equation}
\A(g,\e) =\left\{n\in \N\colon  \left|g(an+b) -g(an+d)\right|<\e\right\}
\end{equation}
has positive lower logarithmic density. In addition, we have that $\ell_0$ is the smallest integer such that $g^{\ell_0}$ is pretentious and $\ell_1$ is the smallest integer such that $g^{\ell_1}(n)=n^{it}$ for some $t\in \R$. In fact, we have $g^{\ell_1}$ is identically 1.

We may assume that $\ell_1>1$, since otherwise we would have that $g$ is
identically $1$, from which the result follows trivially.
To prove our result, we will show that the set 
$$\A'= \left\{n\in \N\colon  g(an+b) =g(an+d)\right\}$$
has positive lower logarithmic density.

Assume that this is not the case.  
Since $g(an+b),g(an+d)$ are $\ell_1$-roots of unity, we have that if $n\notin \A'$, then $g(an+b)\overline{g(an+d)}\in \left\{
e(j/\ell_1)\colon j\in \{1, \dots, \ell_1 -1\}\right\}$ and, thus, 
$$\sum_{m=0}^{\ell_1-1} g^m(an+b)\overline{g^m(an+d)}=0.$$
We conclude that for a set of upper logarithmic density 1, we have 
\begin{equation}\label{E: the bad n's sum up to -1}
    -1=\sum_{m=1}^{\ell_1-1} g^m(an+b)\overline{g^m(an+d)}=
\sum_{m=1}^{\ell_1-1} g^m(v)\overline{g^m(a_0)} g^m(a_0n+b_0)\overline{g^m(vn+d_0)}.
\end{equation}

The set $\A'$ for which the last equality does not hold has lower logarithmic density 0. We let $(X_k)_{k\in\N}$ be a sequence of integers for which the limit zero is realized.
For any arithmetic progression $Ln + r_L$, a calculation similar to the one in \eqref{E: relation between averages along progressions and normal averages} implies that we have \begin{equation*}
    \mathlarger{\logE}_{n\leq X_k/L}{\1_{\A'}(Ln+r_L)}=
    o_{k\to \infty}(1).
\end{equation*}
We may shift the averaging from $X_k/L$ to $X_k$ at the cost of an $\Oh(\log L/\log X_k)$ error term.

As in the previous proposition, we let $K$ be sufficiently large and
we take our progression to be $Q^2n +r_Q$, where $Q\in \Phi_K$ 
($\Phi_K$ is defined in \eqref{E: definition of multiplicative Folner}) 
and the $r_Q$'s are chosen as in 
Proposition \ref{P: binary correlations proposition}.
We deduce that \begin{equation*}
     \mathlarger{\logE}_{n\leq X_k}{\1_{\A'}(Q^2n+r_Q)}=o_{k\to \infty}(1)+\Oh\left(\frac{\log Q}{\log X_k}\right).
\end{equation*}
Averaging over all $Q\in \Phi_K$, we conclude that \begin{equation*}
    \mathlarger{\cesE}_{Q\in \Phi_K}  \mathlarger{\logE}_{n\leq X_k}
    \1_{\A'}(Q^2n+r_Q)=o_{K;k\to \infty}(1).
\end{equation*}
Thus, if we average \eqref{E: the bad n's sum up to -1} along the progression $Q^2n+r_Q$ and then average over all $Q\in \Phi_K$, we deduce that the contribution of the $n$ for which $\1_{A'}(Q^2 n+r_Q)=1$ is negligible. Hence,
\begin{multline}\label{E: gamw ton peiraia}
    \sum_{m=1}^{\ell_1-1}g^m(v)\overline{g^m(a_0)} \:
    \mathlarger{\cesE}_{Q\in \Phi_K}  \mathlarger{\logE}_{n\leq X_k}
    g^m(a_0Q^2 n +a_0 r_Q +b_0)
    \overline{g^m(vQ^2 n + vr_Q  +d_0)}= \\
    -1+o_{K;k\to \infty}(1).
\end{multline}
However, we have that, for every $m\in\{1,\ldots,\ell_1-1\}$,\begin{equation}\label{E: final crap that we want to prove}
    \limsup_{k\to+\infty}\left|\mathlarger{\cesE}_{Q\in \Phi_K} \mathlarger{\logE}_{n\leq X_k}{g^m(a_0Q^2 n + a_0 r_Q +b_0)
    \overline{g^m(vQ^2 n+vr_Q +d_0)}}\right|=
    o_{K\to \infty}(1).
\end{equation}

Indeed, we show this by considering two cases. If $g^{m}$ is non-pretentious, 
then, by Corollary 
\ref{C: finite valued non-pretentious is strongly aperiodic}, it is strongly aperiodic since it is finite valued. Then, 
the result follows from a straightforward application of Theorem 
\ref{T: Tao theorem} for each $Q\in \Phi_K$ separately.

On the other hand, in the case that $g^m$ is pretentious, the third part of 
Proposition \ref{P: binary correlations proposition} implies that 
\eqref{E: final crap that we want to prove} holds unless there exists 
$t\in \R$ such that $g^m(n)=n^{it}$ for all $n\in \N$. However, we have 
that $\ell_1$ is the minimal exponent for which 
$g^{\ell_1}(n)=n^{it}$ for some $t\in \R$ and for all $n\in \N$,
and thus this exceptional case cannot occur.

We have proved that \eqref{E: final crap that we want to prove} holds. However,
this contradicts \eqref{E: gamw ton peiraia}. In conclusion, the set $\A'$ has
positive lower logarithmic density and this finishes the proof.

\section{Pairs of multiplicative functions}\label{proofs pairs}
In this section, we prove 
Theorem \ref{a result on pairs}.

\begin{proof}[Proof of Theorem \ref{a result on pairs}]
Let $a\in \N$, $f,g \in \M$ and $\e >0$. We will prove that the set 
$$\A(f,g,\e):=\left\{ n\in \N \colon | f(an+1) - g(an)| < \e \right\}.$$
has positive upper logarithmic density. 

Performing similar manipulations as in \ref{subsec red to mult}
we see that the result would follow if we prove that 
there exist $L\in \N$ and $0\leq r_L\leq L-1$ such that 
\begin{equation*}
    \limsup_{X\to \infty}\Re\left( \sum_{1\leq |\ell|<R} 
    c_{\ell}\: \mathlarger{\logE}_{n\leq X}     
    f^{\ell}\left(a(Ln+r_L)+1\right)\overline{g^{\ell}
    \left(a(Ln+r_L)\right)}\right)> -\e^2.
\end{equation*}
where $R\in \N$ and $c_{\ell}$ are fixed. We may redefine $c_{\ell}$ to 
absorb the term $\overline{g^{\ell}(a)}$, so then it suffices to prove that 
\begin{equation}\label{suffices 1}
    \limsup_{X\to \infty}\Re\left( \sum_{1\leq |\ell|<R} 
    c_{\ell}\: \mathlarger{\logE}_{n\leq X}     
    f^{\ell}\left(a(Ln+r_L)+1\right)\overline{g^{\ell}
    \left(Ln+r_L\right)}\right)> -\e^2.
\end{equation}

\underline{Case 1}: Assume that there is no $\ell \in \N$, $t\in \R$ and 
Dirichlet characters $\chi_f,\chi_g$ with 
conductors $q_f, q_g$ respectively so that $q_f\mid a^{\infty}$,
$f^{\ell}(p)=\chi_f(p)p^{it}$ for all primes $p$ with $p\nmid a$, and 
$\D(g(n),\chi_g(n)n^{it})<+\infty$.  
We split $\{-R+1,\ldots,R-1\}\setminus\{0\}$ into two sets $\B_1$ and 
$\B_2$, where $\B_2$ consists of those integers 
$\ell\in [-R+1,R-1]\setminus\{0\} $ such that $f^{\ell}$ is pretentious,
and as in Section \ref{subsec X_k}, we fix a $B>0$ to be chosen later
and a sequence $X_k\to \infty$ such that 
\begin{equation}\label{E: large distance along subsequence f-g}
    \lim\limits_{k\to \infty} \inf_{|t|\leq BX_k,q\leq B} \D(f^{\ell}(n),\chi(n)n^{it}; 1,X_k)=+\infty
\end{equation}
for all $\ell \in \B_{1}$.

For each $K\in \N$, let $\Phi_K$ be as in \eqref{E: definition of multiplicative Folner}, and for each $Q\in \Phi_K$, let 
$r_Q\in \{0,1, \ldots, Q^2-1\}$ be the numbers provided by Proposition 
\ref{P: binary correlations proposition}. These numbers exist because the 
coprimality conditions $(a,1)=(1,0)=1$ are satisfied.

For $\ell \in \B_2$, we may apply $(ii)$ or $(iii)$ of Proposition 
\ref{P: binary correlations proposition} (depending on whether 
$g^{\ell}$ is aperiodic or pretentious), to get that 
\begin{equation*}
  \limsup_{X\to \infty}\left|  \mathlarger{\cesE}_{Q\in \Phi_K}
     \mathlarger{\logE}_{n\leq X} 
     f^{\ell}\left(aQ^2n+ar_Q+1\right)\overline{g^{\ell}\left(Q^2n+r_Q\right)}
     \right|=o_{K\to \infty}(1).
\end{equation*}
Then, for $K$ sufficiently large, we have that 
\begin{equation*}\label{E: bound for E_2 f-g _2}
 \limsup_{k\to \infty}\left|\mathlarger{\cesE}_{Q\in \Phi_K} 
 \Re\left( \sum_{\ell\in \B_2}c_{\ell} \:
     \mathlarger{\logE}_{n\leq X_k} 
     f^{\ell}\left(aQ^2n+ar_Q+1\right)
     \overline{g^{\ell}\left(Q^2n+r_Q\right)} \right)\right|
\leq \frac{\e^2}{4},
\end{equation*}
from which we obtain that there is some $Q\in \Phi_K$ so that 
\begin{equation}\label{E: bound for E_2 _3 f-g}
\limsup_{k\to \infty} \Re\left( \sum_{\ell\in \B_2}c_{\ell} \:
     \mathlarger{\logE}_{n\leq X_k} 
    {f^{\ell}\left(aQ^2n+ar_Q+1\right)
     \overline{g^{\ell}\left(Q^2n+r_Q\right)}} \right) \geq -
     \frac{\e^2}{4}.
\end{equation}

Now, from Theorem \ref{T: Tao theorem} we know that there exists $B>0$ such 
that if the 1-bounded multiplicative functions $h_1,h_2$ satisfy 
\begin{equation*}
     \left|\mathlarger{\logE}_{n\leq X} 
     h_1\left(aQ^2n+ar_Q+1\right)h_2\left(Q^2n+r_Q\right)
     \right|\geq \frac{\e^2}{4\sum_{1\leq |\ell|<R}{|c_{\ell}|}},
\end{equation*}
and $X$ is sufficiently large, then \begin{equation*}
    \inf_{|t|\leq BX,q\leq B}\D(h_1(n),\chi(n)n^{it};1,X)\leq B \:\:\text{ and }\:\:
    \inf_{|t|\leq BX,q\leq B}\D(h_2(n),\chi(n)n^{it};1,X)\leq B,
\end{equation*}
and, similarly as in the proofs in the \cref{section_return-times}, we can pick $B$ to work for all choices of $Q\in \Phi_K$.

Our construction of the sequence $(X_k)_{k\in\N}$ implies that for all 
$\ell\in \B_1$, the function $f^{\ell}$ satisfies 
\eqref{E: large distance along subsequence f-g}. Therefore, we have 
\begin{equation*}
     \left|\mathlarger{\logE}_{n\leq X_k} 
     f^{\ell}\left(aQ^2n+ar_Q+1\right)\overline{g^{\ell}
     \left(Q^2n+r_Q\right)}\right|< 
     \frac{\e^2}{4\sum_{1\leq |\ell|<R}{|c_{\ell}|}}
\end{equation*}for all $\ell\in \B_1$ and all $m$ sufficiently large, 
and thus 
\begin{equation}\label{E: bound for E_1 _1 f-g}
    \limsup_{k\to \infty} \left| 
    \Re\left( \sum_{\ell\in \B_1}c_{\ell} \:
     \mathlarger{\logE}_{n\leq X_k} 
     f^{\ell}\left(aQ^2n+ar_Q +1\right)
     \overline{g^{\ell}\left(Q^2n+r_Q\right)} \right)
    \right|< \frac{\e^2}{4}.
\end{equation}

Combining \eqref{E: bound for E_2 _3 f-g} and 
\eqref{E: bound for E_1 _1 f-g} we 
obtain that 
\begin{equation*}
\limsup_{k\to \infty} \Re\left( \sum_{1\leq  |\ell| < R}  
     c_{\ell}\: \mathlarger{\logE}_{n\leq X_k} 
     f^{\ell}\left(aQ^2n+r_Q+1\right)
     \overline{g^{\ell}\left(Q^2n+r_Q\right)} \right) > - \e ^2.
\end{equation*}
This proves \eqref{suffices 1} with $L=Q^2$ and $r_L=r_Q$,
and finishes the proof in this case.

\underline{Case 2}: Assume that there is no $\ell \in \N$, $t\in \R$ and Dirichlet character $\chi_f$ so that $g^{\ell}(n)=n^{it}$ and $\D(f^{\ell}(n), \chi_f(n) n^{it})< +\infty$. Then we split $\{-R+1,\ldots,R-1\}\setminus\{0\}$ into
two sets according to whether $g^{\ell}$ is pretentious or 
aperiodic, and then choosing $r_Q '$ according to Remark \ref{R: symmetric}, we get (by the same proof as in Case 1) that 
\begin{equation*}
\limsup_{k\to \infty} \Re\left( \sum_{1\leq  |\ell| < R}  
     c_{\ell}\: \mathlarger{\logE}_{n\leq X_k} 
     f^{\ell}\left(aQ^2n+r_Q'+1\right)
     \overline{g^{\ell}(Q^2n+r_Q')} \right) > - \e ^2.
\end{equation*}
Again, this proves \eqref{suffices 1} with $L=Q^2$ and $r_L=r_Q'$ and the claim follows in this case. 

\underline{Case 3}: It remains to handle the case when there are $\ell_1, 
\ell_2 \in \N$, 
$t_1, t_2 \in \R$, and Dirichlet characters $\chi_1, \chi_2, \chi_3$ so that
$f^{\ell_1}(p)=\chi_1(p)p^{it_1}$ for all
$p\nmid a$, $\D(g^{\ell_1}(n),\chi_2(n) n^{it})<+ \infty$,
$\D(f^{\ell_2}(n),\chi_3(n) n^{it_2})<+\infty$, $g^{\ell_2}(n) = n^{it_2}$ 
and the conductor
$q_1$ of $\chi_1$ satisfies $q_1 \mid a^{\infty}$. 

Let us choose
$\ell_1, \ell_2$ to be the minimal numbers with the previous property. 
Let also $\ell_{f}, \ell_{g}$  be the minimal natural numbers so that 
$f^{\ell_{f}}, g^{\ell_{g}} $ are pretentious, and assume that we have
$f^{\ell_{f}}(n)\sim \chi_f(n) n^{it_f}, g^{\ell_{g}}(n)\sim \chi_g(n) n^{it_g}$.
Since $f^{\ell_1}$ is pretentious, Corollary \ref{C: minimal power that 
makes f pretentious} implies that $\ell_f\mid \ell_1$. Furthermore, since
$f^{\ell_1}(n)\sim\chi_1(n)n^{it_1}$, by the triangle 
inequality we conclude that $t_1=t_f \ell_1/\ell_f$. Analogously, 
we have that $\ell_g\mid \ell_1$ and $t_1=t_g \ell_1/\ell_g$.
Thus, $t_f/\ell_f= t_g/\ell_g$. Consider the 
functions $\wt{f}(n)=f(n) n ^{-i t_f/\ell_{f}}$ and 
$\wt{g}(n)=g(n) n ^{-i t_g/\ell_{g}}$. Since 
$t_f/\ell_{f}= t_g/\ell_{g}$, using the approximation 
$$(an+1)^{it_f/\ell_{f}}=(an)^{it_f/\ell_{f}}+\Oh_{a}\left(\frac{|t_f|}{\ell_{f}n}\right)$$ 
we see that it suffices to prove that
\begin{equation}
\A'(f,g,\e)=\left\{ n\in \N \colon | \wt{f}(an+1) - \wt{g}(an)| 
< \e \right\}.
\end{equation}

From the above, we have that 
$\ell_{f}, \ell_{g}$  are the minimal natural numbers so that 
$\wt{f}^{\ell_{f}}, \wt{g}^{\ell_{g}} $ are pretentious, 
$\wt{f}^{\ell_{f}}\sim \chi_f , 
\wt{g}^{\ell_{g}}\sim \chi_g $,
$\wt{f}^{\ell_1}(p)= \chi_1(p)$ for all $p\nmid a$, 
$\wt{g}^{\ell_1 }\sim \chi_2 $, 
$\wt{f}^{\ell_2}\sim \chi_3 $ and 
$\wt{g}^{\ell_2}\equiv 1$. In addition, $\ell_1, \ell_2$ are the minimal 
natural numbers with this property. Observe also that 
for all $m\in \Z$, all Dirichlet characters $\chi$ and all $t\neq 0$, we have $\wt{f}^m(n)\not \sim \chi(n) n^{it}$ and $\wt{g}^m(n)\not \sim  \chi(n) n^{it}$.

We say that a set $T\subset \P$ is \emph{thin} is $\sum_{p\in T} \frac{1}{p} <+\infty$. From the fact that $\wt{f}^{\ell_1}(p) = \chi_1(p)$
for all $p\nmid a$, we have that $\wt{f}$ is finitely generated. Since
$\wt{f}^{\ell_2}\sim \chi_3$, there is a thin set $T$ of primes
so that for all $p\notin T$, $\wt{f}^{\ell_2}(p) = \chi_3(p)$.\footnote{Since $\wt{f}$ and $\chi_3$ take values in finite sets, the distance can only be finite if the set $\{\wt{f}(p)\neq \chi_3(p)\} $ is thin.} Let 
$E=\big(\bigcup_{p\in T} p\N \big)^{c}$. Then, we easily check $\1_E$ is completely 
multiplicative (but not in $\mathcal{M}$), it is $1$-pretentious (since $\1_{E}(p)=1$ for all primes outside the thin set $T$), and $\wt{f}^{\ell_2}(n)=\chi_3(n)$ for all 
$n\in E$. 

If $\ell_2=1$, then $\wt{g}\equiv 1$, and we see that it suffices to prove that
\begin{equation}
    \A''=\left\{ n\in \N \colon | \wt{f}(an+1)\1_{E}(an+1) - 1| 
< \e \right\}
\end{equation}
has positive upper density. Since $\wt{f}\1_{E} \sim \chi_3$, then for $K$ 
sufficiently 
large (so that $q_3 \mid Q$ for $Q \in \Phi_K$), Lemma \ref{L: concentration estimate} implies that 
\begin{equation}\label{A''' has positive density}
    \mathlarger{\logE}_{n\leq X} |\wt{f}(aQn+1)\1_{E}(aQn+1) - \exp(F(K,X))|
    \ll \mathbb{D}(\wt{f} \1_E,\chi_3; p_K, X)+\frac{1}{\sqrt{p_K}}
    +\oh_{a,Q; X\to \infty}
\end{equation}
where, from the definition of $Q$, $$F(K,X)=\sum_{p_K < p \leq X} 
\frac{\wt{f}(p)\1_E(p)\overline{\chi_3(p)} -1}{p}$$ only depends on $K$. 
Since $\wt{f} \1_E \sim \chi_3$, the term involving the exponential in 
\eqref{A''' has positive density} is arbitrarily close to $0$ for 
$X$ and $K$ sufficiently large. Then, taking $\limsup_{X\to \infty}$ in 
\eqref{A''' has positive density} and $K$ sufficiently large, we have 
\begin{equation*}
    \limsup_{X\to \infty} 
    \mathlarger{\logE}_{n\leq X} |\wt{f}(aQn+1)\1_{E}(aQn+1) - 1| < \e
\end{equation*}for all $Q\in \Phi_K$.
For any such $Q$, we have that the set $\A''/Q$ has positive lower 
logarithmic density, i.e., 
$$\liminf_{X\to \infty} \mathlarger{\logE}_{n\leq X} \1_{\A''}(Qn) > 0,$$
and similarly as in the last step in the proof of Proposition 
\ref{P: Main proposition for f-f}, we infer that 
$$\liminf_{X\to \infty} \mathlarger{\logE}_{n\leq X}
{\1_{\A''} (n)} >0.$$ Therefore, $\A''$ has positive lower 
logarithmic density, so it also has positive upper logarithmic density.

Now, if $\ell_2 >1$, then since $|g(n)|=1$ for all $n$, 
it suffices to prove that $$\A'' :=\{n\in \N\colon |\wt{f}(an+1) \1_E (an+1)-
\wt{g}(an)| <\varepsilon\} $$ 
has positive upper logarithmic density, which
would follow if we prove that 
\begin{equation}\label{suffices 3}
    \A''' :=\{n\in \N\colon \wt{f}(an+1) \1_E (an+1)=\wt{g}(an) \}
\end{equation}
has positive upper logarithmic density. Assume that this is not the case, so 
$\A'''$ has zero logarithmic density. 
A calculation similar to the one in \eqref{E: relation between averages along progressions and normal averages} yields that for any $Q\in \N$ 
we have 
\begin{equation*}
    \mathlarger{\logE}_{n\leq X/Q} \1_{\A'''}(Qn) =
    o_{X\to\infty}(1).
\end{equation*}
We may shift the averaging from $X/Q$ to $X$ at the cost of an 
$\Oh(\log Q/\log X)$ error term. Hence, we have then that for all $Q\in \N$, 
$\A'''/Q$ also has zero logarithmic density.
Let $K$ be sufficiently large so that the modulus $q_3$ of $\chi_3$ divides
$Q$ for all $Q\in \Phi_K$. Then for $n\notin A'''/Q$,
\begin{itemize}
    \item if $aQn+1\in E$, then $\wt{g}^{\ell_2}(aQn)=1$, and 
    $\wt{f}^{\ell_2}(aQn+1)= 
    \chi_3(aQn+1)=1$, therefore $\wt{f}(aQn+1)\1_E(aQn+1) $, $\wt{g}(aQn)$ are
    $\ell_2$-roots of unity, and since $Qn\notin \A'''$, they are 
    distinct, so 
    $$\sum_{m=1} ^{\ell_2 } \wt{f}^m(aQn+1)  \1_E ^m(aQn+1)
    \overline{\wt{g}^m(aQn)} = 0$$ 
    \item if $aQn+1\notin E$, then $\1_E(aQn+1)=0$, so again 
    $$\sum_{m=1} ^{\ell_2 } \wt{f}^m(aQn+1)  \1_E ^m(aQn+1)
    \overline{\wt{g}^m(aQn)} = 0.$$
\end{itemize}

Since $\N \setminus \left(\A'''/Q\right)$ has logarithmic density $1$, we
conclude that for $K$ sufficiently large and $Q\in \Phi_K$,
\begin{equation*}
\lim_{X\to \infty} \sum_{m=1}^{\ell_2} \mathlarger{\logE}_{n\leq X} 
\wt{f}^m(aQn+1)\1_E(aQn+1)\overline{\wt{g}^m(aQn)}=0.
\end{equation*}
    Hence, for $K$ sufficiently large, 
\begin{equation}\label{E: f-g l eq}
\lim_{X\to \infty} \mathlarger{\cesE}_{Q\in \Phi_K}\sum_{m=1}^{\ell_2} 
\mathlarger{\logE}_{n\leq X}
\wt{f}^m(aQn+1)\1_E(aQn+1)\overline{\wt{g}^m(aQn)}=0.
\end{equation}

Fix a $K$ such that \eqref{E: f-g l eq} holds and such that 
for $1\leq \ell \leq \ell_2 -1$, 
\begin{equation}\label{E: multiplicative averages of g^l tilde are small}
    \left|\mathlarger{\cesE}_{Q\in \Phi_K} \wt{g}^{\ell}(Q)\right|\leq \e.
\end{equation}
We can do this because $g^{\ell}\not \equiv 1$. Assume that $K$ is so large
that for all $\ell \in \{1, \dots, \ell_2 -1\}$, 
if $\wt{f}^{\ell} \sim \wt{\chi}_{\ell}$, 
(in which case we 
also have that $\wt{f}^{\ell} \1_E ^{\ell} \sim \wt{\chi}_{\ell}$), then the conductor of $\wt{\chi}_{\ell}$ divides 
$aQ$ for all $Q\in \Phi_K$ and also 
$$\sum_{p>p_K} 
\frac{\wt{f}^{\ell} \1_E^{\ell}(p)\overline{\wt{\chi}_{\ell}
(p)} -1}{p} \ll \e \:\:\text{ and }\:\: \D(\wt{f}^{\ell} \1_E ^{\ell}, \wt{\chi}_{\ell} ; p_K, +\infty) + \frac{1}{\sqrt{p_K}} \ll \e.$$
Also, we assume that 
$$\sum_{p>p_K} 
\frac{\wt{f}^{\ell_2} \1_E^{\ell_2} (p) \overline{\chi_3 (p)} -1}{p}\ll\e \:\:\text{ and }\:\: 
\D(\wt{f}^{\ell_2} \1_E ^{\ell_2}, \chi_3 ; p_K, +\infty) + \frac{1}{\sqrt{p_K}} \ll \e.$$

For the finite collection of $Q\in \Phi_K$, we use Theorem 
\ref{T: Tao theorem} to choose a $B>0$ so that if 
$h_1,h_2\colon\N\to\C$ are $1$-bounded multiplicative functions, $X\geq B$ and
$$\inf_{|t|\leq BX, q\leq B}\mathbb{D}(h_1(n),\chi(n)n^{it}; 1,X)\geq B \:\:\text{ or }\:\: \inf_{|t|\leq BX, q\leq B}\mathbb{D}(h_2(n),\chi(n)n^{it}; 1,X)\geq B,$$
then
$$\left|\mathlarger{\logE}_{n\leq X}
h_1(a_1n+b_1)h_2(a_2n+b_2)\right|\leq\varepsilon.$$

Similarly as in \ref{subsec X_k}, we split $\{1,\ldots,\ell_2 \}$ into two 
sets $\B_1$ and $\B_2$, where $\B_2$ consists of those integers 
$\ell\in \{1,\ldots,\ell_2 \} $ such that $\wt{g}^{\ell}$ is pretentious,
and as in subsection \ref{subsec X_k}, we fix a $B>0$ to be chosen later
and a sequence $X_k\to \infty$ such that 
\begin{equation}\label{E: large distance along subsequence f-g _2}
    \lim\limits_{k\to \infty} \inf_{|t|\leq BX_k,q\leq B} \D(\wt{g}^{\ell}(n),\chi(n)n^{it}; 1,X_k)=+\infty
\end{equation}
for all $\ell \in \B_{1}$.

Now for $1\leq \ell \leq \ell_2-1$: 
\begin{itemize}
    \item If $\wt{g}^{\ell}$ is aperiodic, then $\wt{g}^{{\ell}}$ satisfies 
    \eqref{E: large distance along subsequence f-g _2}, so
    \begin{equation}\label{E: f-g l eq _1}
    \limsup_{k\to \infty}
    \left| \mathlarger{\cesE}_{Q\in \Phi_K} \mathlarger{\logE}_{n\leq X_k}
    \wt{f}^{\ell} (aQn+1) \1_E ^{\ell}(aQn+1)  
    \overline{\wt{g}^{\ell}(aQn)} \right| \leq \e.
    \end{equation}

    \item If $\wt{g}^{\ell} $ is pretentious and $\wt{f}^{\ell}$
    is aperiodic (in which case also $\wt{f}^{\ell} \1_E ^{\ell} $ is 
    aperiodic), then from (ii) of 
    Proposition \ref{P: binary correlations proposition} we have
    \begin{equation*}
    \sup_{Q\in \Phi_K} \limsup_{X\to \infty}
    \left| \mathlarger{\logE}_{n\leq X}
    \wt{f}^{\ell} (aQn+1) \1_E ^{\ell}(aQn+1)  \overline{\wt{g}^{\ell}
    (aQn) } \right|= \oh_{K\to \infty}(1),
    \end{equation*}
    so if $K$ is sufficiently large, then 
    \begin{equation}\label{E: f-g l eq _2}
    \limsup_{X\to \infty}
    \left| \mathlarger{\cesE}_{Q\in \Phi_K} \mathlarger{\logE}_{n\leq X}
    \wt{f}^{\ell} (aQn+1) \1_E ^{\ell}(aQn+1)  \overline{\wt{g}^{\ell}
    (aQn) } \right| \ll \e.
    \end{equation}
    
    \item Finally, if both $\wt{f}^{\ell}$ and $\wt{g}^{\ell}$ are pretentious, 
    $\wt{f}^{\ell} \sim \wt{\chi}_{\ell}\:(\:= \chi_f ^{\ell/\ell_{f}})$, 
    then we 
    also have that $\wt{f}^{\ell} \1_E ^{\ell} \sim \wt{\chi}_{\ell}$, 
    and since we assumed that the conductor of $\wt{\chi}_\ell$ divides $aQ$, we 
    apply Lemma \ref{L: concentration estimate} to get that 
    \begin{align*}
    \mathlarger{\logE}_{n\leq X} \wt{f}^{\ell} (aQn+1) \1_E ^{\ell}(aQn+1)  \overline{\wt{g}^{\ell}(aQn) }
    = \overline{\wt{g}^{\ell} (Q)} \exp\left(\sum_{p_K<p\leq X} 
    \frac{\wt{f}^{\ell} \1_E^{\ell} (p) \overline{\wt{\chi}_{\ell}(p)} - 1}{p}\right) 
    \mathlarger{\logE}_{n\leq X}
    \overline{\wt{g}^{\ell}(an)} & \\
    + \Oh\left( \D(\wt{f}^{\ell} \1_E ^{\ell}, \wt{\chi}_{\ell} ; p_K, X) + \frac{1}{\sqrt{p_K}}\right) + \oh_{a,Q; X\to \infty}(1). &
    \end{align*}
    Thus, taking $\limsup_{X\to \infty}$, averaging over $Q\in \Phi_K$ and 
    using our assumptions on $K$,\footnote{We use \eqref{E: multiplicative averages of g^l tilde are small} and the fact that all error terms are smaller than $\e$.} we obtain that 
    \begin{equation}\label{E: f-g l eq _3}
    \limsup_{X\to \infty}
    \left| \mathlarger{\cesE}_{Q\in \Phi_K} \mathlarger{\logE}_{n\leq X}
    \wt{f}^{\ell} (aQn+1) \1_E ^{\ell}(aQn+1)  
    \overline{\wt{g}^{\ell}(aQn)}
    \right| \ll \e.
    \end{equation}
\end{itemize}
Combining \eqref{E: f-g l eq}, \eqref{E: f-g l eq _1}, 
\eqref{E: f-g l eq _2} and \eqref{E: f-g l eq _3} we get that 
\begin{equation}\label{E: f-g l eq _4}
\limsup_{k\to \infty}
\left| \mathlarger{\cesE}_{Q\in \Phi_K} \mathlarger{\logE}_{n\leq X_k}  
\wt{f}^{\ell_2} (aQn+1) \1_E ^{\ell_2}(aQn+1)  
\overline{\wt{g}^{\ell_2}(aQn) }\right|
\ll \e.
\end{equation}

On the other hand, $\wt{g}^{\ell_2}\equiv 1$, so 
applying Lemma \ref{L: concentration estimate} again we get that 
\begin{multline*}
\mathlarger{\logE}_{n\leq X_k}   
\wt{f}^{\ell_2} (aQn+1) \1_E ^{\ell_2}(aQn+1)  
\overline{\wt{g}^{\ell_2}(aQn)}
=
\mathlarger{\logE}_{n\leq X_k}   
\wt{f}^{\ell_2} (aQn+1) \1_E ^{\ell_2}(aQn+1) \\
= \exp\left(\sum_{p_K<p\leq X_k} 
\frac{\wt{f}^{\ell_2} \1_E^{\ell_2} (p) \overline{\chi_3 (p)} -1}
{p}\right)
+ \Oh\left( \D(\wt{f}^{\ell_2} \1_E ^{\ell_2}, \chi_3 ; p_K, X_k) + \frac{1}
{\sqrt{p_K}}\right)+ \oh_{a,Q; X\to \infty}(1).
\end{multline*}
Taking $\limsup_{k\to \infty}$ and using again the fact that $K$ is large, we deduce that
\begin{equation*}
\limsup_{k\to \infty} \left|
\mathlarger{\cesE}_{Q\in \Phi_K} \mathlarger{\logE}_{n\leq X_k} 
\wt{f}^{\ell_2} (aQn+1) \1_E ^{\ell_2}(aQn+1)  
\overline{\wt{g}^{\ell_2}(aQn) }\right|=1-\Oh(\e).
\end{equation*}
This contradicts \eqref{E: f-g l eq _4} for $\e$ sufficiently small, 
and concludes the proof of the theorem.
\end{proof}

\begin{remark}
We remark here that if $f,g$ are pretentious or finitely generated, then in 
a similar way as we do in \ref{subsec lower density}, we could prove 
that the set $\A(f,g,\e)$ in Theorem 
\ref{a result on pairs} has positive lower logarithmic density. 
\end{remark}

From Theorem \ref{a result on pairs} one gets that for any $f,g \in \M$, 
$$ \liminf_{n\to \infty} |f(n+1)-g(n)|=0.$$
It is natural to think that one can generalize this to get that for all 
$f,g \in \M$ and $k\in \Z$ 
$$ \liminf_{n\to \infty} |f(n+k)-g(n)|=0.$$
However, this is not true. We provide a counterexample here for the 
case $k=2$.

\begin{counterexample*}
Let $\chi$ be the Dirichlet character modulo $4$ defined by $\chi(1)=1, 
\chi(3)=-1$. Take $\theta_1, \theta_2$ rationals in $(0,1)$ so that 
for all $\ell \in \N_0$, $\ell \theta_1 \pmod{1} \neq \theta_2$, 
$\ell \theta_1 + \frac{1}{2} \pmod{1} \neq \theta_2$, 
$\ell \theta_2 \pmod{1} \neq \theta_1$ and 
$\ell \theta_2 + \frac{1}{2} \pmod{1}  \neq \theta_1$ (it is not too difficult to 
see that, for example, if $\theta_1={1}/{p_1}$, $\theta_2={1}/{p_2}$ 
for different primes $p_1, p_2$, then the previous conditions are satisfied).
Let 
$$\eta:=\min \{ |e(\ell \theta_1)-e(\theta_2)|, 
|e(\ell \theta_1 + 1/2)-e(\theta_2)|, 
|e(\ell \theta_2)-e(\theta_1)|, 
|e(\ell \theta_2 + 1/2)-e(\theta_1)| \colon 
\ell \in \N_0\}.$$

From the conditions on $\theta_1, \theta_2$ we get that $\eta>0$. 
Now, consider the modified Dirichlet characters $f,g$ modulo $4$ defined for 
$n=2^{\ell} k$, $2\nmid k, \ell \geq 0$, by $f(n)=e(\ell \theta_1) \chi(k)$,
$g(n)=e(\ell \theta_2) \chi(k)$. Let $n\in \N$. If $n\equiv 1 \pmod{4}$, 
then $n+2 \equiv 3 \pmod{4}$, so $f(n+2)=-1, g(n)=1$ and $|f(n+2)-g(n)|=2$. 
If $n\equiv 3 \pmod{4}$, then $n+2 \equiv 1 \pmod{4}$, so $f(n+2)=1, 
g(n)=-1$ and $|f(n+2)-g(n)|=2$.

Now, if $n\equiv 0 \pmod{4}$, $n=4j$, then $n+2 = 4j+2=2(2j+1)$, 
so $f(n+2)=e(\theta_1)$ or $f(n+2)=e(\theta_1+ \frac{1}{2})$, and 
$g(n)\in \{e(\ell \theta_2), e(\ell \theta_2 + \frac{1}{2})\colon \ell \in \N\}$,
so by assumption $|f(n+2)-g(n)|\geq \eta>0 $. 

Finally, if $n\equiv 2 \pmod{4}$, $n=4j+2$, then $n+2=4(j+1)$, 
so $f(n+2)\in \{e(\ell \theta_1), e(\ell \theta_1 + \frac{1}{2})\colon \ell \in 
\N\}$ and $g(n)=e(\theta_2)$ or $g(n)=e(\theta_2+ \frac{1}{2})$,
again $|f(n+2)-g(n)|\geq \eta$. Combining all the previous we get that 
$\liminf_{n\to \infty} | f(n+2)-g(n)| \geq \eta >0. $
\end{counterexample*}

\section{Recurrence for finitely generated multiplicative systems}\label{section_mult-rec-for-fg-systems}

In this section, we deal with the multiplicative recurrence properties of the set $R=\{(an+b)/(an+d)\colon n\in \N\}$. In 
particular, we prove Theorem \ref{T: multiplicative recurrence for finitely generated systems}, which asserts that if $b=d$ or $a\mid bd$, then $R$ is a set
of recurrence for finitely generated systems. 
Let $\Mfg$ denote the subset of $\M$ consisting of all finitely generated 
functions in $\M$.

\begin{lemma}\label{mfg is Borel}
    The set $\Mfg$ is Borel measurable.
\end{lemma}
\begin{proof}
    For each $n\in \N$, let $\mathcal{C}_n$ denote the set of functions 
    $f\in \M$ that take at most $n$ different values on $\P$. Then 
    $\Mfg=\bigcup_{n\in \N} \mathcal{C}_n$, so to prove that $\Mfg$ is 
    Borel, it suffices to prove that $\mathcal{C}_n$ is Borel for all 
    $n\in \N$. 

    Given $n\in \N$, consider the function $H_n\colon \M \to [0,+\infty)$ 
    defined by
    $$H_n (f)=\inf_{\xi_1, \ldots, \xi_n \in [0,1)\cap\Q} \sum_{p\in \P} 
    \frac{|(f(p)-e(\xi_1))\cdots(f(p)-e(\xi_n))|}{p^{2}}.$$

    Then, $H_n$ is Borel measurable as a countable infimum of measurable functions.\footnote{Each series in this expression is a Borel measurable function, as it is a pointwise limit of continuous functions.} For 
    $f\in \M$, we have that 
    $f\in \mathcal{C}_n$ if and only if $H_n(f)=0$. Therefore, 
    $\mathcal{C}_n=\{f\in \M\colon H_n(f)=0\}$, so it is Borel 
    measurable. This concludes the proof.     
\end{proof}

Let $\Mfgp$, $\Mfgs$ denote the subsets of $\Mfg$ that consist respectively
of the pretentious and aperiodic functions in $\Mfg$.
From Lemma \ref{mfg is Borel} and {\cite[Lemma 3.6]{Fra-Klu-Mor}} we get that
the sets $\Mfgp, \Mfgs \subset \M$ are Borel measurable.
Recall that 
by Proposition \ref{P: finitely generated non-pretentious is strongly aperiodic},
every $f\in \Mfgs$ is strongly aperiodic.

Let $\left(X,\X,\m,(T_n)_{n\in \N}\right)$ be a finitely 
generated multiplicative system and let 
$A\in \X$ with $\delta:=\mu(A)>0$.
Using the Bochner-Herglotz theorem, we get that there is a finite Borel 
measure $\sigma$ on $\M$ with $\sigma(\{1\})\geq \delta^2$
such that for any $r,s\in\N$,
\begin{equation*}
    \mu(T_r^{-1}A\cap T_s^{-1}A) = \int_{\M} f(r)\overline{f(s)} ~d\sigma(f).
\end{equation*}
The measure $\sigma$ is assigned to the set $A$ and it is the {\em spectral measure} of the function $\1_A\in L^2(X,\mu)$. 
By \cite[Lemma 2.6]{Charamaras-multiplicative}, since $(X,\X,\mu,(T_n)_{n\in \N})$ is
finitely generated, the measure $\sigma$ is supported on $\Mfg$, thus the 
above equation becomes
\begin{equation}\label{eqn_spectral_meas}
    \mu(T_r^{-1}A\cap T_s^{-1}A) = \int_{\Mfg} f(r)\overline{f(s)} ~d\sigma(f).
\end{equation}
Then Theorem 
\ref{T: multiplicative recurrence for finitely generated systems} 
follows from the next theorem.
\begin{theorem}\label{thm_rec_fg_reformulation}
    Let $a\in\N$, $b,c \in \Z$ with $(a,b,d)=1$ satisfy $b=d$ or 
    $a\mid bd$. Let $\delta>0$ and $\sigma$ be a finite Borel measure on $\Mfg$ such that $\sigma(\{1\})\geq \delta^2$ and
    \begin{equation}\label{eqn_rec_fg_assumption}
        \int_{\Mfg} f(r)\overline{f(s)}~d\sigma(f) \geq 0
        \qquad \text{for every}~ r,s\in\N.
    \end{equation}
    Then there is $n\in\N$ such that 
    \begin{equation}\label{eqn_rec_fg_conclusion}
        \int_{\Mfg} f(an+b)\overline{f(an+d)}~d\sigma(f) \geq 
        \frac{\delta^2}{2}.
    \end{equation}
    Moreover, the set of $n\in\N$ satisfying 
    \eqref{eqn_rec_fg_conclusion} has positive lower density.
\end{theorem}

\begin{proof}[Proof of Theorem \ref{thm_rec_fg_reformulation}]
If $b=d$, then for all $f\in \Mfg$ and $n\in \N$,
$f(an+b)\overline{f(an+d)}=1$, so
$$\int_{\Mfg} f(an+b)\overline{f(an+d)}~d\sigma(f) \geq \delta ^2,$$
and the result follows trivially. Hence, from now on, we assume that 
$b\neq d$ and $a\mid bd$.

Similarly to the proof of Proposition \ref{P: Main proposition for f-f}, we
take $v=(a,b)$ and write $an+b=v(a_0n +b_0)$, $an+d= a_0(vn+d_0)$, where 
$(a_0,b_0)=(v,d_0)=(v,a_0)=1$, so the integral in 
\eqref{eqn_rec_fg_conclusion} becomes 
\begin{equation}\label{eqn_rec_fg_conclusion _1}
\int_{\Mfg} f(v) \overline{f(a_0)}f(a_0n+b_0)\overline{f(vn+d_0)}~d\sigma(f).
\end{equation}

Taking real parts, and using that $\sigma(\{1\})\geq \delta^2$, we 
see that for each $n$, the expression in 
\eqref{eqn_rec_fg_conclusion _1} is greater or equal than  
\begin{align}\label{eqn_rec_fg_conclusion _2}
\delta ^2 & + \int_{\Mfgp\setminus \{1\}} \Re\left( f(v) 
\overline{f(a_0)}f(a_0n+b_0)\overline{f(vn+d_0)} \right)~d\sigma(f) \\
&+
 \int_{\Mfgs} \Re\left( f(v) 
\overline{f(a_0)}f(a_0n+b_0)\overline{f(vn+d_0)}\right) ~d\sigma(f).
\nonumber
\end{align}

Observe that for each $f\in \Mfgp\setminus \{1\}$, we have there is no $t\in\R$ such that $f$ is identically equal to $n^{it}$,
so from Proposition \ref{P: binary correlations proposition} we know that 
we can find $r_Q$'s so that 
\begin{equation}\label{eqn_rec_fg_conclusion _3}
    \lim\limits_{K\to \infty}\  \limsup\limits_{X\to \infty}\     \left|\mathlarger{\cesE}_{Q\in \Phi_K} \mathlarger{\logE}_{n\leq X} f(a_0Q^2n+a_0r_Q+b_0)\overline{f(vQ^2n+vr_Q+d_0)}   \right|=0.
\end{equation}
Again from {\cite[Theorem 2.14]{fra-lem-deL}} and {\cite[Theorem 1.5]{correl_klu}} we know that for each $f\in \Mfgp$
and $Q\in \N$, the limit 
$$\lim_{X\to \infty} \mathlarger{\logE}_{n\leq X}
f\left(a_0Q^2n+a_0r_Q+b_0\right) \overline{f\left(vQ^2n+vr_Q+d_0\right)}$$
exists, which implies that in \eqref{eqn_rec_fg_conclusion _3} we can
replace the $\limsup_{X \to \infty}$ by $\lim_{X\to \infty}$. Using that and the 
Dominated Convergence Theorem we get 
\begin{equation}\label{eq dom conv 1}
    \lim_{K\to \infty} \lim_{X\to \infty} 
    \int_{\Mfgp\setminus \{1\}} \left| \mathlarger{\cesE}_{Q\in \Phi_K} 
     \mathlarger{\logE}_{n\leq X}
    f(a_0Q^2n+a_0r_Q+b_0)\overline{f(vQ^2n+vr_Q+d_0)} \right| d\sigma(f)=0.
\end{equation}
Now for each $K$ and $X$, we have  
\begin{multline*}
    \left| \mathlarger{\cesE}_{Q\in \Phi_K} \int_{\Mfgp\setminus \{1\}}
    \Re \left( f(v) \overline{f(a_0)}\: \mathlarger{\logE}_{n\leq X}
    f(a_0Q^2n+a_0r_Q+b_0)\overline{f(vQ^2n+vr_Q+d_0)} \right)
    d\sigma(f) 
    \right| \\
    \leq \int_{\Mfgp\setminus \{1\}} \left| \mathlarger{\cesE}_{Q\in \Phi_K} 
     \mathlarger{\logE}_{n\leq X}
    f(a_0Q^2n+a_0r_Q+b_0)\overline{f(vQ^2n+vr_Q+d_0)} \right| d\sigma(f), 
\end{multline*}
so taking $X\to \infty$ first, and then $K\to \infty$, we obtain that 
\begin{multline}\label{eq dom conv 2}
    \lim_{K \to \infty} \limsup_{X\to \infty} 
    \Bigg| \mathlarger{\cesE}_{Q\in \Phi_K} \int_{\Mfgp\setminus \{1\}}
    \Re \Big( f(v) \overline{f(a_0)}\\
    \mathlarger{\logE}_{n\leq X}
    f(a_0Q^2n+a_0r_Q+b_0)\overline{f(vQ^2n+vr_Q+d_0)} \Big)
    d\sigma(f) 
    \Bigg|=0.
\end{multline}
Hence there is $K$ so that 
\begin{multline*}
    \limsup_{X\to\infty} 
    \Bigg| \mathlarger{\cesE}_{Q\in \Phi_K} \int_{\Mfgp \setminus \{1\}}
    \Re \Big( f(v) \overline{f(a_0)} \\
    \mathlarger{\logE}_{n\leq X}
    f(a_0Q^2n+a_0r_Q+b_0)\overline{f(vQ^2n+vr_Q+d_0)} \Big)
    d\sigma(f) \Bigg| 
    \leq \frac{\delta^2}{4},
\end{multline*}
from which we get that there is $Q\in \Phi_K$ so that 
\begin{equation}\label{eq ld}
\limsup_{X\to \infty}  \int_{\Mfgp \setminus \{1\}} 
\Re\left(f(v) \overline{f(a_0)} \: \mathlarger{\logE}_{n\leq X} 
f(a_0Q^2n+a_0r_Q+b_0)\overline{f(vQ^2n+vr_Q+d_0)} \right) d\sigma(f)\geq 
-\frac{\delta^2}{4}.
\end{equation}
Since for each $f\in \Mfgp$ the limit of the correlations appearing in 
\eqref{eq ld} exists, using the Dominated Convergence Theorem 
in \eqref{eq ld}, we can take $\lim_{X\to \infty}$.

Since $b\neq d$, we have that $a_0 Q^2 (vr_Q + d_0) \neq vQ^2(a_0 r_Q +b_0),$
so we can apply Theorem \ref{T: Tao theorem} to find $B>0$ depending on 
$a_0, b_0, v, d_0, Q, r_Q$ so   
that if the $1$-bounded multiplicative function $h$ satisfies 
\begin{equation*}
     \left|\mathlarger{\logE}_{n\leq X} 
     h\left(a_0Q^2n+a_0r_Q+b_0\right)\overline{h\left(vQ^2n+vr_Q+d_0\right)}
     \right|\geq \frac{\delta^2}{4\sigma(\M)},
\end{equation*}
and $X$ is sufficiently large, then 
$   \inf_{|t|\leq BX,q\leq B}\D(h,\chi(n)n^{it};1,X)\leq B$.
Recall that if $f\in \Mfgs$, then $f$ is strongly aperiodic, so 
 $\inf_{|t|\leq BX,q\leq B}\D(h,\chi(n)n^{it};1,X)\to + \infty$ as $X\to \infty$,
 and therefore, 
 $$\limsup_{X\to \infty} 
 \left|\mathlarger{\logE}_{n\leq X} 
     f\left(a_0Q^2n+a_0r_Q+b_0\right)\overline{f\left(vQ^2n+vr_Q+d_0\right)}
     \right| \leq  \frac{\delta^2}{4\sigma(\M)},$$
which in particular implies that 
\begin{equation}\label{eq dom conv 3}
    \liminf_{X\to \infty} \Re\left(f(v) \overline{f(a_0)}
 \mathlarger{\logE}_{n\leq X} 
     f\left(a_0Q^2n+a_0r_Q+b_0\right)\overline{f\left(vQ^2n+vr_Q+d_0\right)}\right)
     \geq  - \frac{\delta^2}{4\sigma(\M)}.
\end{equation}
Combining Fatou's lemma with \eqref{eq dom conv 3} we get that 
\begin{equation}\label{eq dom conv 4}
    \liminf_{X\to \infty} \int_{\Mfgs} 
    \Re\left(f(v) \overline{f(a_0)}
 \mathlarger{\logE}_{n\leq X} 
     f\left(a_0Q^2n+a_0r_Q+b_0\right)\overline{f\left(vQ^2n+vr_Q+d_0\right)}\right)
     d\sigma(f) \geq - \frac{\delta^2}{4}.
\end{equation}
Then, for the $Q$ we chose above, using \eqref{eqn_rec_fg_conclusion _2},
\eqref{eq ld} with $\lim_{X\to \infty}$ instead of $\limsup_{X\to \infty}$
(which we can do as commented above), and  
\eqref{eq dom conv 4} we get that 
\begin{equation}\label{eq ld 2}
\liminf_{X \to \infty}\: \mathlarger{\logE}_{n\leq X} 
\int_{\Mfg} 
f(v) \overline{f(a_0)}
f(a_0Q^2 n+ a_0 r_Q+b_0)\overline{f(vQ^2 n+vr_Q +d_0)}~d\sigma(f) 
\geq \frac{\delta^2}{2}.
\end{equation}
Consider the set 
$$\A=\left\{ n\in \N \colon 
\int_{\Mfg} f(an+b)\overline{f(an+d)}~d\sigma(f) \geq 
\frac{\delta^2}{2}\right\}. $$ 
Then from \eqref{eq ld 2} and \eqref{E: relation between averages along progressions and normal averages} we have that 
\begin{equation*}
\liminf_{X \to \infty} \mathlarger{\logE}_{n\leq Q^2 X+r_Q}\1_{\A}(n) > 0, 
\end{equation*}
so as in the last step of the proof of Proposition 
\ref{P: Main proposition for f-f} we get that 
$$\liminf_{Y\to \infty} \mathlarger{\logE}_{n\leq Y}
\1_{\A} (n)>0,$$
so $\A$ has indeed positive lower 
logarithmic density. This concludes the proof.
\end{proof}

\begin{remark}\label{new_remark}
Utilizing an almost identical proof, one can show that for each 
$\ell \in \N$, the set 
$\{\ell n/(n+1) \colon n\in \N \}$
is a set of recurrence for finitely generated multiplicative systems. 
For $f\in \Mfgp\setminus \{1\}$, there is no 
$t\in\R$ such that $f$ is identically equal to $n^{it}$,
so using Proposition \ref{P: binary correlations proposition} 
with $a_1=a_2=b_2=1$ and $b_1=0$, we 
can find $r_Q$'s so that 
\begin{equation}\label{remark_eq_1}
    \lim\limits_{K\to \infty}\  \limsup\limits_{X\to \infty}\     \left|\mathlarger{\cesE}_{Q\in \Phi_K} \mathlarger{\logE}_{n\leq X} f(Q^2n+ r_Q)\overline{f(Q^2n+r_Q+1)}   \right|=0.
\end{equation}
On the other hand, applying Theorem \ref{T: Tao theorem} we get that for 
each $f\in \Mfgs$,
\begin{equation}\label{remark_eq_2}
    \liminf_{X\to \infty} \Re\left(f(\ell)
 \mathlarger{\logE}_{n\leq X} 
f\left(Q^2n+r_Q\right)\overline{f\left(Q^2n+r_Q+1\right)}\right)
     \geq  - \frac{\delta^2}{4\sigma(\M)}.
\end{equation}
Following the exact same proof as the one used for Theorem 
\ref{thm_rec_fg_reformulation}, and replacing 
\eqref{eqn_rec_fg_conclusion _3} by \eqref{remark_eq_1},
and \eqref{eq dom conv 3} by \eqref{remark_eq_2}, we obtain that
if $\delta>0$ and $\sigma$ is a finite Borel measure on $\Mfg$ such that 
$\sigma(\{1\})\geq \delta^2$ and 
$\int_{\Mfg} f(r)\overline{f(s)}~d\sigma(f) \geq 0$ for every $r,s\in\N$,
then 
\begin{equation}\label{remark_eq_4}
        \int_{\Mfg} f(\ell n)\overline{f(n+1)}~d\sigma(f) \geq 
        \frac{\delta^2}{2}
    \end{equation}
for a set of $n\in\N$ with positive lower density. As explained in the 
beginning of Section \ref{section_mult-rec-for-fg-systems}, this 
implies that the set 
$\{\ell n/(n+1) \colon n\in \N \}$
is a set of recurrence for finitely generated multiplicative systems. 
\end{remark}


\begin{appendix}\label{appendix_a}
    
\section{Strong aperiodicity for a class of multiplicative functions}\label{appendix}

In this section, we establish that any finitely generated multiplicative function is either pretentious or strongly aperiodic. This will be a simple corollary of the following proposition.

\begin{proposition}\label{P: strong aperiodicity for almost non-dense functions}
    Let $B>0$, $f\in \M$ and $\chi$ be a fixed Dirichlet character of conductor $q$. 
    Assume that there exists an open subset $U$ of the unit circle 
    $\mathbb{S}^1$, such that
    \begin{equation}\label{E: f(p) is essentially not in U}
      \sum_{\substack{p\leq X \\ f(p)\in U}}\frac{1}{p}=o_{X\to \infty}(\log\log X).
    \end{equation}
    Then, we have \begin{equation*}
  \inf_{t\leq BX} \D(f,\chi(n)n^{it};1,X)\gg_{q,B,U} 1+\D(f,\chi;1,X).
    \end{equation*}
\end{proposition}
This implies that if the multiplicative function $f$ satisfies \eqref{E: f(p) is essentially not in U}, then for the infimum over $t$ to be small, we must have that the distance of $f$ from $\chi$ is small.

We recall several standard facts about Dirichlet $L$-functions. We adopt the usual conventions and write $s=\sigma+it$ for a complex number.
Let $\chi$ be a Dirichlet character and throughout this section we assume all characters are primitive. The Dirichlet $L$-function $L(s,\chi):=\sum_{n\geq1}\chi(n)n^{-s}$, $\Re(s)>1$
has Euler product factorization\begin{equation*}
    L(s,\chi)=\prod_{p\in \P} \left(1-\frac{\chi(p)}{p^s}\right)^{-1}
\end{equation*}
valid on the half-plane $\Re(s)>1$. It possesses a meromorphic continuation to the entire complex plane and if $\chi$ is principal, then it has only one pole of order $1$ at $s=1$, otherwise, it is entire.

By the Vinogradov-Korobov zero free region for Dirichlet $L$-functions (see \cite[p. 176]{montgomery_ten_lectures} without proof, or \cite[Lemma 1]{languasco2007note}), letting $q$ be the modulus of $\chi$, we have that $L(\sigma+it,\chi)\neq 0$ in the region
\begin{equation}\label{zero-free_region_Dirichlet}
    \left\{\sigma+it\colon \sigma> 1 - \frac{C}{\log{q} + (\log|t|)^{2/3}(\log\log|t|)^\frac{1}{3}}\right\}
\end{equation}
for all $|t|\geq 3$ and for some absolute constant $C>0$.
Using standard methods (see for instance \cite[pp. 56-60]{titchmarsh_book} for Landau's method) this zero-free region yields the bound
\begin{equation}\label{E: Vinogradov-Korobov bound_Dirichlet}
    \frac{L'(\sigma+it,\chi)}{L(\sigma+it,\chi)}
    \ll \log{q} + (\log|t|)^{2/3}(\log\log|t|)^{1/3}.
\end{equation}
in the same region.
Since $L(s,\chi)$ does not vanish on this set, we can define a complex logarithm that is holomorphic on the zero-free region. When $\sigma>1$, we have that the identity \begin{equation*}
    \log L(s,\chi)=-\sum_{p\in \P}\log \left(1-\chi(p)p^{-s}\right),
\end{equation*}holds, since this is valid for all real $s$ with $s>1$.

Now, if $\chi$ is principal, then we can use Taylor expansion to get that for $|t|\leq 1/100$ (say) and $\sigma$ close to $1$, \begin{equation}\label{E: asymptotic for log derivative near 1}
    -\frac{L'}{L}(s,\chi)=\frac{C_{1,\chi}}{s-1}+C_{2,\chi}\gamma+\Oh(|s-1|)
\end{equation}
where $C_{1,\chi},C_{2,\chi}$ are some constants depending on $\chi$.

On the other hand, if $\chi$ is non-principal, then the logarithmic derivative $L'/L(s,\chi)$ does not have poles in this region, which follows from the fact that this is a zero-free region for $L(s,\chi)$ and $L(s,\chi)$ is pole-free. Then, by continuity, for $|t|\ll 1$ and $\sigma$ close to $1$, we have
\begin{equation}\label{bound for L'/L close to 1}
    \frac{L'}{L}(s,\chi) \ll 1.
\end{equation}
Finally, for any non-principal Dirichlet character $\chi$ modulo $q$, we have the classic bound (see \cite[Theorem 8.18]{tenenbaum_2015}) asserting that for $\sigma\geq1$,
\begin{equation}\label{bound_for_L}
    L(s,\chi) \ll \log(q|t|).
\end{equation}

The main inputs in our method are the following bounds for the sums $\sum_{Y\leq p\leq X}\chi(p)p^{-1-ia}$ where $|a|$ can be as large as $X$.

\begin{lemma}\label{L: asymptotics for sum over p to the 1+ia with characters}
    Let $X\geq 10$ and set $Y=\exp\left((\log{X})^C\right)$ for some constant $ 2/3<C<1$. Let $\chi$ be a Dirichlet character of conductor $q\geq1$.
    \begin{itemize}
        \item[a)] Uniformly for all $(\log X)^{-1}\leq |a|\leq X$, we have \begin{equation*}
            \sum_{Y\leq p\leq X} \frac{\chi(p)}{p^{1+ia}}\ll_{q}  1.
            \end{equation*}
        \item[b)] Uniformly for all $|a|\leq (\log X)^{-1}$, we have \begin{equation*}
        \begin{cases}
            \displaystyle\sum_{p\leq X} \frac{1-\Re(\chi(p)p^{-ia})}{p} \ll_{q} 1, & \text{if}~\chi~\text{is principal}, \\
            \displaystyle\sum_{p\leq X} \frac{\Re(\chi(p)p^{-ia})}{p} \ll_q 1, & \text{otherwise}.
        \end{cases}
        \end{equation*}
    \end{itemize}
\end{lemma}

\begin{proof}
All asymptotic constants are allowed to depend on $\chi$, or equivalently to $q$ (since there are finitely many Dirichlet characters of conductor $q$). First of all, we observe that \begin{equation*}
        \sum_{p\leq X}\frac{\chi(p)}{p^{1+ia}}=\sum_{p\leq X} \frac{\chi(p)}{p^{1+\frac{1}{\log X}+ia}} + \Oh(1)
    \end{equation*}Thus, for the given $X$ and $Y$, we have \begin{equation*}
        \sum_{Y\leq p\leq X}\frac{\chi(p)}{p^{1+ia}}=\sum_{Y\leq p\leq X}\frac{\chi(p)}{p^{1+\frac{1}{\log X}+ia}}-\sum_{p\leq Y}\left(\frac{\chi(p)}{p^{1+\frac{1}{\log X}+ia}}-\frac{\chi(p)}{p^{1+\frac{1}{\log Y}+ia}}\right)+\Oh(1).
    \end{equation*}
    Using the mean value theorem, we easily deduce that each term is the second sum is $\Oh\left(\log p(\log Y)^{-3}\right)$ and thus the whole sum is $\Oh\left((\log Y)^{-2}\right)$ by Mertens' theorem. Taking logarithms in the Euler product of $L(s,\chi)$ (using the branch defined above) and then using Taylor expansion, we deduce that 
    \begin{align*}
        \log\left(L\left(1+\frac{1}{\log X}+ia,\chi\right)\right)
        = -\sum_{p\in \P}\log \left(1-\frac{\chi(p)}{p^{1+\frac{1}{\log X} +ia}}\right) 
        & = \sum_{p\in \P} \left(\frac{\chi(p)}{p^{1+\frac{1}{\log X} +ia}}+\Oh\left(\frac{1}{p^2}\right)\right) \\
        & = \sum_{p\in \P} \frac{\chi(p)}{p^{1+\frac{1}{\log X} +ia}} +\Oh(1).
    \end{align*}
    Finally, we have that 
    \begin{equation*}
        \sum_{p\geq X} \frac{\chi(p)}{p^{1+\frac{1}{\log X} +ia}} \ll 1
    \end{equation*}
    which can be obtained by splitting into dyadic intervals and summing the inequalities $$\left|\sum_{2^k\leq p<2^{k+1}}\frac{\chi(p)}{p^{1+\frac{1}{\log X} +ia}}\right|\leq \frac{1}{2^k}\sum_{2^k\leq p< 2^{k+1}}\frac{1}{X^{\frac{1}{\log X}}}\ll \frac{1}{2^k}\log\left(1+\frac{1}{k}\right), $$
    by Mertens' theorem.  
    In conclusion, we arrive at the formulas 
    \begin{equation}\label{sum_p_as_log_of_L}
         \log\left(L\left(1+\frac{1}{\log X}+ia,\chi\right)\right)=\sum_{p\leq X}\frac{\chi(p)}{p^{1+ia}}+\Oh(1)
    \end{equation}
    and 
    \begin{equation*}
         \log\left(L\left(1+\frac{1}{\log X}+ia,\chi\right)\right)- \log\left(L\left(1+\frac{1}{\log Y}+ia,\chi\right)\right)=\sum_{Y\leq p\leq X}\frac{\chi(p)}{p^{1 +ia}}+\Oh(1).
    \end{equation*}
    Therefore, we can rewrite our sum as \begin{equation}\label{sum_p_as_integral}
        \left|\sum_{Y\leq p\leq X}\frac{\chi(p)}{p^{1 +ia}}\right|
        = \left|\int_{Y}^X \frac{L'}{L}\left(1+\frac{1}{\log u}+ia,\chi\right)\frac{du}{u(\log u)^2} \right|+\Oh(1).
    \end{equation}

a) We prove the first assertion of the lemma.    Firstly, assume that $3\leq |a|\leq X$. Then, we use \eqref{sum_p_as_integral} and \eqref{E: Vinogradov-Korobov bound_Dirichlet} to deduce that \begin{align*}
    \left|\sum_{Y\leq p\leq X}\frac{\chi(p)}{p^{1+ia}}\right| 
    & \ll \int_{Y}^X \left(\log{q}+(\log|a|)^{2/3}(\log\log|a|)^{1/3}\right)\frac{du}{u(\log u)^2} + 1 \\
    & \ll (\log X)^{2/3}(\log\log X)^{1/3}\left(\frac{1}{\log Y}-\frac{1}{\log X}\right).
    \end{align*}
    By our assumptions, we have $\log Y=(\log X)^C$, which implies the desired bound.

   Now, suppose that $(\log X)^{-1}\leq |a|\leq 3$.
    If $\chi$ is non-principal, we use \eqref{sum_p_as_integral} and \eqref{bound for L'/L close to 1} to infer that
    $$\left|\sum_{Y\leq p\leq X}\frac{\chi(p)}{p^{1+ia}}\right| \ll \int_Y^X \frac{du}{u(\log u)^2} + 1
    = \frac{1}{\log{Y}} - \frac{1}{\log{X}} + 1
    \ll 1.$$

    If $\chi$ is principal,     
    observe that if $1/100\leq |a|\leq 3$, then we have $\left|\frac{L'}{L}\left(1+\frac{1}{\log u}+ia,\chi\right)\right|\asymp 1$. Thus, the result follows as before.
We now work in the range $(\log Y)^{-1}\leq |a|\leq \frac{1}{100} $. 
    We use \eqref{E: asymptotic for log derivative near 1} to write $$-\frac{L'}{L}(s,\chi)= \frac{C_{\chi}}{s-1}+\Oh(1).$$
    for some constant $C_{\chi}$ depending only on $\chi$. Thus,
    \begin{equation*}
    \sum_{Y\leq p\leq X}\frac{\chi(p)}{p^{1 +ia}}= C_{\chi}\int_{Y}^X \frac{du}{(\frac{1}{\log u}+ia)u(\log u)^2 } +\Oh\left(\int_{Y}^X \frac{du}{u\log^2 u} \right)+\Oh(1)
    \end{equation*}
    The integral in the second term is $\Oh((\log Y)^{-1})$ and therefore we can rewrite the last equality as \begin{equation*}
    \sum_{Y\leq p\leq X}\frac{\chi(p)}{p^{1 +ia}}= \int_{Y}^X \frac{du}{({1}+ia\log u )u\log u } +\Oh(1)
    \end{equation*}
    We rewrite the previous integral as \begin{equation*}
    \int_{Y}^X \frac{1-ia\log u}{\left(1+a^2(\log u)^2\right)u\log u}du=\int_{Y}^X \frac{du}{\left(1+a^2(\log u)^2)\right)u\log u}-ia\int_{Y}^X \frac{du}{\left(1+a^2(\log u)^2 \right)u}.  
    \end{equation*}
    It suffices to show that each integral is $\Oh(1)$.
    Using the inequality $|a|\log Y\geq 1$, we get that 
    \begin{equation*}
    \left|\int_{Y}^X \frac{du}{\left(1+a^2(\log u)^2)\right)u\log u}\right|\leq \int \frac{du}{2u\log u}=\frac{1}{2\log Y}-\frac{1}{2\log X} \ll 1.
    \end{equation*}
    For the second integral, we use the inequality $1+a^2(\log u)^2\geq 2a\log u$ to deduce that 
    \begin{equation*}
        \int_{Y}^X \frac{a}{\left(1+a^2(\log u)^2 \right)u}\, du
        \leq \int_{Y}^{X}\frac{du}{2au(\log u)^2}=\frac{1}{2a}\left(\frac{1}{\log X}-\frac{1}{\log Y}\right) \ll 1,
    \end{equation*}
    since $|a|\log X\geq |a|\log Y\geq 1$.

    Finally, assume that $(\log X)^{-1}\leq |a|\leq (\log Y)^{-1} $. We repeat the same argument as above and arrive at the equality\begin{equation*}
    \sum_{Y\leq p\leq X}\frac{1}{p^{1 +ia}}=\int_{Y}^X \frac{du}{\left(1+a^2(\log u)^2)\right)u\log u}-ia\int_{Y}^X \frac{du}{\left(1+a^2(\log u)^2 \right)u}+\Oh(1).  
    \end{equation*}
    Here, we calculate each integral.
    For the first integral, we have 
    \begin{align*}
    & \int_{Y}^X \frac{du}{\left(1+a^2(\log u)^2)\right)u\log u}
    = \int_{Y}^X \frac{(|a|\log u)'}{(1+a^2(\log u)^2)|a|\log u}du
    = \int_{|a|\log Y}^{|a|\log X} \frac{dt}{t(1+t^2)} \\
    & = \int_{|a|\log Y}^{|a|\log X}\left(\log t-\frac{1}{2}\log(t^2+1)\right)'dt \\
    & = \log(|a|\log X)-\log(|a|\log Y)-
    \frac{1}{2}\left(\log(1+a^2\log^2X)-\log(1+a^2\log^2 Y)\right) \\
    & = \log(|a|\log X)-\log \sqrt{1+a^2(\log X)^2}+\Oh(1),
    \end{align*}
    where we used that $|a|\log Y$ is bounded in the last line. We conclude that 
    \begin{equation*}
    \int_{Y}^X \frac{du}{\left(1+a^2(\log u)^2)\right)u\log u}=-\frac{1}{2}\log\left(1+\frac{1}{|a|\log X}\right)+\Oh(1).
    \end{equation*}
    Since $|a|\log X\geq 1$, we deduce that the last quantity is $\Oh(1)$.
    For the second integral, we have \begin{equation*}
    a\int_{Y}^X \frac{du}{\left(1+a^2(\log u)^2 \right)u}=\int \frac{(a\log u)'}{(1+a^2(\log u)^2) }du=\int_{a\log Y}^{a\log X} \frac{dt}{1+t^2} \leq \int_{-\infty}^{+\infty} \frac{dt}{1+t^2}
    \ll 1.
    \end{equation*} 
    We conclude that the second integral is bounded.

    Combining everything together, we conclude that \begin{equation*}
        \sum_{Y\leq p\leq X}\frac{\chi(p)}{p^{1 +ia}} \ll 1
    \end{equation*}
    for principal $\chi$. This concludes the proof of a).
    
   b) Now we prove the second part, where we always assume that $|a|\leq (\log X)^{-1}$. 
    
    Suppose first that $\chi$ is principal. Note that
    $$\sum_{p\leq X}\frac{\chi(p)}{p^{1+ia}} = \sum_{p\leq X}\frac{1}{p^{1+ia}} + \Oh\left(\sum_{p\mid q}\frac{1}{p}\right)
    = \sum_{p\leq X}\frac{1}{p^{1+ia}} + \Oh(\log\log(\omega(q))).$$
    Hence it suffices to show that \begin{equation*}
    \sum_{p\leq X}\frac{1-\cos(a\log p)}{p}\ll 1.
    \end{equation*}
    By Taylor expansion, we have that \begin{equation*}
    \cos(a\log p)=1+\Oh(|a\log p|^2)
    \end{equation*}Thus, the sum in question is bounded by $\Oh(1)$ times \begin{equation*}
    |a^2|\sum_{p\leq X} \frac{(\log p)^2}{p} 
    \end{equation*}
    By Mertens' theorem and partial summation, we have that \begin{equation*}
    \sum_{p\leq X} \frac{(\log p)^2}{p} \ll (\log X)^2.
    \end{equation*}
    Thus, the original sum is $\Oh(|a|^2(\log X)^2)=\Oh(1)$, as desired.

    Finally, suppose that $\chi$ is non-principal. Then by \eqref{sum_p_as_log_of_L} and \eqref{bound for L'/L close to 1} we have
    $$\sum_{p\leq X}\frac{\Re(\chi(p)p^{-ia})}{p} = \log\left|L\left(1+\frac{1}{\log{X}}+ia,\chi\right)\right| + \Oh(1)
    \ll \log\log(q|t|) + 1 \ll 1.$$
    This concludes the proof of b), and hence the proof of the lemma.
\end{proof}

\begin{proof}[Proof of Proposition \ref{P: strong aperiodicity for almost non-dense functions}]

We prove the proposition in the case $B=1$. The general case follows by noting that the contribution of primes from $X$ to $BX$ in the distance function is $O(\log B)$, which is acceptable.

Let $f$ be a function satisfying the assumptions of our main theorem. Since intervals generate the topology on $\mathbb{S}^{1}$, we easily see that the statement follows if we prove it in the case where $U$ is a closed interval.

We want to bound 
\begin{equation*}
\inf_{|t|\leq X}   \sum_{p\leq X} \frac{1-\Re(f(p)\overline{\chi(p)p^{it}})}{p} 
\end{equation*}from below.
We will consider two cases depending on the size of $t$ and we will invoke Lemma \ref{L: asymptotics for sum over p to the 1+ia with characters}.

First of all, assume that $|t|\leq (\log X)^{-1}$.
Using the triangle inequality, we have that \begin{equation*}
    \D(f(n),\chi(n)n^{it};1,X)\geq \D(f,\chi;1,X)-\D(1,n^{it};1,X)
\end{equation*}Using part b) of Lemma \ref{L: asymptotics for sum over p to the 1+ia with characters} for the trivial character, we conclude that $$\sup_{|t|\leq (\log X)^{-1}}\D(n^{it},1;1,X)\ll 1.$$ We deduce that \begin{equation*}
    \inf_{|t|\leq (\log X)^{-1}}  \D(f(n),\chi(n)n^{it};1,X)\geq  \D(f,\chi;1,X)+\Oh(1),
\end{equation*}so our claim follows.

Now, we assume that $(\log X)^{-1}\leq |t|\leq X$. We let $Y=\exp((\log X)^C)$ for some $2/3<C<1$.
Our assumption implies that there exists a closed interval $U$ such that 
\begin{equation*}
     \sum_{\substack{Y\leq p\leq X \\ f(p)\in U}}\frac{1}{p} = \oh_{X\to\infty}(\log\log{X}).
\end{equation*}
We let $V_{\e}$ be a sub-interval of $U^c$ obtained by removing two intervals of length $\e$ from its endpoints (where $\e$ is sufficiently small depending on $\m(U^c)$).

We consider a smooth function $F$ on the unit circle, such that $F(x)=0$ for all $x\notin V_{\e}$, $F(x)\leq \1_{V_\e}(x)$ for all $x\in \mathbb{S}^1$ and such that $\int F(t)\, dt> \m(V_\e)/2$. Then, we observe that 
\begin{align*}
    & \sum_{p\leq X}\frac{1-\Re(f(p)\overline{\chi(p)}p^{-it})}{p}
    = \frac{1}{2}\sum_{p\leq X} \frac{|f(p) - \chi(p)p^{it}|^2}{p} + \Oh(1) \\
    & \geq \frac{1}{2}\sum_{Y\leq p\leq X}\frac{|f(p)-\chi(p)p^{it}|^2}{p} + \Oh(1) \geq \frac{1}{2}\sum_{\substack{Y\leq p\leq X\\
    \chi(p)p^{it}\in V_\e}}\frac{|f(p)-\chi(p)p^{it}|^2}{p} + \Oh(1) \\
    & = \frac{1}{2} \sum_{{Y\leq p\leq X}}\frac{|f(p)-\chi(p)p^{it}|^2 \1_{V_\e}(\chi(p)p^{it})}{p}+\Oh(1) 
    \geq \frac{\e^2}{2}\sum_{{Y\leq p\leq X}}\frac{F\left(\chi(p)p^{it}\right)}{p}+\oh_{X\to\infty}(\log\log X).
\end{align*}
Here, we used the fact that $f(p)\in U$ for all primes outside a set of at most $o_{X\to \infty}(\log\log X)$ primes, and thus $|f(p)-\chi(p)p^{it}|>\e$ for all those primes.

Now, it suffices to show that \begin{equation}\label{E: sum that we want to be large}
    \sum_{{Y\leq p\leq X}}\frac{F(\chi(p)p^{it})}{p}\gg \log\log X
\end{equation}holds for all $|t|\geq (\log X)^{-1}$.
Then, the result would follow by combining the bounds from the two cases and noting that $\D(f,\chi;1,X)\ll \sqrt{
\log\log X}$.

Since $F$ is smooth, we know that \begin{equation}\label{E: approximation of smooth F by trigs}
    \sup_{z\in \mathbb{S}^1}\Big|F(z)-\sum_{-R\leq \ell\leq R}\widehat{F}(\ell)z^{\ell}\Big|=o_{R\to \infty}(1).
\end{equation}We pick $R$ to be some large, but finite integer and we let $X,Y$ be very large in terms of $R$.
Using this approximation and Mertens' theorem, we rewrite the sum in \eqref{E: sum that we want to be large} as\begin{equation}\label{E: sum after approximation}
   (1-C) \log\log X\int F(t)\, dt+\sum_{0<|\ell|\leq R}\widehat{F}(\ell)\left(\sum_{Y\leq p\leq X} \frac{\chi^{\ell}(p)p^{i\ell t}}{p}\right)+o_{R\to \infty}(1)\log\log X+\Oh(1).
\end{equation}
Recall that $|t|\geq (\log X)^{-1}$ and $\chi^{\ell}$ has conductor bounded by $q$. Then, Lemma \ref{L: asymptotics for sum over p to the 1+ia with characters} implies that \begin{equation*}
   \left| \sum_{Y\leq p\leq X} \frac{\chi(p)p^{i\ell t}}{p}\right|\ll 1
\end{equation*}for all $0<|\ell|\leq R$. Since $F$ is bounded, we conclude that the sum in \eqref{E: sum after approximation} is \begin{equation*}
    (1-C+o_R(1))\log\log X\int F(t)\, dt+\Oh(1).
\end{equation*} In particular, we conclude that \begin{multline*}
     \sum_{Y\leq p\leq X}\frac{1-\Re(f(p)p^{it})}{p}\gg \frac{\e^2}{2}(1-C+o_{R\to \infty}(1))\log\log X\int F(t)dt+o_{X\to \infty}(\log\log X)\\
     +\Oh(1),
\end{multline*}which gives the desired conclusion.
\end{proof}

We can now prove Proposition \ref{P: finitely generated non-pretentious is strongly aperiodic}.
\begin{proof}
    Assume that $f$ is a non-pretentious finitely generated function. We want to show that \begin{equation*}
      \lim_{X\to \infty}  \inf_{|t|\leq BX, q\leq B} \D(f(n),\chi(n)n^{it};1,X)=+\infty
    \end{equation*}for every $B>0$. Fixing a Dirichlet character $\chi$ of modulus at most $B$, we want to show that \begin{equation}\label{E: finitely generated times Dirichlet is strongly aperiodic}
         \lim_{X\to \infty}  \inf_{{|t|\leq BX}} \D(f(n)\overline{\chi(n)},n^{it};1,X)=+\infty
    \end{equation}
    By modifying our functions on the finitely many primes for which $\chi(p)=0$, we may assume that our function is in $\M$, while only perturbing the associated distances by an $\Oh(1)$ term.
    However, the function $f\cdot \overline{\chi}$ is finitely generated, since $f$ and $\chi$ are. Therefore, it satisfies the assumptions of Proposition \ref{P: strong aperiodicity for almost non-dense functions} for an interval that does not contain any of the finitely many values of $f(p)\chi(p)$. Therefore, we have that \eqref{E: finitely generated times Dirichlet is strongly aperiodic} holds and the result follows.
\end{proof}
\end{appendix}

\bibliographystyle{abbrv}
\bibliography{refsmultiplicative}

\begin{thebibliography}{10}

\bibitem{Bergelson_Zm}
V.~Bergelson.
\newblock Sets of recurrence of {$\mathbb{Z}^m$}-actions and properties of sets of differences in {$\mathbb{ Z}^m$}.
\newblock {\em Journal of the London Mathematical Society. Second Series}, \textbf{31}(2):295--304, 1985.

\bibitem{bergelson_survey}
V.~Bergelson.
\newblock Ergodic ramsey theory—an update.
\newblock {\em Ergodic theory of $\mathbb{Z}^d$ actions (Warwick, 1993--1994)}, \textbf{228}:1--61, 1996.

\bibitem{Bergelson_multiplicative}
V.~Bergelson.
\newblock Multiplicatively large sets and ergodic {R}amsey theory.
\newblock {\em Israel Journal of Mathematics}, \textbf{148}:23--40, 2005.

\bibitem{Bergelson-Leibman-polynomial-VDW}
V.~Bergelson and A.~Leibman.
\newblock {Polynomial extensions of van der Waerden's and Szemer\'{e}di's theorems}.
\newblock {\em Journal of the American Mathematical Society}, \textbf{9}:725--753, 1996.

\bibitem{Bergelson-Lesinge_Zd}
V.~Bergelson and E.~Lesigne.
\newblock Van der {C}orput sets in {$\mathbb{Z}^d$}.
\newblock {\em Colloquium Mathematicum}, \textbf{110}(1):1--49, 2008.

\bibitem{Bergelson_McCutcheon_1}
V.~Bergelson and R.~McCutcheon.
\newblock Recurrence for semigroup actions and a non-commutative {S}chur theorem.
\newblock In {\em Topological dynamics and applications ({M}inneapolis, {MN}, 1995)}, volume \textbf{215} of {\em Contemporary Mathematics}, pages 205--222. American Mathematical Society, Providence, RI, 1998.

\bibitem{Bergelson-Richter}
V.~Bergelson and F.~K. Richter.
\newblock Dynamical generalizations of the prime number theorem and disjointness of additive and multiplicative semigroup actions.
\newblock {\em Duke Mathematical Journal}, \textbf{171}(15):3133--3200, 2022.

\bibitem{Charamaras-multiplicative}
D.~Charamaras.
\newblock Mean value theorems in multiplicative systems and joint ergodicity of additive and multiplicative actions.
\newblock {\em Transactions of the American Mathematical Society}, \textbf{378}(3):1883--1937, 2025.

\bibitem{Moreira&friends}
S.~Donoso, A.~N. Le, J.~Moreira, and W.~Sun.
\newblock Additive averages of multiplicative correlation sequences and applications.
\newblock {\em Journal d'Analyse Math\'{e}matique}, \textbf{149}(2):719--761, 2023.

\bibitem{Elliott_additive}
P.~D. T.~A. Elliott.
\newblock The value distribution of additive arithmetic functions on a line.
\newblock {\em Journal f\"{u}r die {Reine} und {Angenwandte Mathematik}}, \textbf{642}:57--108, 2010.

\bibitem{Fra-Host-structure-multiplicative}
N.~Frantzikinakis and B.~Host.
\newblock Higher order {F}ourier analysis of multiplicative functions and applications.
\newblock {\em Journal of the American Mathematical Society}, \textbf{30}(1):67--157, 2017.

\bibitem{fra-klu-mor-2}
N.~Frantzikinakis, O.~Klurman, and J.~Moreira.
\newblock Partition regularity of generalized {P}ythagorean pairs.
\newblock Preprint 2024, arXiv:2407.08360.

\bibitem{Fra-Klu-Mor}
N.~Frantzikinakis, O.~Klurman, and J.~Moreira.
\newblock Partition regularity of {P}ythagorean pairs.
\newblock {\em Forum of Mathematics Pi}, \textbf{13}:Paper No. e5, 52, 2025.

\bibitem{fra-lem-deL}
N.~Frantzikinakis, M.~Lema\'nczyk, and T.~de~la Rue.
\newblock Furstenberg systems of pretentious and {MRT} multiplicative functions.
\newblock {\em Ergodic Theory and Dynamical Systems}, \textbf{45}(9):2765--2844, 2025.

\bibitem{Furstenberg-original}
H.~Furstenberg.
\newblock {Ergodic behavior of diagonal measures and a theorem of Szemerédi on arithmetic progressions}.
\newblock {\em Journal d'Analyse Math\'{e}matique}, \textbf{31}(1):204--256, 1977.

\bibitem{Furstenberg-book}
H.~Furstenberg.
\newblock {\em Recurrence in Ergodic Theory and Combinatorial Number Theory}.
\newblock Princeton University Press, 1981.

\bibitem{FK_commuting}
H.~Furstenberg and Y.~Katznelson.
\newblock An ergodic {S}zemer\'edi theorem for commuting transformations.
\newblock {\em Journal d'Analyse Math\'ematique}, \textbf{34}:275--291, 1978.

\bibitem{Granville-Sound-book}
A.~Granville and K.~Soundararajan.
\newblock Multiplicative number theory: The pretentious approach. {B}ook manuscript in preparation.

\bibitem{Halasz}
G.~Hal\'{a}sz.
\newblock \"{U}ber die {M}ittelwerte multiplikativer zahlentheoretischer {F}unktionen.
\newblock {\em Acta Mathematica Hungarica}, \textbf{19}(3-4):356--403, 1968.

\bibitem{K-MF_vdC}
T.~Kamae and M.~Mend\`es~France.
\newblock Van der {C}orput's difference theorem.
\newblock {\em Israel Journal of Mathematics}, \textbf{31}(3-4):335--342, 1978.

\bibitem{correl_klu}
O.~Klurman.
\newblock Correlations of multiplicative functions and applications.
\newblock {\em Compositio Mathematica}, \textbf{153}(8):1622–1657, 2017.

\bibitem{Klu-Man}
O.~Klurman and A.~P. Mangerel.
\newblock Rigidity theorems for multiplicative functions.
\newblock {\em Mathematische Annalen}, \textbf{372}(1):651--697, 2018.

\bibitem{Klu-Man-Po-Ter}
O.~Klurman, A.~P. Mangerel, C.~Pohoata, and J.~Ter\"{a}v\"{a}inen.
\newblock Multiplicative functions that are close to their mean.
\newblock {\em Transactions of the American Mathematical Society}, \textbf{371}(11), 2019.

\bibitem{Klu-Man-Ter}
O.~Klurman, A.~P. Mangerel, and J.~Ter\"{a}v\"{a}inen.
\newblock On {Elliott's} conjecture and applications.
\newblock Preprint 2023, arXiv:2304.05344.

\bibitem{languasco2007note}
A.~Languasco and A.~Zaccagnini.
\newblock A note on {Mertens'} formula for arithmetic progressions.
\newblock {\em Journal of Number Theory}, \textbf{127}(1):37--46, 2007.

\bibitem{MRT_averaged}
K.~Matom\"aki, M.~Radziwi\l\l, and T.~Tao.
\newblock An averaged form of {C}howla's conjecture.
\newblock {\em Algebra Number Theory}, \textbf{9}(9):2167--2196, 2015.

\bibitem{montgomery_ten_lectures}
H.~L. Montgomery.
\newblock {\em Ten lectures on the interface between analytic number theory and harmonic analysis}, volume \textbf{84}.
\newblock American Mathematical Society, 1994.

\bibitem{sarkozy}
A.~S{\'{a}}rk\"{o}zy.
\newblock {On difference sets of sequences of integers. I}.
\newblock {\em Acta Mathematica Hungarica}, \textbf{31}(1-2):125--149, 1978.

\bibitem{Szemeredi_original}
E.~Szemer\'edi.
\newblock On sets of integers containing no {$k$} elements in arithmetic progression.
\newblock {\em Acta Arithmetica}, \textbf{27}:199--245, 1975.

\bibitem{Tao-Chowla}
T.~Tao.
\newblock The logarithmically averaged {C}howla and {E}lliott conjectures for two-point correlations.
\newblock {\em Forum of Mathematics Pi}, \textbf{4}:8--36, 2016.

\bibitem{tenenbaum_2015}
G.~Tenenbaum.
\newblock {\em Introduction to analytic and probabilistic number theory}, volume \textbf{163} of {\em Graduate Studies in Mathematics}.
\newblock American Mathematical Society, Providence, RI, third edition, 2015.

\bibitem{titchmarsh_book}
E.~Titchmarsh.
\newblock {\em The theory of the Riemann zeta-function}.
\newblock The Clarendon Press Oxford University Press, 1986.

\end{thebibliography}

\bigskip
\bigskip
\footnotesize
\noindent
Dimitrios Charamaras\\
\textsc{École Polytechnique Fédérale de Lausanne (EPFL)} \par\nopagebreak
\noindent
\href{mailto:dimitrios.charamaras@epfl.ch}
{\texttt{dimitrios.charamaras@epfl.ch}}

\bigskip
\footnotesize
\noindent
Andreas Mountakis\\
\textsc{University of Crete} \par\nopagebreak
\noindent
\href{mailto:a.mountakis@uoc.gr}
{\texttt{a.mountakis@uoc.gr}}

\bigskip
\footnotesize
\noindent
Konstantinos Tsinas\\
\textsc{École Polytechnique Fédérale de Lausanne (EPFL)} \par\nopagebreak
\noindent
\href{mailto:konstantinos.tsinas@epfl.ch}
{\texttt{konstantinos.tsinas@epfl.ch}}

\end{document}